\documentclass[12pt]{amsart}
\usepackage{amsmath,amsfonts,amssymb,amscd,mathtools,amsxtra,amsthm}
\usepackage[abbrev,alphabetic]{amsrefs}
\usepackage{mathrsfs}
\usepackage{hyperref}
\hypersetup{hidelinks}
\renewcommand{\eprint}[1]{\href{https://arxiv.org/abs/#1}{\nolinkurl{arXiv:#1}}}
\usepackage[margin=1.25in]{geometry}
\usepackage{needspace}
\usepackage{aliascnt}
\usepackage{cleveref}
\usepackage{tikz-cd}
\makeatletter
\def\@tocline#1#2#3#4#5#6#7{\relax
  \ifnum #1>\c@tocdepth 
  \else
    \par \addpenalty\@secpenalty\addvspace{#2}%
    \begingroup \hyphenpenalty\@M
    \@ifempty{#4}{%
      \@tempdima\csname r@tocindent\number#1\endcsname\relax
    }{%
      \@tempdima#4\relax
    }%
    \parindent\z@ \leftskip#3\relax \advance\leftskip\@tempdima\relax
    \rightskip\@pnumwidth plus4em \parfillskip-\@pnumwidth
    #5\small\setlength{\baselineskip}{10.5pt}\leavevmode\hskip-\@tempdima
      \ifcase #1
       \or\or \hskip 1em \or \hskip 2em \else \hskip 3em \fi%
      #6\nobreak\relax
    \hfill\hbox to\@pnumwidth{\@tocpagenum{#7}}\par
    \nobreak
    \endgroup
  \fi}
\makeatother

\newcommand{\Q}{\mathbb Q}
\newcommand{\R}{\mathbb R}
\newcommand{\C}{\mathbb C}
\newcommand{\Supp}{\operatorname{Supp}}

\newcommand{\coeff}{\operatorname{coeff}}
\newcommand{\ord}{\operatorname{ord}}
\newcommand{\Nm}{\operatorname{Nm}}
\newcommand{\lct}{\operatorname{lct}}
\newcommand{\vol}{\operatorname{vol}}
\newcommand{\Nklt}{\operatorname{Nklt}}
\newcommand{\Nlc}{\operatorname{Nlc}}
\newcommand{\Spec}{\operatorname{Spec}}
\newcommand{\Proj}{\operatorname{Proj}}

\theoremstyle{plain}
\newtheorem{theorem}{Theorem}[section]

\newaliascnt{proposition}{theorem}
\newtheorem{proposition}[proposition]{Proposition}
\aliascntresetthe{proposition}

\newaliascnt{lemma}{theorem}
\newtheorem{lemma}[lemma]{Lemma}
\aliascntresetthe{lemma}

\newaliascnt{corollary}{theorem}
\newtheorem{corollary}[corollary]{Corollary}
\aliascntresetthe{corollary}

\theoremstyle{definition}
\newaliascnt{notation}{theorem}
\newtheorem{notation}[notation]{Notation}
\aliascntresetthe{notation}

\newaliascnt{definition}{theorem}
\newtheorem{definition}[definition]{Definition}
\aliascntresetthe{definition}

\theoremstyle{remark}
\newaliascnt{remark}{theorem}
\newtheorem{remark}[remark]{Remark}
\aliascntresetthe{remark}
\newtheorem*{ackn}{Acknowledgements}

\theoremstyle{plain}
\newaliascnt{conjecture}{theorem}
\newtheorem{conjecture}[conjecture]{Conjecture}
\aliascntresetthe{conjecture}

\crefname{conjecture}{Conjecture}{Conjectures}
\Crefname{conjecture}{Conjecture}{Conjectures}

\numberwithin{equation}{section}
\crefname{theorem}{Theorem}{Theorems}
\crefname{proposition}{Proposition}{Propositions}
\crefname{lemma}{Lemma}{Lemmas}
\crefname{corollary}{Corollary}{Corollaries}
\crefname{definition}{Definition}{Definitions}
\crefname{notation}{Notation}{Notation}
\crefname{remark}{Remark}{Remarks}



\title[Construction of log CY boundaries from int-amplified endomorphisms]{Construction of log Calabi--Yau boundaries from int-amplified endomorphisms}
\author{Shou Yoshikawa}
\address{Institute of Science Tokyo, Tokyo 152-8551, Japan}
\email{yoshikawa.s.9fe9@m.isct.ac.jp}
\begin{document}
\begin{abstract}
In this paper, we prove that $\Q$-Gorenstein projective complex varieties admitting int-amplified endomorphisms are of Calabi--Yau type. 
\end{abstract}

\maketitle
\tableofcontents

\section{Introduction}

Let $X$ be a normal projective variety over $\C$.
A surjective endomorphism $h\colon X\to X$ is \emph{polarized} if $h^*H\sim qH$ for an ample Cartier divisor $H$ and an integer $q>1$.
Such an endomorphism is called $q$-polarized.
We say that $X$ is of \emph{Calabi--Yau type} if there is an effective $\Q$-Weil divisor $\Theta$ such that $(X,\Theta)$ is log canonical (lc) and $K_X+\Theta\sim_\Q0$.

Broustet and Gongyo proposed the following conjecture.
\begin{conjecture}[\cite{BroustetGongyo}*{Conjecture 1.2}]\label{conj:BG}
Every normal projective complex variety admitting a polarized endomorphism is of Calabi--Yau type.   
\end{conjecture}

The conjecture has been established for normal projective surfaces and Mori dream spaces by  \cite{BroustetGongyo}*{Theorem 1.3}, for smooth projective threefolds by  \cite{MengThreefoldCY}*{Theorem 1.2}, for Gorenstein terminal projective threefolds by \cite{CZ}*{Theorem 1.2}, and for smooth rationally connected projective varieties by \cite{YoshikawaFanoFibrations}*{Corollary 1.4}.
A related question asks whether log Calabi--Yau boundaries can be obtained directly from the ramification divisors of iterates.
\cite{LuoMeng26}*{Theorem 1.3} proved that, for every $q$-polarized endomorphism $h\colon\mathbf P^2\to\mathbf P^2$, there exists an integer $s\geq1$ such that $(\mathbf P^2,R_{h^s}/(q^s-1))$ is a log Calabi--Yau pair, where $R_{h^s}$ denotes the ramification divisor of $h^s$.
The proof has a methodological connection with \cite{LuoMeng26}: both control a fixed totally invariant locus, construct divisorial eigenvaluations, extract them equivariantly, and use adjunction and inversion of adjunction.

We work more generally with int-amplified endomorphisms.
A surjective endomorphism $h\colon X\to X$ is
\emph{int-amplified} if there exists an ample Cartier divisor
$H$ such that $h^*H-H$ is ample.
Every polarized endomorphism is int-amplified.

Our main result is the following.

\begin{theorem}\label{thm:main}
Let $X$ be a normal projective variety over $\C$, and let
$h\colon X\to X$ be an int-amplified endomorphism.
Let $B$ be an effective $\Q$-Weil divisor such that
$K_X+B$ is $\Q$-Cartier and
\[
 R_{h,B}:=R_h+B-h^*B\geq0.
\]
Then there exists an effective $\Q$-Weil divisor
$\Theta\geq B$ such that $(X,\Theta)$ is lc and
$K_X+\Theta\sim_\Q0$.
\end{theorem}

As a consequence, we obtain the following result, which
in particular gives an affirmative answer to
\cref{conj:BG} when $X$ is $\Q$-Gorenstein or of klt type.

\begin{corollary}\label{cor:all-dimensional-qg-cy}
Let $X$ be a normal projective variety over $\C$
admitting an int-amplified endomorphism.
If $X$ is $\Q$-Gorenstein or of klt type, then
$X$ is of Calabi--Yau type.
Here, we say that $X$ is of klt type if there exists an effective $\Q$-Weil divisor $\Delta$ such that $(X,\Delta)$ is klt.
\end{corollary}

We outline the proof of \cref{thm:main} in the case where
$h$ is $q$-polarized and
\[
 h^*(K_X+B)\equiv q(K_X+B).
\]
Put $P_h:=(\deg h)^{-1}h_*$.
By \cref{cor:polarized-one-term-correction}, after replacing
$h$ by an iterate and $q$ by the corresponding power,
the divisors
\[
 \Theta_N
 :=B+\sum_{j=1}^N P_h^jR_{h,B}
      +\frac{q^{1-N}}{q-1}P_hR_{h,B}
\]
satisfy
\[
 (X,\Theta_N)\text{ is lc},
 \qquad K_X+\Theta_N\sim_\Q0
\]
for all sufficiently large $N$.

In the proof, we define an analytic counterpart of the boundary construction.
\cref{thm:log-current-construction} constructs a closed positive $(1,1)$-current
\[
 T_{h,B}:=\sum_{j\geq1}[P_h^jR_{h,B}].
\]
Here $[D]$ denotes the current of integration along a $\Q$-Weil divisor $D$.
Then $(X,B;T_{h,B})$ is analytically lc, and $T_{h,B}$ represents the numerical class of $-(K_X+B)$.
Thus, we can regard it as an analytic log Calabi--Yau boundary, which is defined canonically by $(h,B)$.
We note that $[\Theta_N-B]$ converges weakly to $T_{h,B}$.
The non-klt locus of this current is algebraic and totally invariant; see
\cref{thm:log-current-construction,thm:invariant-nklt}.

Suppose first that
\[
 \Nklt(X,B;T_{h,B})=\Nklt(X,B).
\]
By \cref{lem:uniform-nlc-confinement}, for all sufficiently large $N$,
\[
 \Nlc(X,\Theta_N)
 \subseteq \Nklt(X,B;T_{h,B})
 =\Nklt(X,B).
\]
Since $(X,B)$ is lc, the right-hand side is the union of its lc centers.
After a further iterate, these centers are totally invariant, thus $h$ induces polarized endomorphisms on each lc center.
Then the inversion of adjunction and induction on dimension give log canonicity of $(X,\Theta_N)$.
Since $K_X+\Theta_N\equiv0$, the pair $(X,\Theta_N)$ is log Calabi--Yau.

Since being log Calabi--Yau descends, it suffices to show the existence of an equivariant projective birational model
\[
\pi\colon(W,B_W,g)\longrightarrow(X,B,h)
\]
with
\[
 \Nklt(W,B_W;T_{g,B_W})=\Nklt(W,B_W).
\]
To construct this model, consider an irreducible component $C$ of $\Nklt(X,B;T_{h,B})$ that is not contained in $\Nklt(X,B)$.
After replacing $h$ by an iterate, \cref{thm:generic-eigen-log-pair} constructs a prime divisor $E$ over $X$ with center $C$ satisfying
\[
 h_*\ord_E=e\ord_E
\]
for an integer $e>1$.
By \cref{prop:absorb-new-component}, the divisor can be extracted on an equivariant birational model with
\[
 \pi^{-1}(C)=\Supp E,
 \qquad \coeff_E B_W=1.
\]
Thus the entire inverse image of $C$ is contained in $\Nklt(W,B_W)$.
Repeating this construction gives the required equality of non-klt loci, as proved in \cref{thm:all-dimensional-lc-model}.

The construction of $E$ is the main local step.
After making $C$ totally invariant, put $R:=\mathcal O_{X,\eta_C}$, where $\eta_C$ is the generic point of $C$, and let $\mathfrak m$ be its maximal ideal.
By \cref{prop:invariant-average}, we construct a linearly bounded $\mathfrak m$-filtration $\mathfrak c_\bullet$ with $\lct_R(\mathfrak c_\bullet)=1$.
Its threshold-computing valuations, normalized by
\[
 A_{X,B}(v)=v(\mathfrak c_\bullet)=1,
\]
form a set $\mathscr V$ satisfying
\[
 h_*u\in\mathscr V
 \quad\Longrightarrow\quad
 qu\in\mathscr V.
\]
Using the quasi-monomiality theorem of \cite{Xu}*{Theorem~1.1} and the convexity results of \cite{BLQ}, we prove that $\mathscr V$ has a unique volume minimizer $v$, which is quasi-monomial; see \cref{prop:constrained-volume-minimum}.
The volume comparison in \cref{lem:nonflat-volume-equality} shows that every $u$ with $h_*u=v$ gives another minimizer $qu$.
Uniqueness therefore gives $qu=v$, and hence
\[
 h_*v=qv.
\]
Finally, a rational perturbation of the monomial weights yields a divisorial eigenvaluation with the same center; see the proof of \cref{thm:generic-eigen-log-pair}.

We also study normalized ramification divisors.
Gongyo conjectured that, for a $q$-polarized endomorphism $f\colon X\to X$ of a smooth projective variety, the pair $(X,R_f/(q-1))$ is log canonical after replacing $f$ by an iterate; see \cite{MengThreefoldCY}*{Conjecture~1.3}.
Meng established the surface case apart from certain endomorphisms of $\mathbf P^2$ \cite{MengThreefoldCY}*{Section~4}, and Luo--Meng settled the remaining case \cite{LuoMeng26}*{Theorem~1.3}.
We answer this conjecture affirmatively as follows.

\begin{theorem}
\label{thm:intro-normalized-ramification}
Let $X$ be a normal $\Q$-Gorenstein projective complex variety admitting a $q$-polarized endomorphism $f\colon X\to X$.
Then there exists an integer $r_0\geq1$ such that
\[
 \left(X,\frac{R_{f^{mr_0}}}{q^{mr_0}-1}\right)
\]
is log canonical for every integer $m\geq1$.
\end{theorem}

\begin{ackn}
The author is also grateful to Sheng Meng for helpful comments. 
The author was supported by JSPS KAKENHI Grant number JP24K16889.
\end{ackn}

\medskip
\noindent\textbf{Use of generative AI.}
The definition of the current $T_{h,B}$ and the overall strategy of passing to suitable equivariant birational models and proceeding by induction on dimension were conceived by the author.
The proofs outside Section~\ref{sec:eigenvaluations} were developed through discussions between the author and GPT (OpenAI).
The proofs in Section~\ref{sec:eigenvaluations} were developed primarily by GPT through interactions with the author.
GPT was also used to assist with the exposition and to identify relevant references.
The author has verified the correctness of all mathematical content in this paper independently of GPT.

\section{Notation and terminology}
\label{sec:notation}

\begin{itemize}
\item
Throughout these conventions, $X$ denotes a normal integral separated scheme essentially of finite type over $\C$.
A variety is an integral separated scheme of finite type over $\C$.

We write $K(X)$ for the function field of $X$.
If $X$ is a variety, we write $X^{\mathrm{an}}$ for its analytification.
Canonical divisors on birational models are chosen compatibly.

\item
For a finite surjective morphism $f\colon X'\to X$ between normal varieties, we use finite pullback of Weil divisors.
For a prime divisor $D$ on $X$, set
\[
 f^*D:=\sum_{E:\,f(E)=D}e(E/D)E,
\]
where $e(E/D)$ is the ramification index of the corresponding extension of discrete valuation rings.
Extend this definition $\Q$-linearly.
For $\Q$-Cartier divisors it agrees with Cartier pullback.

\item
\label{it:weak-numerical}
Let $X$ be a normal projective variety of dimension
$n>0$.
Two $\R$-Weil divisors $D,D'$ on $X$ are
\emph{weakly numerically equivalent}, written
$D\equiv_w D'$, if
\[
 (D-D')\cdot H_1\cdots H_{n-1}=0
\]
for all Cartier divisors $H_1,\ldots,H_{n-1}$ on $X$.
For $n=1$, this means $\deg(D-D')=0$.
See \cite{MengZhangIntAmplified}*{Section 2}.
For a $\Q$-Cartier divisor $D$ on $X$, we write
$[D]_{\mathrm{num}}\in N^1(X)_\R$ for its numerical class and use
$D\equiv D'$ to express numerical equivalence.

\item
A \emph{log pair} $(X,B)$ consists of $X$ as above and an effective $\Q$-Weil divisor $B$ such that $K_X+B$ is $\Q$-Cartier.
A $\Q$-Weil divisor is a \emph{boundary} if all its coefficients belong to $[0,1]$.

\item
A \emph{prime divisor over $X$} is a prime divisor $E$
on a normal integral scheme $Y$ equipped with a proper
birational morphism $\pi\colon Y\to X$.
For a $\Q$-Weil divisor $B$ such that $K_X+B$ is $\Q$-Cartier, the \emph{log discrepancy} of $(X,B)$ is
\[
 A_{X,B}(E)
 :=1+\operatorname{coeff}_E
       \bigl(K_Y-\pi^*(K_X+B)\bigr).
\]

\item
A log pair $(X,B)$ is \emph{log canonical} (lc),
respectively \emph{Kawamata log terminal} (klt), if
\[
 A_{X,B}(E)\geq0,
 \qquad\text{respectively}\qquad
 A_{X,B}(E)>0,
\]
for every prime divisor $E$ over $X$.
We denote the non-klt and non-lc loci by $\operatorname{Nklt}(X,B)$ and $\operatorname{Nlc}(X,B)$.

An \emph{lc center} of $(X,B)$ is the image $\pi(E)$
of a prime divisor $E$ over $X$ with
$A_{X,B}(E)=0$, provided that $(X,B)$ is lc at the
generic point of $\pi(E)$.

\item
For a field $K$ containing $\C$, a \emph{valuation}
of $K$ means an additive valuation
\[
 v\colon K^\times\longrightarrow\R.
\]
We set $v(0)=+\infty$ and allow the trivial valuation, denoted by $0$.

We denote by $\operatorname{Val}_X$ the set of all valuations of $X$.
We equip $\operatorname{Val}_X$ with the weakest topology for which every evaluation map
\[
\operatorname{Val}_X\longrightarrow\R,
 \qquad v\longmapsto v(f),
 \qquad f\in K(X)^\times,
\]
is continuous.

\item
A valuation $v$ of $K(X)$ is \emph{centered at}
$\xi\in X$ if
\[
 v(f)\geq0
 \quad(f\in\mathcal O_{X,\xi}),
 \qquad
 \{f\in\mathcal O_{X,\xi}\mid v(f)>0\}
 =\mathfrak m_{X,\xi}.
\]
A \emph{valuation of $X$} is a valuation of $K(X)$
having a center on $X$.
The center point is unique because $X$ is separated.
We write
\[
 c_X(v):=\overline{\{\xi\}}.
\]
For an irreducible closed subset $C\subseteq X$,
we say that $v$ is \emph{centered at $C$} if
$c_X(v)=C$.
Write $\kappa(v)$ for the residue field of the valuation ring
$\mathcal O_v:=\{f\in K(X)\mid v(f)\geq0\}$, and define the \emph{rational rank} by
\[
 \operatorname{ratrank}(v)
 :=\dim_\Q\bigl(v(K(X)^\times)\otimes_{\mathbb Z}\Q\bigr).
\]
If $v$ is centered at $\xi$, then $\kappa(\xi)$ embeds naturally in $\kappa(v)$;
the residue transcendence degree relative to this center means
$\operatorname{trdeg}_{\kappa(\xi)}\kappa(v)$.

\item
\label{it:valuation-evaluation}
Let $v$ be a valuation of $X$ centered at $\xi$.
For a nonzero coherent ideal $\mathfrak b$ defined
near $\xi$ and a $\Q$-Cartier divisor $D$ on $X$, set
\[
 v(\mathfrak b)
 :=\min_{0\ne f\in\mathfrak b_\xi}v(f),
 \qquad
 v(D):=\frac{v(g)}m,
\]
where $m>0$ is an integer and
$mD=\operatorname{div}(g)$ near $\xi$.

\item
\label{it:valuation-types}
A valuation of $X$ is \emph{divisorial} if it is
$c\,\operatorname{ord}_E$ for some $c>0$ and some
prime divisor $E$ over $X$.
We write
\[
 c_X(E):=c_X(\operatorname{ord}_E).
\]

For a reduced simple normal crossings (snc) divisor on a regular scheme,
a \emph{stratum} means an irreducible component of an intersection of one or more of its components.
Thus strata are closed subschemes.

\item
\label{it:valuation-discrepancy}
Let $B$ be a $\Q$-Weil divisor such that $K_X+B$ is $\Q$-Cartier.
For a divisorial valuation $c\ord_E$, where $c>0$, set
\[
 A_{X,B}(c\ord_E):=cA_{X,B}(E).
\]
Following the lower semicontinuous extension construction of
\cite{BoucksomJonssonDivisorial}*{Definition A.2 and equation (A.2)},
stated there for projective pairs, we define, for every
nontrivial valuation $v\in\operatorname{Val}_X$,
\begin{equation}\label{eq:valuation-discrepancy-envelope}
 A_{X,B}(v)
 :=
 \sup_{\mathcal U\ni v}
 \inf_{\substack{z\in\mathcal U\\z\ \mathrm{divisorial}}}
 A_{X,B}(z),
\end{equation}
where $\mathcal U$ ranges over the open neighborhoods of $v$
in $\operatorname{Val}_X$.
We set $A_{X,B}(0)=0$.

For smooth $X$ and $B=0$, this agrees with the definition by
monomial retractions in
\cite{JonssonMustata}*{Section 5.2, equation (5.2)};
see \cite{BoucksomJonssonDivisorial}*{Proposition A.3}.
More generally, if $\pi\colon Y\to X$ is a log resolution and
$K_Y+B_Y=\pi^*(K_X+B)$, then
\[
 A_{X,B}(v)=A_{Y,0}(v)-v(B_Y)
 \qquad(v\in\operatorname{Val}_X),
\]
by \cite{BoucksomJonssonDivisorial}*{Proposition A.4}.
In particular, the values lie in $\R\cup\{+\infty\}$.

We use the same construction on localizations.
More precisely, for $\xi\in X$, put
\[
 X_\xi:=\operatorname{Spec}\mathcal O_{X,\xi},
 \qquad B_\xi:=B|_{X_\xi}.
\]
In defining $A_{X_\xi,B_\xi}$ by
\eqref{eq:valuation-discrepancy-envelope}, both the neighborhoods
and the divisorial valuations are taken in
$\operatorname{Val}_{X_\xi}$.
Then
\begin{equation}\label{eq:valuation-discrepancy-localization}
 A_{X_\xi,B_\xi}(v)=A_{X,B}(v)
 \qquad(v\in\operatorname{Val}_{X_\xi}).
\end{equation}
Indeed, the preceding transformation formula reduces this assertion to the monomial-retraction description on a log resolution in \cite{JonssonMustata}*{Section 5.2, equation (5.2)}.

If $(X,B)$ is lc, its divisorial log discrepancies are
nonnegative, so \eqref{eq:valuation-discrepancy-envelope}
gives $A_{X,B}(v)\geq0$ for every valuation $v$ of $X$.

\item
\label{it:valuation-dynamics}
Let $f\colon X'\to X$ be a dominant morphism,
where $X'$ is a normal integral separated scheme essentially of finite type over $\C$.
For a valuation $u$ of $K(X')$, define
\[
 (f_*u)(a):=u(f^*a)
 \qquad(a\in K(X)^\times).
\]
If $u$ is a valuation of $X'$ centered at $\xi'$,
then $f_*u$ is a valuation of $X$ centered at
$f(\xi')$.

For a valuation $v$ of $K(X)$, an
\emph{extension of $v$ through $f$} is a valuation
$u$ of $K(X')$ satisfying $f_*u=v$.

Let $h\colon X\to X$ be a dominant endomorphism.
A nontrivial valuation $v$ of $X$ is an
\emph{eigenvaluation with eigenvalue $\lambda$} if
\[
 h_*v=\lambda v,
 \qquad \lambda>0.
\]
\end{itemize}

\section{Ramification series and corrections}\label{sec:ramification-series}

We construct finite approximations to the ramification series with numerical class $-(K_X+B)$.
We first bound the values of the series on valuations in terms of log discrepancies, and then correct its finite partial sums.

\Needspace{6\baselineskip}
\begin{definition}\label{def:dynamical-triple}
\leavevmode\par
\begin{enumerate}
\item\label{it:dynamical-triple}
A \emph{dynamical triple} $(X,B,h)$ is a log pair $(X,B)$ together with a finite endomorphism $h\colon X\to X$ satisfying
\[
 R_{h,B}:=R_h+B-h^*B\geq0.
\]
Here $R_h$ is the ramification divisor.
We call $R_{h,B}$ the \emph{log ramification divisor}.
The triple is \emph{projective} if $X$ is projective, and \emph{lc} or \emph{klt} if $(X,B)$ has the corresponding property.
\item\label{it:amplification}
A projective dynamical triple $(X,B,h)$ is \emph{$q$-polarized} if $h^*H\sim qH$ for an ample Cartier divisor $H$ and an integer $q>1$, and \emph{polarized} if such a $q$ exists.
It is \emph{int-amplified} if $h^*H-H$ is ample for an ample Cartier divisor $H$.
We use these terms for $h$ as well.
\item\label{it:total-invariance}
A closed subset $C\subset X$ is \emph{totally invariant} under an endomorphism $h\colon X\to X$ if $h^{-1}(C)=C$.
\item\label{it:ramification-series} 
Let $(X,B,h)$ be an int-amplified dynamical triple and define $P_h:=(\deg h)^{-1}h_*$.
We use the same notation for its induced action on $N^1(X)_\R$ and for normalized proper pushforward of currents on $X^{\mathrm{an}}$.
For a positive integer $N$ and a real valuation $v$ of $X$, define 
\begin{equation}\label{eq:notation}
 S_{h,B,N}:=\sum_{j=1}^N P_h^jR_{h,B},\qquad
 s_{h,B}(v):=\sum_{j\geq1}v(P_h^jR_{h,B}) \in [0,\infty].
\end{equation}
For a prime divisor $E$ over $X$, write $s_{h,B}(E)=s_{h,B}(\ord_E)$.
\end{enumerate}
\end{definition}

Note that a finite endomorphism of an integral variety is surjective.
The effectivity condition in Definition~\ref{def:dynamical-triple}(\ref{it:dynamical-triple}) says that $(X,B)$ is an $h$-pair in the terminology of \cite{YoshikawaFanoFibrations}*{Definition 3.1}.

\subsection{Ramification series and uniform tail estimates}\label{sec:tails}

\begin{lemma}\label{lem:finite-divisorial}
Let $f\colon X'\to X$ be finite and surjective between normal varieties.
Identify $K=K(X)$ with a subfield of $L=K(X')$ via $f^*$, and put $d=[L:K]$.
Let $v$ be a valuation of $X$, and let $u_1,\ldots,u_q$ be its distinct extensions to $L$.
Then, there are positive integers $n_i$, with $\sum_i n_i=d$, such that
\begin{equation}\label{eq:finite-norm}
 v(\Nm_{L/K}(a))=\sum_{i=1}^q n_i u_i(a)\qquad(a\in L^*).
\end{equation}
For every $\Q$-Cartier divisor $D$ on $X'$, its pushforward is $\Q$-Cartier.
Every irreducible component of $f^{-1}(c_X(v))$ is the center of an extension.
If $(X,B)$ and $(X',B')$ are log pairs, put $R_f^{\log}=R_f+B'-f^*B$, which is $\Q$-Cartier.
For $v \in \operatorname{Val}_X$, one has
\begin{equation}\label{eq:finite-divisorial-discrepancy}
 A_{X,B}(v)=A_{X',B'}(u_i)+u_i(R_f^{\log}).
\end{equation}
The equality holds in $\R\cup\{+\infty\}$.
In particular, $v$ has finite log discrepancy with respect
to $(X,B)$ if and only if any of its extensions has finite
log discrepancy with respect to $(X',B')$.
\end{lemma}

\begin{proof}
Choose a finite Galois closure $M/K$ containing $L$, and put
\[
 G=\operatorname{Gal}(M/K),\qquad
 H=\operatorname{Gal}(M/L).
\]
Extend $v$ to a valuation $w$ of $M$.
By extension and conjugacy of valuations, the distinct restrictions
$(w\circ g)|_L$, for $gH\in G/H$, are precisely the $u_i$;
see \cite{EnglerPrestel}*{Sections 3.1--3.2}.
Thus, setting
\[
 n_i=\#\{gH\in G/H\mid (w\circ g)|_L=u_i\},
\]
the identity
\[
 \Nm_{L/K}(a)=\prod_{gH\in G/H}g(a)
\]
gives \eqref{eq:finite-norm} and $\sum_i n_i=d$.

For a $\Q$-Cartier divisor $D$ on $X'$, fix $x\in X$
and choose $m>0$ such that $mD$ is Cartier.
On the semilocal scheme
$X'\times_X\Spec\mathcal O_{X,x}$, this Cartier divisor
has a single rational equation $a\in L^*$.
The formula
\begin{equation}\label{eq:norm-principal-divisor}
 f_*\operatorname{div}(a)
 =\operatorname{div}(\Nm_{L/K}(a))
\end{equation}
from \cite{StacksProject}*{Tag 02RT} shows that $mf_*D$
is principal near $x$.
Hence $f_*D$ is $\Q$-Cartier.

The valuation $w$ has a center on the normalization of $X$
in $M$, by properness.
The group $G$ acts transitively on the points above the generic
point of $c_X(v)$; see \cite{StacksProject}*{Tag 0BRK}.
Consequently, the centers of its conjugates, followed by restriction
to $L$, give every irreducible component of $f^{-1}(c_X(v))$.

Choose canonical divisors compatibly with $f$.
Then
\[
 R_f^{\log}
 =K_{X'}+B'-f^*(K_X+B)
\]
is $\Q$-Cartier.
Suppose first that $v=c\ord_E$ is divisorial.
On the normalization of the base change of a model carrying $E$,
let $F_i$ be the prime divisors above $E$, with ramification
indices $e_i$.
The DVR ramification formula gives
\[
 u_i=\frac{c}{e_i}\ord_{F_i},\qquad
 A_{X',B'}(F_i)+\ord_{F_i}(R_f^{\log})
 =e_iA_{X,B}(E);
\]
see \cite{KM98}*{Proposition 5.20}.
Multiplication by $c/e_i$ proves
\eqref{eq:finite-divisorial-discrepancy} in this case.

We extend the equality to arbitrary valuations.
The trivial valuation has only the trivial extension, for which
the equality is immediate.
We first show that the restriction map
\[
 \rho:=f_*\colon\operatorname{Val}_{X'}
                  \longrightarrow\operatorname{Val}_X
\]
is open.
Let $\widetilde X$ be the normalization of $X$ in $M$.
Write
\[
 p\colon\widetilde X\to X',
 \qquad q:=f\circ p\colon\widetilde X\to X
\]
for the induced finite morphisms.

We claim that $q_*$ is closed.
Let $\mathcal F\subseteq\operatorname{Val}_{\widetilde X}$ be closed, and let $v_\lambda\in q_*(\mathcal F)$ be a net converging to $v\in\operatorname{Val}_X$.
Choose $u_\lambda\in\mathcal F$ with $q_*u_\lambda=v_\lambda$.
For $a\in M^*$, write its minimal monic polynomial over $K$ as $\sum_{j=0}^{r}c_jT^j$.
Since $\sum_jc_ja^j=0$, the minimum of $u_\lambda(c_ja^j)$ over the nonzero coefficients is attained at least twice.
Thus, for each $\lambda$, there are indices $i<j$ with $c_ic_j\ne0$ such that $u_{\lambda}(c_ia^i)=u_{\lambda}(c_ja^j)$.
Consequently,
\[
u_\lambda(a)=\frac{v_\lambda(c_i)-v_\lambda(c_j)}{j-i}.
\]
For each nonzero coefficient $c_i$, the net $v_\lambda(c_i)$ converges to the finite value $v(c_i)$.
Since there are only finitely many possible pairs $(i,j)$, for each $a\in M^*$ there exist $C_a>0$ and $\lambda_a$ such that $|u_\lambda(a)|\leq C_a$ whenever $\lambda\geq\lambda_a$.
Compactness of $[-\infty,+\infty]^{M^*}$ gives a convergent subnet, and the eventual bound in each coordinate ensures that its limit $u$ is finite on $M^*$.
The valuation identities pass to the limit, and $u|_K=v$.
Since $v$ has a center on $X$ and $q$ is proper, the valuative criterion of properness gives a center of $u$ on $\widetilde X$.
The subnet therefore converges in $\operatorname{Val}_{\widetilde X}$, and closedness of $\mathcal F$ gives $u\in\mathcal F$.
Thus $v=q_*u\in q_*(\mathcal F)$, proving the claim.

The continuous surjection $q_*$ is therefore a quotient map.
Its fibers are the $G$-orbits, by Galois conjugacy of
extensions of valuations.
For every open subset
$\mathcal O\subseteq\operatorname{Val}_{\widetilde X}$,
\[
 q_*^{-1}\bigl(q_*(\mathcal O)\bigr)
   =\bigcup_{\sigma\in G}\sigma(\mathcal O)
\]
is open, so $q_*$ is open.
Since $p_*$ is continuous and surjective, every open subset
$\mathcal U\subseteq\operatorname{Val}_{X'}$ satisfies
\[
 \rho(\mathcal U)=q_*\bigl(p_*^{-1}(\mathcal U)\bigr).
\]
Consequently, $\rho$ is open.

Restriction preserves divisoriality in both directions, since finite field extensions preserve rational rank and residue transcendence degree.
Moreover, evaluation on the $\Q$-Cartier divisor
$R_f^{\log}$ is continuous and finite-valued.
Fix an extension $u$ of $v$.
By the definition in
\eqref{eq:valuation-discrepancy-envelope} and the formula in the divisorial case,
\[
\begin{aligned}
 A_{X',B'}(u)+u(R_f^{\log})
 &=
 \sup_{\mathcal W\ni u}
 \inf_{\substack{w\in\mathcal W\\
                 w\ \mathrm{divisorial}}}
 \bigl(A_{X',B'}(w)+w(R_f^{\log})\bigr)\\
 &=
 \sup_{\mathcal W\ni u}
 \inf_{\substack{z\in\rho(\mathcal W)\\
                 z\ \mathrm{divisorial}}}
 A_{X,B}(z)\\
 &=
 \sup_{\mathcal V\ni v}
 \inf_{\substack{z\in\mathcal V\\
                 z\ \mathrm{divisorial}}}
 A_{X,B}(z)\\
 &=A_{X,B}(v).
\end{aligned}
\]
Here $\mathcal W$ and $\mathcal V$ range over open
neighborhoods of $u$ and $v$, respectively.
The third equality follows because openness and continuity
of $\rho$ make the sets $\rho(\mathcal W)$ a neighborhood
basis at $v$.

All equalities hold in $\R\cup\{+\infty\}$.
Since $u(R_f^{\log})$ is finite, finiteness of either log
discrepancy is equivalent to finiteness of the other.
\end{proof}

\begin{lemma}\label{lem:basic-dynamical-properties}
Let $(X,B,h)$ be an int-amplified dynamical triple.
Then the pair $(X,B)$ is lc, each $P_h^jR_{h,B}$ is effective and $\Q$-Cartier, and $P_h=(h^*)^{-1}$ on $N^1(X)_\R$.
For every positive integer $N$, the pair $(X,B+S_{h,B,N})$ is lc.
The spectral radius $\rho(P_h)$ on $N^1(X)_\R$ is less than one.
For any norm on this space and any $\rho(P_h)<\theta<1$, as the positive integer $N$ tends to infinity,
\begin{equation}\label{eq:spectral-decay}
 \|P_h^N\|=O(\theta^N).
\end{equation}
\end{lemma}
\begin{proof}
Log canonicity is \cite{CZ}*{Theorem 1.8}.
By Lemma~\ref{lem:finite-divisorial}, the ramification formula 
\[
R_{h,B}\sim_\Q K_X+B-h^*(K_X+B)
\]
makes $R_{h,B}$ $\Q$-Cartier.
The projection formula gives $h_*h^*=(\deg h)\,\mathrm{id}$.
Every eigenvalue of $h^*|_{N^1(X)_\R}$ has modulus greater than one by \cite{MengZhangIntAmplified}*{Proposition 2.4}, so $\rho(P_h)<1$.
Jordan normal form gives \eqref{eq:spectral-decay}.
Let $v$ be a valuation of $X$ with $A_{X,B}(v)<\infty$, let $u_1,\ldots,u_q$ be its extensions through $h$, and let $n_1,\ldots,n_q$ be their norm weights from \cref{lem:finite-divisorial}.
For every effective $\Q$-Cartier divisor $D$ on $X$, the norm and discrepancy formulas give
\begin{equation}\label{eq:finite-stationarity}
 (\deg h)\bigl(A_{X,B}(v)-v(P_h(R_{h,B}+D))\bigr)
 =\sum_i n_i\bigl(A_{X,B}(u_i)-u_i(D)\bigr).
\end{equation}
Induction with $D=0,S_{h,B,1},\ldots$ gives
$v(S_{h,B,N})\leq A_{X,B}(v)$ for every $N$.
This inequality is automatic when $A_{X,B}(v)=\infty$; in particular, $(X,B+S_{h,B,N})$ is lc.
For later use, letting $N\to\infty$ with $D=S_{h,B,N}$ in \eqref{eq:finite-stationarity} yields
\begin{equation}\label{eq:nonflat-stationarity}
 (\deg h)\bigl(A_{X,B}(v)-s_{h,B}(v)\bigr)
 =\sum_i n_i\bigl(A_{X,B}(u_i)-s_{h,B}(u_i)\bigr)
 \qquad(A_{X,B}(v)<\infty).
\end{equation}
All terms are finite and nonnegative.
\end{proof}

\begin{proposition}
\label{prop:log-lc-center-avoidance}
Let $(X,B,h)$ be an int-amplified dynamical triple.
If an lc center $C$ is totally invariant, then $C\nsubseteq\Supp R_{h,B}$ and $C\nsubseteq\Supp(P_h^jR_{h,B})$ for every $j\geq1$.
After replacing $h$ by an iterate, every lc center is totally invariant.
\end{proposition}
\begin{proof}
Let $C$ be an lc center and choose a divisorial valuation $v$ with $c_X(v)=C$ and $A_{X,B}(v)=0$.
Every extension $u$ with $h_*u=v$ satisfies $0=A_{X,B}(u)+u(R_{h,B})$ by \eqref{eq:finite-divisorial-discrepancy}.
Since both summands are nonnegative, we have $u(R_{h,B})=0$ and $c_X(u)$ is an lc center. 
Thus every component of $h^{-1}(C)$ is not contained in $\Supp R_{h,B}$.
If $C$ is totally invariant, we have $C\nsubseteq\Supp (P^j_hR_{h,B})$.
We note that $\Supp(P_h^jR_{h,B})=h^j(\Supp R_{h,B})$.

On a log resolution $\rho\colon Y\to X$, write $K_Y+B_Y=\rho^*(K_X+B)$.
Every coefficient of $B_Y$ is at most one.
The monomial discrepancy formula shows that divisorial valuations of zero discrepancy have centers among the strata of the coefficient-one divisor.
The lc centers are the images of strata of the coefficient-one divisor of $B_Y$, so there are only finitely many.
Each $h^{-a}(C)$ is a union of lc centers, and only finitely many such unions occur.
Hence $h^{-a}(C)=h^{-b}(C)$ for some $0\leq a<b$.
Applying $h^a$ gives $C=h^{-(b-a)}(C)$.
Taking a common multiple of these finitely many periods, every lc center is totally invariant.
\end{proof}

\begin{lemma}
\label{lem:lc-log-degree}
Let $(X,B)$ be a projective lc log pair.
Assume $n=\dim X>0$ and fix a very ample Cartier divisor $H$ on $X$.
Then there is $C>0$ such that, for every effective $\Q$-Cartier divisor $D$ on $X$ containing no lc center and every valuation $v$ of $X$,
\begin{equation}\label{eq:lc-log-degree}
 v(D)\leq C(D\cdot H^{n-1})A_{X,B}(v).
\end{equation}
\end{lemma}

\begin{proof}
Fix a projective log resolution $\rho\colon Z\to X$, with $K_Z+B_Z=\rho^*(K_X+B)$, $B_Z=\sum b_iE_i$, $b_i\leq1$.
Set $E_0=\sum_{b_i=1}E_i$.
Since $(Z,\sum E_i)$ is lc, we have the inequality $\sum_{0<b_i<1}v(E_i)\leq A_Z(v)-v(E_0)$, and it shows that
\begin{equation}\label{eq:lc-discrepancy-comparison}
 A_{X,B}(v)\geq cA_{Z,E_0}(v),\qquad
 c=1-\max(\{0\}\cup\{b_i \mid 0<b_i<1\})>0.
\end{equation}

Write $E_0=\sum_{i=1}^r F_i$, and let $\mathscr S$
consist of $Z$ and the strata of $E_0$.
For $W\in\mathscr S$, put
\[
 D_W:=(\rho^*D)|_W,
 \qquad
 E_W:=\sum_{i:\,W\not\subseteq E_i, b_i=1}E_i|_W.
\]
Thus $E_Z=E_0$, and $E_W$ is a reduced snc divisor
whose components belong to $\mathscr S$.
Since every $W\ne Z$ maps to an lc center of $(X,B)$,
\[
 W\not\subseteq\operatorname{Supp}(\rho^*D)
 \qquad(W\in\mathscr S).
\]
Hence $D_W$ is a well-defined effective
$\Q$-Cartier divisor.

Assume $D\ne0$ and put $d_D:=D\cdot H^{n-1}>0$.
For each positive-dimensional $W\in\mathscr S$,
fix a very ample Cartier divisor $H_W$.
The linear functional
\[
 \ell_W\colon N^1(X)_\R\longrightarrow\R,
 \qquad
 \ell_W(\alpha)
 :=((\rho|_W)^*\alpha)\cdot H_W^{\dim W-1}
\]
is bounded on the compact slice
\[
 \{\alpha\in\overline{\operatorname{Eff}}^{\,1}(X)
       \mid \alpha\cdot H^{n-1}=1\}.
\]
By the compactness, there is $M_W>0$, independent of $D$, such that
\[
 D_W\cdot H_W^{\dim W-1}
 =d_D\,\ell_W\!\left(\frac{[D]_{\mathrm{num}}}{d_D}\right)
 \leq M_Wd_D.
\]
Since $\mathscr S$ is finite, there is an integer $M>0$, such that
\[
 D_W\cdot H_W^{\dim W-1}\leq Md_D
 \qquad(W\in\mathscr S,\ \dim W>0).
\]

For a divisorial valuation $w$ of $W$, let $\xi$
be its center.
By \cite{JonssonMustata}*{Proposition 5.10} and
the multiplicity bound obtained by intersecting
with general hyperplanes through a general point
of $\overline{\{\xi\}}$,
\[
 \begin{aligned}
 w(D_W)
 &\leq A_W(w)\operatorname{ord}_{\xi}(D_W)\\
 &\leq A_W(w)
          (D_W\cdot H_W^{\dim W-1})\\
 &\leq Md_D A_W(w).
 \end{aligned}
\]
Here $\operatorname{ord}_{\xi}(D_W)$ denotes the
order of a local equation, divided by a Cartier
multiple of $D_W$.
Set $\varepsilon:=(Md_D)^{-1}$. Then
\[
 (W,\varepsilon D_W)\text{ is lc}
 \qquad(W\in\mathscr S).
\]

We prove by induction on $\dim W$ that
\[
 (W,E_W+\varepsilon D_W)\text{ is lc}
 \qquad(W\in\mathscr S).
\]
The assertion is immediate when $\dim W=0$.
For a prime component $T$ of $E_W$, adjunction gives
\[
 (K_W+E_W+\varepsilon D_W)|_T
 =K_T+E_T+\varepsilon D_T.
\]
The right-hand pair is lc by induction.
Thus inversion of adjunction
\cite{Kawakita}*{Theorem} gives
\[
 (T,E_T+\varepsilon D_T)\text{ lc}
 \quad\Longrightarrow\quad
 (W,E_W+\varepsilon D_W)\text{ lc near }T.
\]
Away from $\operatorname{Supp}E_W$, log canonicity
follows from that of $(W,\varepsilon D_W)$.
This completes the induction.

Taking $W=Z$, we obtain
\[
 (Z,E_0+\varepsilon\rho^*D)\text{ is lc}.
\]
Consequently, for every valuation $v$ of $X$
with finite log discrepancy,
\[
 \begin{aligned}
 0
 &\leq A_{Z,E_0+\varepsilon\rho^*D}(v)\\
 &=A_{Z,E_0}(v)-\varepsilon v(D),
 \end{aligned}
\]
and hence
\[
 v(D)
 \leq Md_D A_{Z,E_0}(v)
 \leq \frac{M}{c}\,d_D A_{X,B}(v)
\]
by \eqref{eq:lc-discrepancy-comparison}.
The case $D=0$ is immediate.
\end{proof}

\begin{theorem}\label{thm:lc-log-tail}
Let $(X,B,h)$ be an int-amplified dynamical triple whose lc centers are all totally invariant, and choose $\rho(P_h)<\theta<1$.
There is $C>0$ such that, for every positive integer $N$ and every valuation $v$ of $X$,
\begin{equation}\label{eq:lc-log-tail}
 s_{h,B}(v)\leq A_{X,B}(v),\qquad
 0\leq s_{h,B}(v)-v(S_{h,B,N})\leq C\theta^NA_{X,B}(v).
\end{equation}
On any set of these valuations on which $A_{X,B}$ is bounded, $s_{h,B}$ is continuous in the topology of pointwise convergence of valuations.
\end{theorem}
\begin{proof}
In dimension zero every term is zero. 
Set $n=\dim X>0$ and fix a very ample Cartier divisor $H$ on $X$.
The proof of \cref{lem:basic-dynamical-properties} gives
$v(S_{h,B,N})\leq A_{X,B}(v)$ for every $v$ and $N$.
Letting $N\to\infty$ gives $s_{h,B}(v)\leq A_{X,B}(v)$.
By \eqref{eq:spectral-decay}, there exists $C_0>0$ such that
\[
 0\leq (P_h^jR_{h,B})\cdot H^{n-1}\leq C_0\theta^j \qquad (j \geq 1).
\]
The support of $P_h^jR_{h,B}$ contains no lc center of $(X,B)$ by Proposition~\ref{prop:log-lc-center-avoidance}, so there exists $C_1>0$ such that 
\[
 v(P_h^jR_{h,B})\leq C_1\theta^jA_{X,B}(v) \qquad (j \geq 1)
\]
by Lemma~\ref{lem:lc-log-degree}.
Thus, we have
\[
s_{h,B}(v)-v(S_{h,B,N})=\sum_{j>N}v(P_h^jR_{h,B}) \leq \sum_{j>N}C_1\theta^j A_{X,B}(v)\leq \frac{\theta}{1-\theta}C_1\theta^NA_{X,B}(v).
\]
The constant $C:=\frac{\theta}{1-\theta}C_1$ satisfies the desired inequality.

For every positive integer $N$, the function
\[
\operatorname{Val}_X \to \R, \qquad  v\longmapsto v(S_{h,B,N})
\]
is continuous.
On any subset $\mathcal V$ of $\operatorname{Val}_X$ with
$\sup_{v\in\mathcal V}A_{X,B}(v)<+\infty$,
\eqref{eq:lc-log-tail} gives
\[
 \sup_{v\in\mathcal V}
 \bigl|s_{h,B}(v)-v(S_{h,B,N})\bigr|
 \leq C\theta^N
       \sup_{v\in\mathcal V}A_{X,B}(v)
 \longrightarrow0.
\]
Thus $s_{h,B}|_{\mathcal V}$ is continuous as
a uniform limit of continuous functions.
\end{proof}

For any int-amplified dynamical triple $(X,B,h)$, $P_h h^*=\mathrm{id}$ gives
\begin{equation}\label{eq:iterate-blocks}
 P_{h^a}^{\,j}R_{h^a,B}
    =\sum_{k=a(j-1)+1}^{aj}P_h^kR_{h,B},\qquad a,j\geq1.
\end{equation}
In particular, for all positive integers $a,N$,
\[
 S_{h^a,B,N}=S_{h,B,aN},\qquad s_{h^a,B}=s_{h,B}.
\]

\subsection{Interpolation conditions and numerical corrections}\label{sec:correction}
The partial sums $S_{h,B,N}$ need not have numerical class $-(K_X+B)$.
We correct them by adding linear combinations of finitely many terms $P_h^jR_{h,B}$, with coefficients tending to zero as $N\to\infty$.

\begin{definition}\label{def:moment-conditions}
Let $(X,B,h)$ be an int-amplified dynamical triple.
A \emph{coefficient family} is a family
\[
 \mathbf c=(c_{j,N})_{
  1\leq j\leq m,\ N\geq\max\{1,m\}}
\]
of rational numbers, for some integer $m\geq0$, such that
\[
 \lim_{N\to\infty}c_{j,N}=0
 \qquad(1\leq j\leq m).
\]
For each coefficient family, we use $m$ to denote its number of coefficient sequences; the empty family ($m=0$) is allowed.
For a coefficient family $\mathbf c$, we define the following notions.
\begin{enumerate}
\item\label{it:corrections} 
The \emph{correction associated with $\mathbf c$} is the $\Q$-Weil divisor
\begin{equation}\label{eq:closed-correction}
 \mathcal C_{h,N}(B;\mathbf c)=B+S_{h,B,N}
       +\sum_{j=1}^m c_{j,N}P_h^jR_{h,B}.
\end{equation}
\item\label{it:moments} 
For a finite set $\Lambda\subset\Q_{>1}$, the family $\mathbf c$ satisfies the \emph{interpolation conditions for $\Lambda$} if, for every $N\geq\max\{1,m\}$,
\begin{equation}\label{eq:moment}
 \sum_{j=1}^m c_{j,N}\lambda^{-j}
       =\frac{\lambda^{-N}}{\lambda-1}
 \qquad(\lambda\in\Lambda).
\end{equation}
These conditions prescribe the values of $\sum_{j=1}^m c_{j,N}t^j$ at $t=\lambda^{-1}$ for $\lambda\in\Lambda$.
For a positive integer $a$, write $\Lambda^a=\{\lambda^a \mid \lambda\in\Lambda\}$.
\end{enumerate}
\end{definition}

We impose the interpolation conditions only when stated explicitly.
For the data of Definition~\ref{def:moment-conditions}, the identity
\[
 \mathcal C_{h,N}(B;\mathbf c)-B
 =\sum_{j=1}^m(1+c_{j,N})P_h^jR_{h,B}
   +\sum_{j=m+1}^N P_h^jR_{h,B}
\]
shows that $\mathcal C_{h,N}(B;\mathbf c)\geq B$ for all sufficiently large $N$.

\begin{lemma}\label{lem:numerical-correction}
Let $(X,B,h)$ be an int-amplified dynamical triple.
Set $L=-(K_X+B)$.
For every finite set $\Lambda\subset\Q_{>1}$, possibly empty, there exists a coefficient family $\mathbf c=(c_{j,N})$ satisfying the interpolation conditions \eqref{eq:moment} for $\Lambda$ such that $K_X+\mathcal C_{h,N}(B;\mathbf c)\equiv0$ for every integer $N\geq\max\{1,m\}$.
\end{lemma}

\begin{proof}
The ramification formula and $P_h=(h^*)^{-1}$ in $N^1(X)_{\Q}$ give
\begin{equation}\label{eq:numerical-telescope}
 L-P_hL\equiv P_hR_{h,B},\qquad
 S_{h,B,N}\equiv L-P_h^NL.
\end{equation}
Write $\Lambda=\{\lambda_1,\ldots,\lambda_s\}$.
When $\Lambda=\varnothing$, take $s=0$ and omit all $\lambda_i$-coordinates below; the interpolation conditions are then vacuous.
Define the $\Q$-linear endomorphism
\[
 \mathcal P\colon
 N^1(X)_\Q\oplus\Q^s
 \longrightarrow N^1(X)_\Q\oplus\Q^s,
 \qquad
 (\alpha,a_1,\ldots,a_s)
 \longmapsto
 (P_h\alpha,\lambda_1^{-1}a_1,\ldots,\lambda_s^{-1}a_s).
\]
Let $I$ denote the identity operator on this vector space.
Set
\[
 z_j:=
 \bigl([P_h^jR_{h,B}]_{\mathrm{num}},
       \lambda_1^{-j},\ldots,\lambda_s^{-j}\bigr) \in N^1(X)_\Q\oplus\Q^s.
\]
Then $z_{j+1}=\mathcal Pz_j$.
Let $V=\operatorname{Span}_\Q\{z_j \mid j\ge1\}$ and $m=\dim V$.
The vectors $z_1,\ldots,z_m$ form a basis (the empty basis if $m=0$): the first linear dependence makes the span of the preceding vectors $\mathcal P$-invariant, so that span already equals $V$.

Since $(I-\mathcal P)|_V$ is a $\Q$-linear
automorphism, the vector
\[
 t_N:=\mathcal P^N(I-\mathcal P)^{-1}z_1=\left([P_h^NL]_{\mathrm{num}},
 \frac{\lambda_1^{-N}}{\lambda_1-1},\ldots,
 \frac{\lambda_s^{-N}}{\lambda_s-1}\right)
\]
belongs to $V$.
Define $c_{j,N}\in\Q$ by $t_N=\sum_{j=1}^m c_{j,N}z_j$.
Since the operator $\mathcal{P}$ on $V$ has spectral radius less than one,
$t_N\to0$ in $V\otimes_\Q\R$.
Thus, we have
$c_{j,N}\to0$ for every $j$.
Comparing components gives
\[
 \sum_{j=1}^m c_{j,N}P_h^jR_{h,B}
 \equiv P_h^NL,
 \qquad
 \sum_{j=1}^m c_{j,N}\lambda_i^{-j}
 =\frac{\lambda_i^{-N}}{\lambda_i-1}
 \quad(1\leq i\leq s).
\]
Therefore, we have
\begin{align*}
K_X+\mathcal{C}_{h,N}(B;\mathbf{c})
&=K_X+B+\sum^N_{j=1}P^j_hR_{h,B}+\sum_{j=1}^mc_{j,N}P^j_hR_{h,B} \\
&\equiv -L+S_{h,B,N}+P^N_hL \equiv 0,
\end{align*}
as required.
\end{proof}

\section{Currents of dynamical triples}\label{sec:analytic-current}

We now construct the analytic limit of the ramification series and study its non-klt locus.

\subsection{Currents, integrability, and valuations}
\label{subsec:weighted-integrability}

\begin{notation}\label{not:analytic-conventions}
We use the analytic topology on complex analytic spaces.
\begin{itemize}
\item
We normalize $dd^c=(i/\pi)\partial\bar\partial$.
A function on a complex analytic space is \emph{plurisubharmonic}, abbreviated psh,
if, locally under an embedding into a complex manifold, it is the restriction of a psh function on an ambient open neighborhood.
A function is \emph{quasi-plurisubharmonic}, abbreviated quasi-psh,
if it is locally the sum of a psh function and a smooth real function.

\item\label{it:local-psh-potentials}
A closed positive $(1,1)$-current $U$ on the analytification $X^{\mathrm{an}}$ of a normal complex variety $X$
is said to \emph{admit local psh potentials} if every point has a connected analytic open neighborhood $W$ and a psh function $\varphi_W\not\equiv-\infty$ on $W$ such that
\[
 U|_W=dd^c\varphi_W.
\]

\item\label{it:divisor-currents}
Let $X$ be a normal complex variety.
For a $\Q$-Weil divisor $D=\sum_i a_iD_i$, set
\[
 [D]:=\sum_i a_i[D_i],
\]
where $[D_i]$ is the current of integration over the regular locus of $D_i^{\mathrm{an}}$.
Then $[D]$ is a closed real $(1,1)$-current, and
\[
 [D]\geq0\quad\Longleftrightarrow\quad D\geq0.
\]
If $D$ is $\Q$-Cartier, then locally on an analytic open set $W$ one can choose an integer $m>0$ and a nonzero meromorphic function $f$ with $mD|_W=\operatorname{div}(f)$.
The Poincar\'e--Lelong formula gives
\[
 [D]|_W=dd^c\left(\frac1m\log|f|\right).
\]
Here $\log|f|$ is interpreted as a locally integrable function: if $f=g/h$ with $g,h$ holomorphic, then $\log|f|=\log|g|-\log|h|$ almost everywhere.
If $D\geq0$, normality allows $f$ to be chosen holomorphic, so $m^{-1}\log|f|$ is psh.
These assertions follow from the properties of integration currents and the Poincar\'e--Lelong formula on a resolution, followed by proper pushforward; see \cite{DemaillyAnalyticMethods} for the smooth case.

\item
Let $U$ be a closed positive $(1,1)$-current on a
complex manifold $M$.
For a point $y\in M$ and a prime divisor $E$ on $M$,
we write
\[
 \nu(U,y)\in\R_{\geq0},
 \qquad
 \nu_E(U)\in\R_{\geq0},
 \qquad
 \{U\}\in H^2_{\mathrm{dR}}(M,\R)
\]
for its Lelong number at $y$, its generic Lelong
number along $E$, and its de Rham cohomology class,
respectively.

\item
For a complex analytic space $T$ and a smooth closed
real $(1,1)$-form $\omega$ on $T$, we write
\[
 \operatorname{PSH}(T,\omega)
 :=
 \left\{
 \varphi\colon T\to[-\infty,+\infty)
 \ \mid\
 \varphi\text{ is quasi-psh},\
 \omega+dd^c\varphi\geq0
 \right\}.
\]
The inequality is understood in the sense of currents.

\item\label{it:integrability}
On a complex manifold $W$, local integrability is measured with respect to a smooth positive volume form.
We use the usual notation $L^1_{\mathrm{loc}}(W)$ and write $u\in L^1(y)$ if $u$ is integrable on some neighborhood of $y\in W$.
These conditions are independent of the chosen volume form, since any two such forms are locally comparable by positive constants.

\item
For a quasi-psh function $\psi$ on a complex manifold
$M$, we denote by
$\mathcal J_M(\psi)\subseteq\mathcal O_M$
the multiplier ideal sheaf defined by
\[
 \mathcal J_M(\psi)_y
 :=
 \left\{
 g\in\mathcal O_{M,y}
 \ \mid\
 |g|^2e^{-2\psi} \in L^1(y)
 \right\}.
\]
Integrability is measured with respect to a smooth
positive volume form and does not depend on its choice.
For psh functions, coherence of the ideal is \cite{DemaillyAnalyticMethods}*{Proposition 5.7}.
For a quasi-psh function $\psi$, locally write $\psi=\varphi+u$ with $\varphi$ psh and $u$ smooth.
Since $e^{-2u}$ and its reciprocal are locally bounded, $\mathcal J_M(\psi)=\mathcal J_M(\varphi)$ on this neighborhood.
Hence $\mathcal J_M(\psi)$ is coherent as well.


\end{itemize}
\end{notation}

\begin{definition}\label{def:analytic-singularities}
Let $V$ be a normal complex variety and let $B$ be a
$\Q$-Weil divisor such that $K_V+B$ is $\Q$-Cartier.
The divisor $B$ is allowed to have negative coefficients.
Let $U$ be a closed positive $(1,1)$-current on $V^{\mathrm{an}}$ admitting local psh potentials.

For a proper birational morphism $\pi\colon Y\to V$, we use the same notation for its analytification and define pullback locally by
\[
 \pi^*U:=dd^c(\varphi\circ\pi)
 \qquad\text{where }U=dd^c\varphi.
\]

\begin{enumerate}
\item\label{it:current-valuations}
For a divisorial valuation $v=c\ord_E$ of $V$, choose a smooth algebraic birational model $\pi\colon Y\to V$ carrying $E$ and set
\[
 v(U):=c\,\nu_E(\pi^*U).
\]
If $\psi$ is a quasi-psh function on an analytic neighborhood of $c_V(v)$, define similarly
\[
 v(\psi):=c\,\nu_E\bigl(dd^c(\psi\circ\pi)\bigr).
\]
Passing to a common resolution shows that these values are independent of the chosen model.
For an effective $\Q$-Cartier divisor $D$, the Poincar\'e--Lelong formula gives $v([D])=v(D)$.

\item\label{it:analytic-integrability}
Fix a log resolution $\rho\colon Z\to V$ and write
\[
 K_Z+B_Z=\rho^*(K_V+B).
\]
We also write $\rho$ for its analytification.

Choose a coordinate neighborhood $W\subseteq Z^{\mathrm{an}}$
with coordinates $z_1,\ldots,z_n$, where $n=\dim V$, such that
\[
 B_Z|_W=\sum_{i=1}^r b_i\{z_i=0\}.
\]
Write $(\rho^*U)|_W=dd^c\varphi_W$ with $\varphi_W$ psh, and
choose a smooth positive volume form $dV_W$.
Define an ideal sheaf $\mathcal I_\rho\subseteq\mathcal O_{Z^{\mathrm{an}}}$
by
\[
 (\mathcal I_\rho)_y
 :=
 \left\{
 g\in\mathcal O_{Z^{\mathrm{an}},y}
 \ \mid\
 |g|^2e^{-2\varphi_W}
       \prod_{i=1}^r|z_i|^{-2b_i} \in L^1(y)
 \right\}
\]
for $y \in W$.
The integrand is interpreted almost everywhere with respect to $dV_W$.
Thus this definition is independent of the choice of $z_1,\ldots,z_n$ and $\varphi_W$ by \cite{DemaillyAnalyticMethods}*{Section 5.B,
discussion preceding Theorem 5.11}.

To see that $\mathcal I_\rho$ is coherent, choose
integers $M_i\geq\max\{-b_i,0\}$ and put
\begin{equation}\label{eq:weighted-local-functions}
 F:=\prod_{i=1}^r z_i^{M_i},
 \qquad
 \psi:=\varphi_W+
       \sum_{i=1}^r(M_i+b_i)\log|z_i|.
\end{equation}
Then $\psi$ is psh and
\begin{equation}\label{eq:weighted-local-integrand}
 |g|^2e^{-2\varphi_W}\prod_{i=1}^r|z_i|^{-2b_i}
 =|gF|^2e^{-2\psi}
 \qquad\text{almost everywhere}.
\end{equation}
Consequently,
\begin{equation}\label{eq:weighted-ideal-colon}
 \mathcal I_\rho|_W
 =(\mathcal J_W(\psi):F)
 =\ker\left(
   \mathcal O_W
   \xrightarrow{\ g\mapsto Fg\ }
   \mathcal O_W/\mathcal J_W(\psi)
 \right).
\end{equation}
Since $\mathcal J_W(\psi)$ is coherent, so is $\mathcal I_\rho$.
Since $\rho$ is proper, $\rho_*\mathcal I_\rho$ is coherent by Grauert's direct image theorem.

Normality of $V$ gives
\[
 \rho_*\mathcal I_\rho
 \subseteq\rho_*\mathcal O_{Z^{\mathrm{an}}}
 =\mathcal O_{V^{\mathrm{an}}}.
\]
For $x\in V^{\mathrm{an}}$ and
$s\in\mathcal O_{V^{\mathrm{an}},x}$,
membership in $(\rho_*\mathcal I_\rho)_x$
means that there are an analytic neighborhood $O$
of $x$ and a holomorphic representative $s_O$ of $s$
such that, on every coordinate neighborhood
$W\subseteq\rho^{-1}(O)$ as above,
\begin{equation}\label{eq:relative-integrability}
 |s_O\circ\rho|^2e^{-2\varphi_W}
       \prod_{i=1}^r|z_i|^{-2b_i}
 \in L^1_{\mathrm{loc}}(W).
\end{equation}

When $B\geq0$, we define the
\emph{multiplier ideal sheaf} of $(V,B;U)$ by
\[
 \mathcal J(V,B;U):=\rho_*\mathcal I_\rho
 \subseteq\mathcal O_{V^{\mathrm{an}}}.
\]
Independence of the chosen log resolution is proved
in Lemma~\ref{lem:analytic-integrability-independence}.

\item\label{it:analytic-singularities}
We use the following singularity conditions.
\begin{itemize}
\item
The triple $(V,B;U)$ is
\emph{analytically sub-klt at $x$} if $1\in(\rho_*\mathcal I_\rho)_x$.
It is \emph{analytically sub-klt} if this holds at every $x\in V^{\mathrm{an}}$.

\item
The triple $(V,B;U)$ is \emph{analytically sub-lc}
if for every divisorial valuation $v$ of $V$, we have
\[
A_{V,B}(v)-v(U) \geq 0.
\]
\end{itemize}

When $B\geq0$, we omit the prefix \emph{sub-} and set
\[
 \operatorname{Nklt}(V,B;U)
 :=
 \left\{
 x\in V^{\mathrm{an}}
 \ \mid\
 (V,B;U)\text{ is not analytically klt at }x
 \right\}.
\]
\end{enumerate}
\end{definition}

\begin{lemma}\label{lem:analytic-integrability-independence}
Under the assumptions of Definition~\ref{def:analytic-singularities}, the ideal sheaf $\rho_*\mathcal I_\rho$ is independent of the choice of log resolution $\rho$.
\end{lemma}

\begin{proof}
Let $\tau\colon\widetilde Z\to Z$ be a projective birational morphism between smooth varieties such that $\rho\circ\tau$ is again a log resolution.
We also write $\tau$ for its analytification.
We prove
$\tau_*\mathcal I_{\rho\circ\tau}=\mathcal I_\rho$.
Choose a coordinate neighborhood $W\subseteq Z^{\mathrm{an}}$
as in \cref{def:analytic-singularities}
(\ref{it:analytic-integrability}).
Take $F$ and $\psi$ as in \eqref{eq:weighted-local-functions}.
Put $\theta:=dz_1\wedge\cdots\wedge dz_n$.

For $g\in\mathcal O_W(W)$, \eqref{eq:weighted-ideal-colon}
gives the first equivalence below, and the transformation rule
\cite{DemaillyAnalyticMethods}*{Proposition 5.8}
gives the middle equivalence below:
\[
 \begin{aligned}
 g\in\Gamma(W,\mathcal I_\rho)
 &\Longleftrightarrow
 gF\theta\in
 \Gamma\bigl(W,\omega_W\otimes\mathcal J_W(\psi)\bigr)\\
 &\Longleftrightarrow
 \tau^*(gF\theta)\in
 \Gamma\bigl(
   \tau^{-1}(W),
   \omega_{\tau^{-1}(W)}
   \otimes\mathcal J_{\tau^{-1}(W)}(\psi\circ\tau)
 \bigr)\\
 &\Longleftrightarrow
 g\circ\tau\in
 \Gamma\bigl(\tau^{-1}(W),\mathcal I_{\rho\circ\tau}\bigr).
 \end{aligned}
\]
Here $\omega_W$ and $\omega_{\tau^{-1}(W)}$ denote the canonical bundle of $W$ and $\tau^{-1}(W)$, respectively.
The last equivalence follows from the defining integrability condition and
\[
 B_{\widetilde Z}
 =\tau^*B_Z-K_{\widetilde Z/Z}.
\]

Since
$\tau_*\mathcal O_{\widetilde Z^{\mathrm{an}}}
=\mathcal O_{Z^{\mathrm{an}}}$,
these equivalences prove the asserted equality
of ideal sheaves.
The independence of $\rho_*\mathcal I_\rho$
then follows by passing to a common log resolution.
\end{proof}

\begin{lemma}
\label{lem:analytic-klt-discrepancy-gap}
Let $(X,B)$ be a log pair, let $x\in X$ be a closed point, and assume $(X,B)$ is klt near $x$.
Let $U$ be a closed positive $(1,1)$-current on $X^{\mathrm{an}}$ admitting local psh potentials.
Then $(X,B;U)$ is analytically klt at $x$ if and only if there is $0\leq c<1$ such that
\begin{equation}\label{eq:analytic-strict-gap}
 v(U)\leq cA_{X,B}(v)
\end{equation}
for every divisorial valuation $v$ of $X$ whose center contains $x$.
It suffices to require \eqref{eq:analytic-strict-gap} for the divisorial valuations with center exactly $\{x\}$.
\end{lemma}

\begin{proof}
After replacing $X$ by a Zariski neighborhood of $x$,
assume that $(X,B)$ is klt.
Take a log resolution $\rho\colon Y\to X$ and write
\[
 K_Y+B_Y=\rho^*(K_X+B).
\]

For each $y\in\rho^{-1}(x)$, choose a coordinate neighborhood
$W$ with coordinates $z_1,\ldots,z_n$ centered at $y$ such that
\[
 B_Y|_W=\sum_{i=1}^r b_i(E_i\cap W),
 \qquad E_i\cap W=\{z_i=0\},
\]
where the $E_i$ are the components of $\operatorname{Supp}B_Y$ through $y$.
Write $(\rho^*U)|_W=dd^c\varphi_W$ with $\varphi_W$ psh.
After shrinking $W$ and adding a constant to $\varphi_W$, assume
that $|z_i|\leq1$ and $\varphi_W\leq0$.
Take $M_i,F,\psi$ as in \eqref{eq:weighted-local-functions}; then $\psi\leq0$.

Suppose first that $(X,B;U)$ is analytically klt at $x$.
The defining integrability condition and strong openness
\cite{GZ}*{Theorem 1.1} give a number $\delta_y>0$ and
a neighborhood $W_y\subseteq W$ of $y$ such that
\[
 |F|^2e^{-2(1+\delta_y)\psi}
 \in L^1_{\mathrm{loc}}(W_y).
\]
For $0<\delta\leq\delta_y$, we have almost everywhere
\[
 \begin{aligned}
 e^{-2(1+\delta)\varphi_W}
       \prod_{i=1}^r|z_i|^{-2b_i}
 &=
 \left(\prod_{i=1}^r
       |z_i|^{2\delta(M_i+b_i)}\right)
       |F|^2e^{-2(1+\delta)\psi}\\
 &\leq |F|^2e^{-2(1+\delta_y)\psi}.
 \end{aligned}
\]
Here we used $M_i+b_i\geq0$, $|z_i|\leq1$, and $\psi\leq0$.
Choose finitely many such $W_y$ covering the compact fiber
$\rho^{-1}(x)$, and take $\delta>0$ no larger than their
corresponding $\delta_y$.
Properness of $\rho$ gives an analytic neighborhood $O$ of
$x$ whose inverse image lies in this union.
Thus $(X,B;(1+\delta)U)$ is analytically klt on $O$.

Let $E$ be a prime divisor over $X$ whose center contains $x$.
Take a log resolution $\sigma\colon Y'\to X$ carrying $E$.
By \cref{lem:analytic-integrability-independence}, the
defining integrability condition also holds on
$\sigma^{-1}(O)$.
This open set meets $E$.
At a general point of this intersection, write
\[
 E=(t=0),\qquad
 \operatorname{coeff}_E B_{Y'}=1-A_{X,B}(E).
\]
A local potential $\theta$ of $\sigma^*U$ satisfies
\[
 \theta\leq\nu_E(\sigma^*U)\log|t|+O(1).
\]
Consequently, the integrand in the defining condition is
bounded below by a positive constant times
\[
 |t|^{-2\bigl((1+\delta)\nu_E(\sigma^*U)
                    +1-A_{X,B}(E)\bigr)}.
\]
Its integrability implies
\[
 (1+\delta)\nu_E(\sigma^*U)<A_{X,B}(E).
\]
By homogeneity, \eqref{eq:analytic-strict-gap} holds with
$c=(1+\delta)^{-1}$ for every divisorial valuation whose
center contains $x$.

Conversely, suppose that \eqref{eq:analytic-strict-gap}
holds for every divisorial valuation with center $\{x\}$.
Fix $y\in\rho^{-1}(x)$ and use the local notation above.
Let $v$ be a divisorial valuation whose center on $Y$
is $\{y\}$.
Its center on $X$ is $\{x\}$, so
\[
 v(U)\leq cA_{X,B}(v).
\]
Since the $E_i$ have simple normal crossings,
\[
 A_{Y,0}(v)\geq\sum_{i=1}^r v(E_i),
 \qquad
 A_{X,B}(v)\geq\sum_{i=1}^r(1-b_i)v(E_i).
\]
Put
\[
 \kappa:=
 \max\left(
 \{0\}\cup
 \left\{\frac{M_i+b_i}{1-b_i}\ \mid\ 1\leq i\leq r\right\}
 \right).
\]
Then
\[
 \sum_{i=1}^r(M_i+b_i)v(E_i)
 \leq\kappa A_{X,B}(v).
\]
Choose $\eta>0$ such that
$\eta(c+\kappa)\leq1-c$.
Using
\[
 v(\psi)=v(U)+\sum_{i=1}^r(M_i+b_i)v(E_i),
 \qquad
 v(\log|F|)=\sum_{i=1}^r M_i v(E_i),
\]
we obtain
\[
 \begin{aligned}
 &A_{Y,0}(v)+v(\log|F|)-(1+\eta)v(\psi)\\
 &\quad=
 A_{X,B}(v)-v(U)
 -\eta\left(v(U)+\sum_{i=1}^r(M_i+b_i)v(E_i)\right)\\
 &\quad\geq
 \bigl(1-c-\eta(c+\kappa)\bigr)A_{X,B}(v)
 \geq0.
 \end{aligned}
\]

Apply \cite{BermanBoucksomJonsson}*{Theorem B.7}
to the smooth algebraic variety $Y$, with the trivial
$\C^*$-action, the compact Zariski closed subset $\{y\}$,
and the functions $\log|F|$ and $\psi$.
The first function is the logarithm of the absolute value
of a holomorphic function, and the preceding inequality
verifies condition~(iii) of that theorem.
It follows that, on some analytic neighborhood $W_y$ of $y$,
\[
 |F|^2e^{-2\psi}\in L^1_{\mathrm{loc}}(W_y).
\]
By \eqref{eq:weighted-local-integrand} with $g=1$, this is precisely the
integrability condition in
\cref{def:analytic-singularities}.
Finally, finitely many such neighborhoods cover
$\rho^{-1}(x)$, and properness of $\rho$ gives a neighborhood
of $x$ whose inverse image lies in their union.
Hence $(X,B;U)$ is analytically klt at $x$.
\end{proof}

\subsection{The dynamical current and its non-klt locus}\label{subsec:current-construction}
\label{subsec:generic-thresholds}

\begin{theorem}
\label{thm:log-current-construction}
Let $(X,B,h)$ be an int-amplified dynamical triple.
Set $L:=-(K_X+B)$.
Then
\[
 T_{h,B}=\sum_{j\geq1}[P_h^jR_{h,B}]
\]
is a closed positive $(1,1)$-current on $X^{\mathrm{an}}$ admitting local psh potentials.
For every projective resolution $\rho:Y\to X$,
\begin{equation}\label{eq:constructed-current-pullback}
 \begin{gathered}
 \rho^*T_{h,B}=\sum_{j\geq1}[\rho^*P_h^jR_{h,B}],\\
 \{\rho^*T_{h,B}\}=c_1(\rho^*L)
       \in H^2_{\mathrm{dR}}(Y^{\mathrm{an}},\R).
 \end{gathered}
\end{equation}
For every divisorial valuation $v$ of $X$,
\begin{equation}\label{eq:constructed-current-values}
 v(T_{h,B})=\sum_{j\geq1}v(P_h^jR_{h,B})\leq A_{X,B}(v).
\end{equation}
In particular, the pair $(X,B;T_{h,B})$ is analytically lc.
Its multiplier ideal $\mathcal J(X,B;T_{h,B})$ is coherent and $\operatorname{Nklt}(X,B;T_{h,B})$ is an algebraic closed subset.
Finally 
\[
T_{h,B}=P_h([R_{h,B}]+T_{h,B}), \quad T_{h^a,B}=T_{h,B} \qquad (a \in \mathbb{Z}_{\geq 1}).
\]
\end{theorem}
\begin{proof}
If $\dim X=0$, all divisors and currents vanish and the assertions are immediate.
Assume $\dim X>0$.
Choose $\rho(P_h)<\theta<1$ and a norm on $N^1(X)_\R$.
Lemma~\ref{lem:basic-dynamical-properties} makes every $P_h^jR_{h,B}$ $\Q$-Cartier and gives lc of $(X,B+S_{h,B,N})$ for every positive integer $N$.
By \eqref{eq:numerical-telescope} and its proof,
\begin{equation}\label{eq:constructed-current-classes}
 S_{h,B,N}\equiv L-P_h^NL,\qquad \|P_h^N\|=O(\theta^N).
\end{equation}

Fix a smooth projective resolution $\rho:Y\to X$ and a Kähler form $\omega_Y$.
For a positive $(1,1)$-current $U$ on $Y$, write
\[
 \|U\|_{\omega_Y}:=\int_Y U\wedge\omega_Y^{\dim Y-1}.
\]
The currents $[\rho^*S_{h,B,N}]$ are increasing and positive.
The first Chern class of a real divisor on the smooth projective $Y$ factors through $N^1(Y)_\R$.
Thus their cohomology classes converge to $c_1(\rho^*L)$ by \eqref{eq:constructed-current-classes}, so their masses are bounded.
Hence
\[
 T^Y_{h,B}:=\sum_{j\geq1}[\rho^*P_h^jR_{h,B}]
\]
exists as a closed positive current.
Consequently $\{T^Y_{h,B}\}=c_1(\rho^*L)$, and the positive tails satisfy
\begin{equation}\label{eq:resolution-mass-tail}
 \bigl\|T^Y_{h,B}-[\rho^*S_{h,B,N}]\bigr\|_{\omega_Y}
                         =O_\rho(\theta^N).
\end{equation}

Choose an ample Cartier divisor $A$ such that $A+L$
is ample, and Kähler representatives
$\omega_A,\omega_{A+L}$ on $X^{\mathrm{an}}$.
The $dd^c$-lemma \cite{BermanBoucksomGuedjZeriahi}*{Section 1.1, p.~186} on $Y$ gives a quasi-psh function $\varphi_Y$ such that
\[
 T^Y_{h,B}+\rho^*\omega_A
 =\rho^*\omega_{A+L}+dd^c\varphi_Y.
\]
Since the left-hand side is positive,
\[
 \varphi_Y\in
 \operatorname{PSH}(Y^{\mathrm{an}},\rho^*\omega_{A+L}).
\]
By normality of $X$ and
\cite{BermanBoucksomEyssidieuxGuedjZeriahi}*{Section 1.1},
there exists a unique function
\[
 \varphi_X\in
 \operatorname{PSH}(X^{\mathrm{an}},\omega_{A+L})
 \quad\text{such that}\quad
 \varphi_Y=\varphi_X\circ\rho.
\]
Set
\[
 T_{h,B}:=\omega_{A+L}-\omega_A+dd^c\varphi_X.
\]
Then $T_{h,B}$ admits local psh
potentials.

Since $\rho$ has degree one, change of variables gives
\[
 \rho_*\rho^*\omega_A=\omega_A,\qquad
 \rho_*\rho^*\omega_{A+L}=\omega_{A+L},\qquad
 \rho_*\varphi_Y=\varphi_X.
\]
As pushforward commutes with $dd^c$, we obtain
\[
 \rho_*T^Y_{h,B}
 =\omega_{A+L}-\omega_A+dd^c\varphi_X
 =T_{h,B}.
\]
In particular, $T_{h,B}$ is closed and positive.

For every $N\geq1$,
\[
 \rho_*[\rho^*S_{h,B,N}]=[S_{h,B,N}].
\]
Continuity of proper pushforward for weak convergence therefore gives
\[
 \begin{aligned}
 T_{h,B}
 &=\rho_*T^Y_{h,B}\\
 &=\lim_{N\to\infty}\rho_*[\rho^*S_{h,B,N}]\\
 &=\lim_{N\to\infty}[S_{h,B,N}]
 =\sum_{j\geq1}[P_h^jR_{h,B}].
 \end{aligned}
\]
The series on $X$ is independent of $\rho$.
Since $\rho$ was arbitrary, this proves \eqref{eq:constructed-current-pullback} for every projective resolution.

Next, let $v=c\ord_E$ be a divisorial valuation of $X$ and choose a smooth projective resolution $\rho\colon Y\to X$ carrying $E$.
Put
\[
 Q_N:=\rho^*T_{h,B}-[\rho^*S_{h,B,N}]\geq0.
\]
The Siu decomposition \cite{DemaillyAnalyticMethods}*{Decomposition Formula 2.18} and \eqref{eq:resolution-mass-tail} give
\[
 \begin{aligned}
 0\leq v(T_{h,B})-v(S_{h,B,N})
 &=c\,\nu_E(Q_N)\\
 &\leq
 \frac{c\,\|Q_N\|_{\omega_Y}}
      {\int_E\omega_Y^{\dim Y-1}}
 \longrightarrow0.
 \end{aligned}
\]
Since $(X,B+S_{h,B,N})$ is lc, one has $v(S_{h,B,N})\leq A_{X,B}(v)$ for every $N$.
Taking the limit proves \eqref{eq:constructed-current-values} and analytic log canonicity in the sense of \cref{def:analytic-singularities}.

The zero set of $\mathcal{J}(X,B;T_{h,B})$ is algebraic by GAGA \cite{SerreGAGA} and it coincides with $\mathrm{Nklt}(X,B;T_{h,B})$.
Stationarity follows from the recurrence $S_{h,B,N+1}=P_h(R_{h,B}+S_{h,B,N})$ and continuity of proper pushforward.
The iteration assertion follows from \eqref{eq:iterate-blocks}.
\end{proof}

\begin{proposition}\label{prop:lc-valuation-center}
Let $(X,B,h)$ be an int-amplified dynamical triple.
For every nontrivial valuation $v$ of $X$ satisfying
\[
 A_{X,B}(v)=s_{h,B}(v)<+\infty,
\]
one has
\[
 c_X(v)\subseteq\Nklt(X,B;T_{h,B}).
\]
\end{proposition}

\begin{proof}
Put $Z:=\Nklt(X,B;T_{h,B})$ and $D:=c_X(v)$.
Suppose that $D\nsubseteq Z$, and choose a closed point
$x\in D\setminus Z$.
Since $\Nklt(X,B)\subseteq Z$, the pair $(X,B)$ is klt
near $x$, hence also at $\eta_D$.
As $v$ is nontrivial, this gives $A_{X,B}(v)>0$.

By \cref{lem:analytic-klt-discrepancy-gap} and
\eqref{eq:constructed-current-values}, there is $0<c<1$ such that
\[
 s_{h,B}(w)\leq cA_{X,B}(w)
\]
for every divisorial valuation $w$ whose center contains $x$.
In particular, if $w$ is a divisorial valuation of $\mathcal O_{X,\eta_D}$, then $c_X(w)$ contains $D$, hence $x$, and
\[
 w(S_{h,B,N})\leq s_{h,B}(w)\leq cA_{X,B}(w)
 \qquad(N\geq1).
\]
Log discrepancies are unchanged by localization by \eqref{eq:valuation-discrepancy-localization}, and the same holds for divisor values.
Continuity of evaluation on the $\Q$-Cartier divisor
$S_{h,B,N}$ and \eqref{eq:valuation-discrepancy-envelope},
applied on this localization, therefore give
\[
 v(S_{h,B,N})\leq cA_{X,B}(v).
\]
Letting $N\to\infty$, we obtain
\[
 s_{h,B}(v)\leq cA_{X,B}(v)<A_{X,B}(v),
\]
contrary to the assumption.
\end{proof}

\begin{lemma}
\label{lem:uniform-nlc-confinement}
Let $(X,B,h)$ be an int-amplified dynamical triple.
Let $\mathbf c=(c_{j,N})$ be a coefficient family.
Then, for all sufficiently large $N$,
\begin{equation}\label{eq:uniform-nlc-confinement}
 \Nlc\bigl(X,\mathcal C_{h,N}(B;\mathbf c)\bigr)
 \subseteq\Nklt(X,B;T_{h,B}).
\end{equation}
\end{lemma}

\begin{proof}
Put $Z:=\Nklt(X,B;T_{h,B})$.
For each positive integer $\ell$, set
\[
 Z_\ell:=\Nklt(X,B;(1+\ell^{-1})T_{h,B}).
\]
Coherence of multiplier ideals and GAGA show that the
$Z_\ell$ are algebraic closed subsets, and
\[
 Z_1\supseteq Z_2\supseteq\cdots\supseteq Z.
\]

Fix a closed point $x\in X\setminus Z$.
Since $(X,B)$ is klt near $x$, \cref{lem:analytic-klt-discrepancy-gap} gives $0\leq c<1$ such that
\[
 v(T_{h,B})\leq cA_{X,B}(v)
 \qquad\text{whenever }c_X(v)=\{x\}.
\]
Choose $\ell$ with $(1+\ell^{-1})c<1$.
Applying the same lemma to $(1+\ell^{-1})T_{h,B}$ gives $x\notin Z_\ell$.
Thus $\bigcap_{\ell\geq1}Z_\ell=Z$.
Noetherianity implies that $Z_{\ell_0}=Z$ for some
$\ell_0$, so $(X,B;(1+\ell_0^{-1})T_{h,B})$ is
analytically klt on $X\setminus Z$.

Since $c_{j,N}\to0$, for all sufficiently large $N$ we have
$N\geq\max\{1,m\}$ and $-1\leq c_{j,N}\leq\ell_0^{-1}$ for
$1\leq j\leq m$.
By \cref{def:moment-conditions,thm:log-current-construction}, we have
\[
 0\leq \mathcal C_{h,N}(B;\mathbf c)-B\leq(1+\ell_0^{-1})S_{h,B,N},\qquad
 [\mathcal C_{h,N}(B;\mathbf c)-B]\leq(1+\ell_0^{-1})T_{h,B}.
\]
Comparison of singularities shows that $(X,\mathcal C_{h,N}(B;\mathbf c))$ is
klt on $X\setminus Z$, proving
\eqref{eq:uniform-nlc-confinement}.
\end{proof}

\begin{theorem}\label{thm:invariant-nklt}
Let $(X,B,h)$ be an int-amplified dynamical triple, and put $Z=\Nklt(X,B;T_{h,B})$.
After replacing $h$ by an iterate, every lc center of $(X,B)$ and every irreducible component of $Z$ is totally invariant.
The current $T_{h,B}$ and the set $Z$ are unchanged by this replacement.
\end{theorem}

\begin{proof}
The assertion is immediate if $\dim X=0$.
By \cref{prop:log-lc-center-avoidance}, replace $h$ by an iterate so that every lc center of $(X,B)$ is totally invariant.
By \cref{thm:log-current-construction}, iteration leaves the current and its non-klt locus unchanged.
Write $T:=T_{h,B}$, $R:=R_{h,B}$, $A:=A_{X,B}$, and $d:=\deg h$.

The divisor $R$ contains no lc center by \cref{prop:log-lc-center-avoidance}.
Thus \cref{lem:lc-log-degree} gives a rational number $M>0$ such that $(X,B+M^{-1}R)$ is lc.
Consequently,
\[
 u(R)\leq M A(u)
\]
for every divisorial valuation $u$ of $X$.

Fix a closed point $y\notin Z$ and put $x:=h(y)$.
Since $\Nklt(X,B)\subseteq Z$ is totally invariant, $(X,B)$ is klt near both points.
By \cref{lem:analytic-klt-discrepancy-gap}, there is $\epsilon>0$ such that
\begin{equation}\label{eq:invariance-source-gap}
 A(u)-u(T)\geq\epsilon A(u)
\end{equation}
for every divisorial valuation $u$ with $c_X(u)=\{y\}$.

Let $v$ be any divisorial valuation with $c_X(v)=\{x\}$.
By \cref{lem:finite-divisorial}, its extensions $u_1,\ldots,u_q$ through $h$ are divisorial, and one of them, say $u_1$, has center $\{y\}$.
Normalize them by $h_*u_i=v$ and let $n_i$ be their positive integer norm weights, with $\sum_i n_i=d$.
By \eqref{eq:nonflat-stationarity}, \eqref{eq:constructed-current-values}, and
$A(v)=A(u_1)+u_1(R)\leq(1+M)A(u_1)$, we obtain
\begin{equation}\label{eq:analytic-stationary-gap}
 \begin{aligned}
 d\bigl(A(v)-v(T)\bigr)
 &\geq A(u_1)-u_1(T)\\
 &\geq\epsilon A(u_1)\\
 &\geq\frac{\epsilon}{1+M}A(v).
 \end{aligned}
\end{equation}
This is a uniform strict gap for all divisorial valuations centered at $x$, so
\cref{lem:analytic-klt-discrepancy-gap} implies $x\notin Z$.

Thus $h^{-1}(Z)\subseteq Z$.
Since $Z$ is algebraic and closed, the descending chain of its inverse images stabilizes.
Surjectivity of $h$ then gives $h^{-1}(Z)=Z$.
A further iterate makes every irreducible component of $Z$ totally invariant, as in the proof of \cref{prop:log-lc-center-avoidance}.
\end{proof}

\section{Local dynamics and divisorial eigenvaluations}\label{sec:eigenvaluations}

After replacing $h$ by an iterate, we construct a divisorial eigenvaluation centered at each irreducible component of $\operatorname{Nklt}(X,B;T_{h,B})$ not contained in $\operatorname{Nklt}(X,B)$.
In Section~6, these valuations will be realized as prime divisors on equivariant birational models.

\begin{theorem}
\label{thm:generic-eigen-log-pair}
Let $(X,B,h)$ be an int-amplified dynamical triple.
After replacing $h$ by an iterate, for every irreducible component $C$ of $\Nklt(X,B;T_{h,B})$ not contained in $\Nklt(X,B)$, there are a prime divisor $E$ over $X$ with center $C$ and an integer $e>1$ such that, for all integers $j\geq0$ and $M\geq1$,
\begin{equation}\label{eq:global-eigen-series}
\begin{gathered}
 h_*\ord_E=e\ord_E,\qquad s_{h,B}(E)=A_{X,B}(E),\\
 \ord_E(P_h^jR_{h,B})=(e-1)e^{-j}A_{X,B}(E)\quad(j\geq0),\qquad
 A_{X,B+S_{h,B,M}}(E)=e^{-M}A_{X,B}(E).
\end{gathered}
\end{equation}
The only extension of $\ord_E$ through $h$ is $\ord_E/e$.
The integer $e$ may depend on $C$.
\end{theorem}

\subsection{Local setting and valuation estimates}\label{subsec:local-valuation-tools}
\label{subsec:inverse-valuation-dynamics}

\begin{notation}\label{not:local-dynamics}
Throughout this section, fix an int-amplified dynamical triple $(X,B,h)$ such that every lc center of $(X,B)$ and every irreducible component of
\[
 Z:=\Nklt(X,B;T_{h,B})
\]
is totally invariant.
Fix an irreducible component $C$ of $Z$ not contained in $\Nklt(X,B)$.
These assumptions are available after an iterate by \cref{thm:invariant-nklt}.
Put
\[
\begin{gathered}
 T:=T_{h,B},\qquad \Delta:=R_{h,B},\qquad
 S_N:=S_{h,B,N},\qquad s:=s_{h,B},\qquad A:=A_{X,B},\\
 R:=\mathcal O_{X,\eta_C},\qquad \mathfrak m:=\mathfrak m_R,\qquad
 K:=\operatorname{Frac}(R)=K(X),\\
 \kappa:=R/\mathfrak m,\qquad r:=\dim R=\operatorname{codim}_X C>0,
 \qquad \operatorname{Val}_R:=\operatorname{Val}_{\Spec R}.
\end{gathered}
\]
The pair $(\Spec R,B|_{\Spec R})$ is klt.
Valuations in $\operatorname{Val}_R$ are nonnegative on $R$, and their centers on $X$ contain $C$.
Total invariance of $C$ induces a finite injective local map
$\varphi:=h^\#\colon R\to R$, with commutative diagram
\[
\begin{tikzcd}[column sep=large,row sep=small]
 \Spec R \arrow[r,"\Spec\varphi"] \arrow[d]
   & \Spec R \arrow[d] \\
 X \arrow[r,"h"'] & X.
\end{tikzcd}
\]
For $N\geq1$, write $(\varphi^N)_*u:=u\circ\varphi^N$.
An extension of $v$ through $\varphi^N$ is a valuation $u$ with
$(\varphi^N)_*u=v$; if $v$ is centered at $\mathfrak m$, so is every extension.
Set
\begin{equation}\label{eq:local-degrees}
 d:=\deg h,\qquad d_C:=\deg(h|_C),\qquad
 \gamma:=(d/d_C)^{1/r},\qquad F:=\gamma^{-1}\varphi_*.
\end{equation}
We note that $\gamma>1$.
To see this, choose an ample Cartier divisor $H$ on $X$.
Since $(h^N)_*[C]=d_C^N[C]$, the projection formula and \cite{MengBuildingBlocks}*{Lemma~3.8} give
\[
 \left(\frac{d_C}{d}\right)^N(H^{\dim C}\cdot C)
 =\frac{[C]\cdot(h^N)^*(H^{\dim C})}{d^N}
 \longrightarrow0.
\]
As $H^{\dim C}\cdot C>0$, we obtain $d_C<d$, hence $\gamma>1$.
Replacing $h$ by $h^N$ leaves $R,A,s,T$ unchanged and raises
$d,d_C,\gamma$ to their $N$th powers.
\end{notation}

For $0\ne f\in R$, put
$\ord_{\mathfrak m}(f):=\max\{j\in\mathbb Z_{\geq0}\mid f\in\mathfrak m^j\}$.
For an $\mathfrak m$-primary ideal $I$, write
\[
 e_R(I):=\lim_{q\to\infty}
              \frac{r!\,\ell_R(R/I^q)}{q^r},
\]
where $\ell_R$ denotes module length.
For a valuation $v$ centered at $\mathfrak m$, its valuation ideals
and volume are
\[
 \mathfrak a_\lambda(v):=\{f\in R\mid v(f)\geq\lambda\}
 \quad(\lambda>0),\qquad
 \vol_R(v):=\lim_{m\to\infty}
              \frac{e_R(\mathfrak a_m(v))}{m^r}.
\]
The latter limit exists by \cite{Cutkosky}*{Theorem 6.5}.
We omit the subscript $R$ when the local ring is clear.

\begin{lemma}
\label{lem:local-klt-tools}
There is a constant $C_R>0$ such that, for every valuation $v$ of $R$ centered at $\mathfrak m$ with finite log discrepancy,
\begin{equation}\label{eq:singular-izumi}
 v(\mathfrak m)\operatorname{ord}_{\mathfrak m}(f)
 \leq v(f)\leq C_R A(v)\operatorname{ord}_{\mathfrak m}(f)
 \qquad(0\ne f\in R).
\end{equation}
For every $M\geq0$,
\[
\{v \in \operatorname{Val}_R \mid A(v) \leq M \}
\]
is compact.
\end{lemma}

\begin{proof}
The upper bound in \eqref{eq:singular-izumi} is \cite{KimNonClosedVolume}*{Lemma 2.3}, applied at $\eta_C$ on a klt neighborhood in $X$.
The lower bound is immediate from the definition of $\ord_{\mathfrak m}$.

For every $0\ne f\in R$, klt of the localized pair gives
a number $t_f>0$ such that
$t_fv(f)\leq A(v)$ for every divisorial valuation $v$ of $R$.
Continuity of evaluation and
\eqref{eq:valuation-discrepancy-envelope} extend this to
all $v\in\operatorname{Val}_R$.
Thus evaluation gives a topological embedding
\[
\{v \in \operatorname{Val}_R \mid A(v) \leq M \}
\hookrightarrow
\prod_{f \in R\setminus\{0\}}
[0,M/t_f],\qquad v\longmapsto(v(f))_f.
\]
The conditions $v(1)=0$, $v(fg)=v(f)+v(g)$, and
$v(f+g)\geq\min\{v(f),v(g)\}$ whenever $f+g\ne0$
define a closed subset of this product.
Since all coordinates are finite, each such point extends uniquely
to a real valuation of $K$ nonnegative on $R$.
The additional condition $A(v)\leq M$ is closed by the lower
semicontinuity in \eqref{eq:valuation-discrepancy-envelope};
at the trivial valuation this follows from $A\geq0$ and $A(0)=0$.
Hence the image is closed, and Tychonoff's theorem proves compactness.

\end{proof}

\begin{lemma}\label{lem:compact-lc-valuations}
One has
\[
 \varphi_*^{-1}\bigl(\{v\in\operatorname{Val}_R\mid
                          A(v)=s(v)<\infty\}\bigr)
 \subseteq\{v\in\operatorname{Val}_R\mid
                          A(v)=s(v)<\infty\}.
\]

Put
\[
 \begin{aligned}
 K_A&:=\{v\in\operatorname{Val}_R\mid A(v)\leq1,\ s(v)=1\},\\
 \mathscr K&:=\{v\in\operatorname{Val}_R\mid
             v(\mathfrak m)=1,\ A(v)=s(v)<\infty\}.
 \end{aligned}
\]
These sets are nonempty and compact.
Every member of $K_A$ is centered at $\mathfrak m$, and
\[
 \mathscr K=\{\frac{v}{v(\mathfrak m)}\mid v\in K_A\}.
\]
There is $M\geq1$ such that
\begin{equation}\label{eq:izumi}
 \operatorname{ord}_{\mathfrak m}(f)\leq v(f)
 \leq M\operatorname{ord}_{\mathfrak m}(f)
 \qquad(0\ne f\in R,v \in \mathscr{K}).
\end{equation}
Discrepancy is bounded above on $\mathscr K$.
Moreover, for every valuation $u$ of $R$,
\begin{equation}\label{eq:inverse}
 \varphi_*u\in\mathscr K
 \quad\Longrightarrow\quad
 \frac{u}{u(\mathfrak m)}\in\mathscr K.
\end{equation}
\end{lemma}

\begin{proof}
Let $v\in\operatorname{Val}_R$ with $A(v)<\infty$, and let
$u_1,\ldots,u_q$ be all its extensions through $\varphi$.
Their discrepancies are finite by \cref{lem:finite-divisorial}.
Applying \eqref{eq:nonflat-stationarity} and the nonnegativity of its summands,
we obtain $A(u_i)=s(u_i)$ for every $i$ whenever $A(v)=s(v)$.

We now study $K_A$ and $\mathscr{K}$.
By \cref{lem:local-klt-tools}, the set
\[
 L:=\{v\in\operatorname{Val}_R\mid A(v)\leq1\}
\]
is compact.
The continuous functions $v\mapsto v(S_N)$ converge uniformly to $s$ on $L$ by \cref{thm:lc-log-tail}.
Thus $s$ is continuous on $L$ and $\max_L s\leq1$.
Suppose that $\max_L s<1$, and choose $0<a_0<1$ with $\max_L s<a_0$.
Homogeneity gives $s(w)\leq a_0A(w)$ for every valuation $w$ of $R$.
Choose a rational number $\lambda>1$ with $\lambda a_0<1$.
For every $N$, the inequality $\lambda w(S_N)\leq\lambda a_0A(w)$ shows that $(X,B+\lambda S_N)$ is klt at $\eta_C$.
Choose $\rho(P_h)<\theta<1$.
By \cref{thm:lc-log-tail} and \eqref{eq:constructed-current-values}, there is $C_0>0$ such that
\[
 0\leq w(T)-w(S_N)\leq C_0\theta^N A(w)
\]
for every divisorial valuation $w$ of $X$.
Choose $N$ with $\lambda^{-1}+C_0\theta^N<1$, and a Zariski open neighborhood $U'$ of $\eta_C$ on which $(X,B+\lambda S_N)$ is klt.
For every divisorial valuation $w$ whose center meets $U'$, one has
\[
 w(T)\leq w(S_N)+C_0\theta^N A(w)
       \leq(\lambda^{-1}+C_0\theta^N)A(w).
\]
By \cref{lem:analytic-klt-discrepancy-gap}, $(X,B;T)$ is analytically klt on $U'$, contradicting $C\subseteq Z$.
Consequently $\max_L s=1$, and $K_A=s^{-1}(1)\cap L$ is nonempty and compact.
Every $v\in K_A$ satisfies $A(v)=s(v)=1$ by \cref{thm:lc-log-tail}.

For $v\in K_A$, \cref{prop:lc-valuation-center}
gives
\[
 C\subseteq c_X(v)\subseteq Z.
\]
As $C$ is an irreducible component of $Z$, it follows
that $c_X(v)=C$.
Thus every $v\in K_A$ is centered at $\mathfrak m$.

Evaluation on $\mathfrak m$ is continuous, so $\delta:=\min_{K_A}v(\mathfrak m)>0$.
By homogeneity, we have the surjection
\[
K_A \twoheadrightarrow \mathscr K, \qquad v\mapsto \frac{v}{v(\mathfrak m)}.
\]
This map is continuous, proving compactness of $\mathscr K$.
\eqref{eq:singular-izumi} gives \eqref{eq:izumi} with $M:=\max\{1,C_R/\delta\}$, and $A\leq\delta^{-1}$ on $\mathscr K$.
Finally, if $\varphi_*u\in\mathscr K$, then \eqref{eq:nonflat-stationarity} gives $0<A(u)=s(u)<\infty$ and $u$ is centered at $\mathfrak m$.
Thus, $u/u(\mathfrak m) \in \mathscr{K}$, which proves \eqref{eq:inverse}.
\end{proof}

\subsection{Construction of an invariant filtration}\label{sec:local-eigen}
\label{subsec:invariant-valuations}\label{subsec:invariant-filtration}

\begin{definition}\label{def:filtrations}
An \emph{$\mathfrak m$-filtration} $\mathfrak b_\bullet$
is a family of $\mathfrak m$-primary ideals indexed by
$\lambda>0$ such that
\[
 \mathfrak b_\lambda\subseteq\mathfrak b_\mu\quad(\lambda>\mu>0),
 \qquad
 \mathfrak b_\lambda\mathfrak b_\mu
       \subseteq\mathfrak b_{\lambda+\mu}\quad(\lambda,\mu>0),
\]
and $\mathfrak b_\lambda=\mathfrak b_{\lambda-\epsilon}$
for all sufficiently small $\epsilon>0$.
Set $\mathfrak b_\lambda=R$ for $\lambda\leq0$.
It is \emph{linearly bounded} if
$\mathfrak b_\lambda\subseteq\mathfrak m^{\lceil c\lambda\rceil}$
for some $c>0$ and all $\lambda>0$;
see \cite{BLQ}*{Definitions 2.1 and 2.2}.
\end{definition}

For a nonzero ideal $I\subseteq R$, set
\[
 \lct_R(I):=\inf_{\substack{v\in\operatorname{Val}_R\\v(I)>0}}
                         \frac{A(v)}{v(I)},
\]
with $\inf\varnothing=+\infty$.
For an $\mathfrak m$-filtration, define
\begin{equation}\label{eq:filtration-invariants}
\begin{gathered}
 v(\mathfrak b_\bullet):=\inf_{m\geq1}\frac{v(\mathfrak b_m)}m,
 \qquad e_R(\mathfrak b_\bullet):=\lim_{m\to\infty}
                                      \frac{e_R(\mathfrak b_m)}{m^r},\\
 \lct_R(\mathfrak b_\bullet):=
                         \sup_{m\geq1}m\lct_R(\mathfrak b_m).
\end{gathered}
\end{equation}
The multiplicity limit exists by \cite{Cutkosky}*{Theorem 6.5}.
The valuation ideals form an $\mathfrak m$-filtration, with
$e_R(\mathfrak a_\bullet(v))=\vol_R(v)$.
We will use
\[
 v(\mathfrak b_\bullet)=
               \inf_{\lambda>0}\frac{v(\mathfrak b_\lambda)}\lambda,
 \qquad
 \lct_R(\mathfrak b_\bullet)=
    \inf_{\substack{v\in\operatorname{Val}_R\\v(\mathfrak b_\bullet)>0}}
                       \frac{A(v)}{v(\mathfrak b_\bullet)}.
\]
The first equality follows from
$\mathfrak b_\lambda^q\subseteq\mathfrak b_{\lfloor q\lambda\rfloor}$
by letting $q\to\infty$; the second is the valuative threshold
formula in \cite{Blum}*{Sections 2.9 and 7}.
Choosing $c_0$ with $\mathfrak m^{c_0}\subseteq\mathfrak b_1$
also gives
$\mathfrak m^{c_0m}\subseteq\mathfrak b_1^m\subseteq\mathfrak b_m$.

For $N\geq1$, write
\[
 R_N:=\varphi^N(R)\subseteq R,\qquad
 K_N:=\varphi^N(K)\subseteq K.
\]
The ring $R$ is the integral closure of $R_N$ in $K$,
has rank $d^N$ over $R_N$, and has residue-field degree $d_C^N$.
The finite-map multiplicity formula is
\begin{equation}\label{eq:nonflat-filtration}
 e_R(\varphi^N(I)R)=\gamma^{rN}e_R(I),\qquad
 e_R(\varphi^N(\mathfrak b_\bullet)R)=\gamma^{rN}e_R(\mathfrak b_\bullet)
 \quad(N\geq1),
\end{equation}
where $I$ is $\mathfrak m$-primary and $\mathfrak b_\bullet$ is an $\mathfrak m$-filtration.
Indeed, $\ell_{R_N}(M)=d_C^N\ell_R(M)$ for every finite-length $R$-module $M$.
The associativity formula \cite{StacksProject}*{Tag 0AZX} gives
\[
 d_C^N e_R(\varphi^N(I)R)
 =e_{R_N}(\varphi^N(I),R)
 =d^N e_R(I).
\]
This proves the ideal formula, and \eqref{eq:filtration-invariants} gives the filtration formula by taking limits.
Here the extended filtration has terms $\varphi^N(\mathfrak b_\lambda)R$; locality and stabilization from the left preserve its defining properties.

\begin{proposition}\label{prop:nonflat-attraction}
Let $M$ be as in \cref{lem:compact-lc-valuations}.
For $N\geq0$, put
\[
 I_N:=\varphi^N(\mathfrak m)R,\qquad
 c_N:=\max\{j\in\mathbb Z_{\geq0}\mid I_N\subseteq\mathfrak m^j\}=\ord_{\mathfrak m}(I_N).
\]
Then
\begin{equation}\label{eq:nonflat-attraction}
 M^{-2}\gamma^N\le c_N\le\gamma^N\qquad(N\ge1).
\end{equation}
\end{proposition}

\begin{proof}
Fix an integer $N\geq1$, a valuation $v\in\mathscr K$, and all its extensions $u_1,\ldots,u_q$ through $\varphi^N$.
Put $a_i:=u_i(\mathfrak m)$.
By \cref{lem:compact-lc-valuations}, the valuations $u_i/a_i$ lie in $\mathscr K$.
Since $u_i(I_N)=v(\mathfrak m)=1$, \eqref{eq:izumi} gives 
\[
a_ic_N\leq1\leq Ma_ic_N.
\]
Let $\ell\geq1$ be an integer and take $0\ne f\in\mathfrak m^{\lceil\ell M^2c_N\rceil}$.
Then
\[
u_i(f)\geq \lceil\ell M^2c_N\rceil a_i \geq \ell M
\]
for every $i$.
The characteristic polynomial of multiplication by $f$ on the $K_N$-vector space $K$ gives
\[
 f^{d^N}+\varphi^N(b_1)f^{d^N-1}+\cdots+\varphi^N(b_{d^N})=0,
 \qquad b_j\in R.
\]
Every conjugate of $f$ has value at least $\ell M$ under an extension of $u_i$ to a normal closure, since it is measured by one of the extensions $u_j$.
The symmetric expressions for the coefficients and \eqref{eq:izumi} give
\[
 j\ell M\leq v(b_j)
 \leq M\ord_{\mathfrak m}(b_j)
 \qquad(b_j\ne0).
\]
Thus $\varphi^N(b_j)\in\varphi^N(\mathfrak m^{j\ell})R=I_N^{j\ell}$, so $f$ is integral over $I_N^\ell$.
Consequently
\[
 \mathfrak m^{\lceil\ell M^2c_N\rceil}\subseteq\overline{I_N^\ell}.
\]
Multiplicity is unchanged by integral closure, so \eqref{eq:nonflat-filtration} gives
\[
 \ell^r\gamma^{rN}e(\mathfrak m)
 =e(I_N^\ell)
 \leq\lceil\ell M^2c_N\rceil^re(\mathfrak m).
\]
Divide by $\ell^r$ and let $\ell\to\infty$ to obtain the lower bound in \eqref{eq:nonflat-attraction}.
The upper bound follows from $I_N\subseteq\mathfrak m^{c_N}$ and
\[
 \gamma^{rN}e(\mathfrak m)=e(I_N)\geq c_N^re(\mathfrak m).
\]

\end{proof}

\begin{lemma}\label{lem:nonflat-volume-equality}
For every valuation $v$ centered at $\mathfrak m$ with finite log discrepancy and every integer $N\geq1$,
\begin{equation}\label{eq:volume-monotonicity}
 \vol_R(F^N(v))\geq\vol_R(v).
\end{equation}
If equality holds, then $(F^N)^{-1}(F^N(v))=\{v\}$.
\end{lemma}
\begin{proof}
For every valuation $z$ centered at $\mathfrak m$ with finite log discrepancy,
\eqref{eq:singular-izumi} gives
\[
 \mathfrak m^{\lceil m/z(\mathfrak m)\rceil}
 \subseteq\mathfrak a_m(z)
 \subseteq\mathfrak m^{\lceil m/(C_RA(z))\rceil}.
\]
Taking multiplicities and letting $m\to\infty$, we obtain
\begin{equation}\label{eq:singular-volume-bounds}
 \frac{e_R(\mathfrak m)}{(C_RA(z))^r}
 \leq\vol_R(z)
 \leq\frac{e_R(\mathfrak m)}{z(\mathfrak m)^r}.
\end{equation}

Put $w:=F^N(v)$.
For $t>0$ we have
\[
 \mathfrak b_t:=\varphi^N(\mathfrak a_t(w))R
 \subseteq
 \mathfrak c_t:=\mathfrak a_{t\gamma^N}(v).
\]
Thus, we have
\[
\gamma^{rN}\vol_R(w)=e_R(\mathfrak b_{\bullet}) \geq e_R(\mathfrak c_{\bullet})=\gamma^{rN}\vol_R(v)
\]
by \eqref{eq:nonflat-filtration}, proving the inequality.

Suppose equality holds, and let $u$ be any extension of $w$ through $\varphi^N$.
It has finite log discrepancy by \cref{lem:finite-divisorial} and positive volume by \eqref{eq:singular-volume-bounds}.
Since $e_R(\mathfrak b_{\bullet}) = e_R(\mathfrak c_{\bullet})$, we have
\[
 1=w(\mathfrak a_\bullet(w))=u(\mathfrak b_\bullet)=u(\mathfrak c_\bullet)
 =\gamma^N u(\mathfrak a_\bullet(v)),
\]
where the third equality follows from \cite{BLQ}*{Proposition 3.12}.
In particular, we have
\[
\gamma^Nu(\mathfrak a_m(v)) \geq m \qquad (m > 0).
\]
Thus, for $0\ne f\in\mathfrak m$, it follows that $u(f)\geq\gamma^{-N}v(f)$.
The valuations $u$ and $v/\gamma^N$ agree on $K_N$.
For $0\ne f\in R$, put $a_f:=\Nm_{K/K_N}(f)$.
Then $a_f\in R_N$ and $a_f/f\in R$.
Since the valuations $u$ and $v/\gamma^N$ have the same restriction to $K_N$, we have $u(a_f)=\gamma^{-N}v(a_f)$.
Therefore, we have
\[
0 \leq u(f)-\gamma^{-N}v(f)=\gamma^{-N}v(a_f/f)-u(a_f/f) \leq 0.
\]
Thus, $u(f)=\gamma^{-N}v(f)$ for every $f \in R$, as desired.
\end{proof}

\begin{proposition}\label{prop:invariant-average}
There is a linearly bounded $\mathfrak m$-filtration $\mathfrak c_\bullet$ such that
\[
 \lct_R(\mathfrak c_\bullet)=1,\qquad
 u(\mathfrak c_\bullet)\leq s(u)\leq A(u)
\]
for every valuation $u$ centered at $\mathfrak m$ with finite discrepancy.
Furthermore, if we put
\begin{equation}\label{eq:average-computers}
 \mathscr V:=\{u\in\operatorname{Val}_R\mid
                   A(u)=u(\mathfrak c_\bullet)=1\},
\end{equation}
then $\mathscr V\subseteq K_A$ and $F^{-1}(\mathscr V) \subseteq \mathscr V$.
\end{proposition}

\begin{proof}
On valuations with $v(\mathfrak m)=1$, define
\[
 \Psi(v):=\frac{\varphi_*v}{v(I_1)}=\frac{\varphi_*v}{\varphi_*v(\mathfrak m)},\qquad
 \mathscr K_j:=\{v\in\mathscr K\mid
                   \Psi^i(v)\in\mathscr K\ (0\leq i\leq j)\}.
\]
These sets are compact.
They are nonempty by induction: for $w\in\mathscr K_j$, choose $u \in \operatorname{Val}_R$ with $\varphi_*u=w$ and set $v:=u/u(\mathfrak m)$.
By \eqref{eq:inverse}, $v\in\mathscr K$, and $u(I_1)=w(\mathfrak m)=1$ gives $\Psi(v)=w$.
Thus $v\in\mathscr K_{j+1}$.
Choose $v_0\in\bigcap_j\mathscr K_j$ and put $v_N=\Psi^N(v_0)$.
Then $F^N(v_0)=t_Nv_N$, where 
\begin{equation}\label{eq:bounded-orbit}
 t_N:=\gamma^{-N}v_0(I_N)\in[M^{-2},M].
\end{equation}
The inclusion follows from \eqref{eq:nonflat-attraction}.
Then the orbit closure
\[
K_0:=\overline{\{v_0,F(v_0),F^2(v_0),\ldots\}}
\]
is therefore compact.
We note that $A=s$ on $K_0$ since $v_N \in \mathscr K$ and $F^N(v_0)=t_Nv_N$.
To choose a maximum-volume locus, we verify upper semicontinuity directly.
For every $\mathfrak m$-primary ideal $I$ and every valuation $v$ centered at $\mathfrak m$,
the inclusions $I^q\subseteq\mathfrak a_{qv(I)}(v)$ give
$\vol_R(v)\leq e_R(I)/v(I)^r$.
Taking $I=\mathfrak a_m(v)$ and letting $m\to\infty$ yields
\[
 \vol_R(v)=\inf_{I\text{ $\mathfrak m$-primary}}
                   \frac{e_R(I)}{v(I)^r}.
\]
Each $v\mapsto v(I)$ is continuous and positive on valuations centered at $\mathfrak m$.
Thus volume is upper semicontinuous, and it attains a positive maximum on the compact set $K_0$ by \eqref{eq:singular-volume-bounds}.
Its maximum locus $\Omega$ is compact and satisfies $F(\Omega)\subseteq\Omega$ by \cref{lem:nonflat-volume-equality}.
Choose $v_*\in\Omega$ and set
\[
 \mu_N:=\frac1N\sum_{j=0}^{N-1}\delta_{F^j(v_*)}
 \qquad(N\geq1),
\]
where $\delta_v$ denotes the Dirac measure at $v$.
Each $\mu_N$ is a regular Borel probability measure, being a finite convex combination of Dirac measures.
Since $\Omega$ is compact Hausdorff, the Banach--Alaoglu theorem and the Riesz representation theorem give a subnet of $(\mu_N)$ converging weak-* to a regular Borel probability measure $\mu$.

For every $g\in C(\Omega,\mathbb R)$, we have
\[
 \int_\Omega(g\circ F-g)\,d\mu_N
 =\frac{g(F^N(v_*))-g(v_*)}{N}
 \longrightarrow0,
\]
because $g$ is bounded.
Since $F$ is continuous, passing to the limit along the subnet gives $\int_\Omega g\circ F\,d\mu=\int_\Omega g\,d\mu$.
Thus $F_*\mu=\mu$.

Put $a:=\int_\Omega A(v)\,d\mu(v)$, which is finite and positive.
Let $S(v):=a^{-1}v$, and replace $(\Omega,\mu)$ by $(S(\Omega),S_*\mu)$.
Since $F\circ S=S\circ F$, the new measure remains an $F$-invariant probability measure.
By the homogeneity of $A$, it satisfies
\[
 \int_\Omega A(v)\,d\mu(v)=1,
\]
where we retain the notation $\Omega,\mu$ for the rescaled objects.
Define $\mathcal L(f):=\int_\Omega v(f)\,d\mu(v)$, and similarly define $\mathcal L(D)$ for a $\Q$-Cartier divisor $D$.
All these evaluations are continuous on $\Omega$; continuity of $A=s$ follows from the uniform tail estimate.
Invariance of $\mu$ gives
\[
 \mathcal L(fg)=\mathcal L(f)+\mathcal L(g),\qquad
 \mathcal L(\varphi(f))=\gamma\mathcal L(f).
\]
For $v\in\Omega$, equality of the volumes of $v$ and $F(v)$ implies that the only extension of $F(v)$ is $v/\gamma$.
The norm and discrepancy formulas therefore give
\[
 F(v)(P_hD)=\gamma^{-1}v(D),\qquad
 \gamma A(F(v))=A(v)+v(\Delta) \qquad (v \in \Omega).
\]
Integrating and summing yield
\begin{equation}\label{eq:average-identities}
 \mathcal L(P_hD)=\gamma^{-1}\mathcal L(D),\qquad
 \mathcal L(\Delta)=\gamma-1,\qquad
 \mathcal L(S_N)=1-\gamma^{-N}.
\end{equation}

Put
\begin{equation}\label{eq:average-filtration}
 \mathfrak c^0_\lambda
   :=(f\mid 0\ne f\in R,\ \mathcal L(f)\geq\lambda),\qquad
 \mathfrak c_\lambda
   :=\bigcap_{0<\epsilon<\lambda}\mathfrak c^0_{\lambda-\epsilon}.
\end{equation}
\eqref{eq:singular-izumi} gives
\[
 \delta\ord_{\mathfrak m}(f)\leq\mathcal L(f)
                  \leq C_R\ord_{\mathfrak m}(f)
\]
for $\delta:=\int_{\Omega} v(\mathfrak m)d\mu(v)$.
Consequently
\[
 \mathfrak m^{\lceil\lambda/\delta\rceil}
 \subseteq\mathfrak c_\lambda
 \subseteq\mathfrak m^{\lceil\lambda/C_R\rceil}.
\]
Finite length makes the intersection in \eqref{eq:average-filtration} stationary for sufficiently small $\epsilon$.
Multiplicativity follows from that of $\mathcal L$, so $\mathfrak c_\bullet$ is a linearly bounded $\mathfrak m$-filtration.
For $0<\eta<1$,
\[
 \mathfrak c^0_\lambda\subseteq\mathfrak c_\lambda
                     \subseteq\mathfrak c^0_{(1-\eta)\lambda}.
\]
Evaluating generators and letting $\eta\to0$ gives
\begin{equation}\label{eq:average-ratio}
 u(\mathfrak c_\bullet)=\inf_{0\ne f\in\mathfrak m}
                            \frac{u(f)}{\mathcal L(f)} \qquad (u \in \operatorname{Val}_R).
\end{equation}
Moreover, $\varphi(\mathfrak c_\lambda)R\subseteq\mathfrak c_{\gamma\lambda}$.
By \eqref{eq:nonflat-filtration},
\[
 e_R\bigl(\varphi(\mathfrak c_\bullet)R\bigr)
 =\gamma^r e_R(\mathfrak c_\bullet)
 =e_R\bigl((\mathfrak c_{\gamma\lambda})_{\lambda>0}\bigr).
\]
For every valuation $u$ centered at $\mathfrak m$ with finite log discrepancy,
\eqref{eq:singular-volume-bounds} gives positive volume.
Hence \cite{BLQ}*{Proposition~3.12} gives
\begin{equation}\label{eq:average-eigen-order}
 (\varphi_*u)(\mathfrak c_\bullet)=\gamma u(\mathfrak c_\bullet).
\end{equation}
Applying \eqref{eq:average-ratio} to a local equation of a Cartier multiple of $S_N$ and \eqref{eq:average-identities} give
\begin{equation}\label{eq:ineq-csA}
 u(\mathfrak c_\bullet)
 \leq\frac{u(S_N)}{1-\gamma^{-N}}
 \longrightarrow s(u)\leq  A(u).    
\end{equation}

Thus $\lct_R(\mathfrak c_\bullet)\geq1$.

For the reverse inequality, fix $m\geq1$, $0<\epsilon<1$, and $t=(1-\epsilon)m$.
Choose generators $f_1,\ldots,f_\ell$ of $\mathfrak c_t^0$ from its defining generating set.
Consider
\[
G \colon \Omega \to \mathbb{R}^{\ell+1}, \quad v \mapsto (v(f_1),\ldots,v(f_{\ell}),A(v)).
\]
Carath\'eodory's theorem applied to $G(\Omega)$ and
\[
\int_{\Omega} G(v)d\mu(v)=(\mathcal L(f_1),\ldots, \mathcal L(f_{\ell}),1) \in \operatorname{conv}(G(\Omega))
\]
gives $u_1,\ldots,u_q\in\Omega$ and weights $p_1,\ldots,p_q >0$ with
\[
 \sum_{i=1}^q p_i=1,\qquad
 \sum_{i=1}^q p_iu_i(f_j)=\mathcal L(f_j),\qquad
 \sum_{i=1}^q p_iA(u_i)=1.
\]
Put $P_i:=\sum_{k=1}^i p_k$.
For $\lambda>0$, define recursively
\[
 \begin{aligned}
 \mathfrak g^{(1)}_\lambda
 &:=\mathfrak a_\lambda(u_1),\\
 \mathfrak g^{(i)}_\lambda
 &:=\sum_{\substack{\alpha,\beta\in\mathbb R\\
          P_{i-1}\alpha+p_i\beta=P_i\lambda}}
       \left(\mathfrak g^{(i-1)}_\alpha
             \cap\mathfrak a_\beta(u_i)\right)
       \qquad(2\leq i\leq q).
 \end{aligned}
\]
The sum denotes the sum of ideals, and all filtration terms with nonpositive index are understood to be $R$.
Set $\mathfrak g_\bullet:=\mathfrak g^{(q)}_\bullet$.
The initial filtration $\mathfrak g^{(1)}_\bullet=\mathfrak a_\bullet(u_1)$ is an $\mathfrak m$-filtration.
For $2\leq i\leq q$, apply the geodesic construction of \cite{BLQ}*{Definition~4.1} to $\mathfrak g^{(i-1)}_\bullet$ and $\mathfrak a_\bullet(u_i)$ at parameter $p_i/P_i\in(0,1)$.
This gives precisely $\mathfrak g^{(i)}_\bullet$, which is therefore an $\mathfrak m$-filtration by induction.

For every $0\ne f\in\mathfrak m$, induction gives
\[
 f\in
 \mathfrak g^{(i)}_{\frac1{P_i}\sum_{k=1}^i p_ku_k(f)}
 \qquad(1\leq i\leq q).
\]
Indeed, the first step is the definition of the valuation ideals.
For the induction step, taking
\[
 \alpha:=\frac1{P_{i-1}}\sum_{k=1}^{i-1}p_ku_k(f) \qquad \beta:=u_i(f)
\]
gives the desired inclusion.
Taking $i=q$ and $f=f_j$, we obtain
$f_j\in\mathfrak g_{\mathcal L(f_j)}\subseteq\mathfrak g_t$.
Hence $\mathfrak c_t^0\subseteq\mathfrak g_t$.

Using $\lct_R(\mathfrak a_\bullet(u_i))\leq A(u_i)$, repeated application of \cite{BLQ}*{Proposition~5.1} gives
\[
 \begin{aligned}
 \lct_R(\mathfrak g^{(i)}_\bullet)
 &\leq
 \frac{P_{i-1}}{P_i}
       \lct_R(\mathfrak g^{(i-1)}_\bullet)
 +\frac{p_i}{P_i}
       \lct_R(\mathfrak a_\bullet(u_i))\\
 &\leq\frac1{P_i}\sum_{k=1}^i p_kA(u_k).
 \end{aligned}
\]
In particular, $\lct_R(\mathfrak g_\bullet)\leq\sum_{i=1}^q p_iA(u_i)=1$.

Since $\mathfrak c_m\subseteq\mathfrak c_t^0\subseteq\mathfrak g_t$,
\[
 m\lct_R(\mathfrak c_m)
 \leq\frac m t\lct_R(\mathfrak g_\bullet)
 \leq\frac1{1-\epsilon}.
\]
Let $\epsilon\to0$ and use \eqref{eq:filtration-invariants} to obtain $\lct_R(\mathfrak c_\bullet)=1$.

Next, we study $\mathscr V$.
Every $w\in\mathscr V$ is centered at $\mathfrak m$, since otherwise it vanishes on every $\mathfrak m$-primary ideal.
\eqref{eq:ineq-csA} gives $s(w)=A(w)=1$, hence $\mathscr V\subseteq K_A$.

If $\varphi_*u=w\in\mathscr V$, then \cref{lem:finite-divisorial} gives finite discrepancy of $u$ and \eqref{eq:average-eigen-order} gives $u(\mathfrak c_\bullet)=\gamma^{-1}$.
By \eqref{eq:average-ratio}, $\gamma u(f)\geq\mathcal L(f)$ for every $0\ne f\in R$.
Applying this to a local equation of a Cartier multiple of $\Delta$ yields $u(\Delta)\geq 1-\gamma^{-1}$ by \eqref{eq:average-identities}.
Therefore
\[
 \gamma^{-1}=u(\mathfrak c_{\bullet})\leq A(u)=A(w)-u(\Delta)
                         \leq1-(1-\gamma^{-1})=\gamma^{-1},
\]
where the first inequality follows from \eqref{eq:ineq-csA}.
Thus $\gamma u\in\mathscr V$.
\end{proof}

\subsection{Volume minimization and divisorial eigenvaluations}\label{subsec:volume-minimization}
\label{subsec:divisorial-eigenvaluations}

For the filtration $\mathfrak c_\bullet$ constructed above, we minimize volume on
\[
  \mathscr V
  =\{u\in\operatorname{Val}_R
      \mid A(u)=u(\mathfrak c_\bullet)=1\}.
\]
We prove that the minimizer is unique and quasi-monomial.
Together with $F^{-1}(\mathscr V)\subseteq\mathscr V$
and \cref{lem:nonflat-volume-equality}, uniqueness will imply that the minimizer
is fixed by $F$.

\begin{proposition}\label{prop:closed-point-volume-minimum}
Let $k_0\supseteq\C$ be an algebraically closed field, and let $(Y,D)$ be a klt pair of finite type over $k_0$, with $D$ an effective $\Q$-divisor.
Let $y\in Y$ be a closed point, and put $S:=\mathcal O_{Y,y}$ and $\mathfrak n:=\mathfrak m_yS$.
Let $\mathfrak b_\bullet$ be a linearly bounded $\mathfrak n$-filtration of $S$ such that $\lct_{Y,D}(\mathfrak b_\bullet)=1$.
Set
\[
 \mathscr W:=\left\{u\in\operatorname{Val}_{Y,y}\ \mid\
 A_{Y,D}(u)=u(\mathfrak b_\bullet)=1\right\},
\]
where $\operatorname{Val}_{Y,y}$ denotes the real valuations centered at $y$.
Then $\mathscr W$ is nonempty, and $\vol_S$ has a unique minimizer on $\mathscr W$.
This minimizer is quasi-monomial.
\end{proposition}

\begin{proof}
Put $r:=\dim S$ and $A:=A_{Y,D}$.
All thresholds below are taken with respect to $(Y,D)$, and all multiplicities and volumes are computed over $S$.

\emph{Existence.}
By \cite{Xu}*{Theorem~1.1}, there is a quasi-monomial valuation computing $\lct(\mathfrak b_\bullet)$.
Since the ideals $\mathfrak b_m$ are $\mathfrak n$-primary, this valuation is centered at $y$.
After rescaling it so that its value on $\mathfrak b_\bullet$ is $1$, it belongs to $\mathscr W$.
Thus $\mathscr W\ne\varnothing$.

Choose an integer $c_0>0$ such that $\mathfrak n^{c_0}\subseteq\mathfrak b_1$.
Multiplicativity gives $\mathfrak n^{c_0m}\subseteq\mathfrak b_m$ for every integer $m\geq1$.
For every $u\in\mathscr W$, the inequality $u(\mathfrak b_m)\geq m$ therefore implies $u(\mathfrak n)\geq c_0^{-1}$.
The proof of \eqref{eq:singular-izumi} applies equally to the klt pair $(Y,D)$ over $k_0$.
Together with $A(u)=1$, it gives a constant $C>0$, independent of $u\in\mathscr W$, such that
\begin{equation}\label{eq:closed-point-power-bounds}
 \mathfrak n^{c_0m}
 \subseteq\mathfrak a_m(u)
 \subseteq\mathfrak n^{\lceil m/C\rceil}
 \qquad(m\geq1).
\end{equation}
Consequently,
\[
 0<\frac{e(\mathfrak n)}{C^r}
 \leq V_0:=\inf_{u\in\mathscr W}\vol(u)
 \leq c_0^r e(\mathfrak n)<\infty.
\]
Choose a sequence $u_i\in\mathscr W$ with $\vol(u_i)\to V_0$, and put $\mathfrak a_m^{(i)}:=\mathfrak a_m(u_i)$.
For all positive integers $i,m$, we have
\begin{equation}\label{eq:constrained-finite-levels}
 \mathfrak b_m\subseteq\mathfrak a_m^{(i)},
 \qquad m\lct(\mathfrak a_m^{(i)})\leq1.
\end{equation}
The uniform approximation in \cite{Blum}*{Proposition~3.7 and Section~7} applies because $A(u_i)=1$, $u_i(\mathfrak n)\geq c_0^{-1}$, and the volumes are uniformly bounded above.
Thus, for every integer $j\geq1$, there are positive integers $N_j,i_j$ such that
\begin{equation}\label{eq:constrained-uniform-volume}
 \frac{e(\mathfrak a_m^{(i)})}{m^r}
 \leq V_0+\frac1j
 \qquad(m\geq N_j,\ i\geq i_j).
\end{equation}

For each $m\geq1$, choose an affine parameter space $P_m$
with a universal family representing all ideals
$\mathfrak n^{c_0m}\subseteq J\subseteq\mathfrak n$,
as in \cite{Blum}*{Section~5.1}.
Put $H_N:=\prod_{m=1}^N P_m$.
Let $Z_N\subseteq H_N$ be the locus where the corresponding
ideals $(J_1,\ldots,J_N)$ satisfy
\[
 \begin{aligned}
 \mathfrak b_m&\subseteq J_m
     &&(1\leq m\leq N),\\
 J_aJ_b&\subseteq J_{a+b}
     &&(a,b\geq1,\ a+b\leq N),\\
 m\lct(J_m)&\leq1
     &&(1\leq m\leq N),\\
 \frac{e(J_m)}{m^r}&\leq V_0+\frac1j
     &&(1\leq j\leq N,\ N_j\leq m\leq N).
 \end{aligned}
\]

Ideal containment is constructible in these families.
By \cite{Blum}*{Propositions~A.1 and~A.2} and Noetherian
induction, multiplicity and log canonical threshold are
constructible as well.
For thresholds with respect to $(Y,D)$, the same proof
applies by resolving the universal ideal together with
the pullback of the fixed boundary $D$.
Thus $Z_N$ is constructible.
It is nonempty: for $i\geq\max_{1\leq j\leq N}i_j$, the tuple
$(\mathfrak a_1^{(i)},\ldots,\mathfrak a_N^{(i)})$
satisfies these conditions by
\eqref{eq:closed-point-power-bounds},
\eqref{eq:constrained-finite-levels}, and
\eqref{eq:constrained-uniform-volume}.

The truncation maps $\pi_{N+1,N}\colon H_{N+1}\to H_N$ satisfy
\[
 \pi_{N+1,N}(Z_{N+1})\subseteq Z_N.
\]
Since $k_0$ is uncountable, the argument of
\cite{Blum}*{Lemma~5.1} gives a compatible sequence
of $k_0$-points of the $Z_N$.
The corresponding ideals form a graded sequence
$\mathfrak d_\bullet$ of $\mathfrak n$-primary ideals satisfying
\[
 \begin{gathered}
 \mathfrak b_m\subseteq\mathfrak d_m,\qquad
 m\lct(\mathfrak d_m)\leq1
 \qquad(m\geq1),\\
 \frac{e(\mathfrak d_m)}{m^r}\leq V_0+\frac1j
 \qquad(j\geq1,\ m\geq N_j).
 \end{gathered}
\]

For this graded sequence, put
$u(\mathfrak d_\bullet):=\inf_{m\geq1}u(\mathfrak d_m)/m$.
The asymptotic threshold formula for graded sequences in \cite{Blum}*{Sections~2.9 and~7} gives
\[
 1=\lct(\mathfrak b_\bullet)
 \leq\lct(\mathfrak d_\bullet)
 =\sup_{m\geq1}m\lct(\mathfrak d_m)
 \leq1.
\]
By \cite{Xu}*{Theorem~1.1}, choose a quasi-monomial valuation $u$ computing $\lct(\mathfrak d_{\bullet})$ and normalize it by $u(\mathfrak d_\bullet)=1$.
Then $A(u)=1$, and the inclusions $\mathfrak b_m\subseteq\mathfrak d_m$ give
\[
 1=u(\mathfrak d_\bullet)
 \leq u(\mathfrak b_\bullet)
 \leq A(u)=1.
\]
Thus $u\in\mathscr W$.
Moreover, $\mathfrak d_m\subseteq\mathfrak a_m(u)$, so, for every $j\geq1$,
\[
 \frac{e(\mathfrak a_m(u))}{m^r}
 \leq\frac{e(\mathfrak d_m)}{m^r}
 \leq V_0+\frac1j
 \qquad(m\geq N_j).
\]
Letting $m\to\infty$ and then $j\to\infty$ gives
$V_0\leq\vol(u)\leq V_0$.
Hence a quasi-monomial minimizer exists.

\emph{Uniqueness.}
Let $u_0,u_1\in\mathscr W$ be two minimizers.
Take the geodesic $(\mathfrak g_{\bullet,t})_{t\in[0,1]}$ of $\mathfrak n$-filtrations between $\mathfrak a_\bullet(u_0)$ and $\mathfrak a_\bullet(u_1)$, as in \cite{BLQ}*{Definition~4.1}.
Its defining formula gives
\[
 \mathfrak b_m
 \subseteq\mathfrak a_m(u_0)\cap\mathfrak a_m(u_1)
 \subseteq\mathfrak g_{m,t}
 \qquad(m\geq1,\ 0<t<1).
\]
Since $u_i(\mathfrak a_\bullet(u_i))=1$, threshold convexity \cite{BLQ}*{Proposition~5.1} gives
\[
 \begin{aligned}
 1\leq\lct(\mathfrak g_{\bullet,t})
 &\leq(1-t)\lct(\mathfrak a_\bullet(u_0))
      +t\lct(\mathfrak a_\bullet(u_1))\\
 &\leq(1-t)A(u_0)+tA(u_1)=1.
 \end{aligned}
\]
Apply \cite{Xu}*{Theorem~1.1} to choose a valuation $v_t$ computing this threshold, normalized by $v_t(\mathfrak g_{\bullet,t})=1$.
Then $A(v_t)=1$ and
\[
 1=v_t(\mathfrak g_{\bullet,t})
 \leq v_t(\mathfrak b_\bullet)
 \leq A(v_t)=1,
\]
so $v_t\in\mathscr W$.
Since $\mathfrak g_{m,t}\subseteq\mathfrak a_m(v_t)$, we obtain
\[
 V_0\leq\vol(v_t)
 \leq e(\mathfrak g_{\bullet,t})\leq V_0,
\]
where the last inequality follows from the concavity of $e(\mathfrak g_{\bullet,t})^{-1/r}$ in \cite{BLQ}*{Theorem~1.1(2)} and the equal endpoint multiplicities $\vol(u_0)=\vol(u_1)=V_0$.
Thus $e(\mathfrak g_{\bullet,t})=V_0$ for every $t\in[0,1]$.
By \cite{BLQ}*{Theorem~1.1(3) and Lemma~3.20}, the endpoint valuation filtrations agree up to rescaling, and hence $u_1=c u_0$ for some $c>0$.
The normalization $u_0(\mathfrak b_\bullet)=u_1(\mathfrak b_\bullet)=1$ forces $c=1$.
Therefore $u_0=u_1$.
\end{proof}

\begin{proposition}\label{prop:constrained-volume-minimum}
The volume function on the set $\mathscr V$ in \eqref{eq:average-computers} has a unique minimizer, and this minimizer is quasi-monomial.
\end{proposition}

\begin{proof}
By \cref{prop:invariant-average}, the filtration $\mathfrak c_\bullet$ is linearly bounded and satisfies $\lct_R(\mathfrak c_\bullet)=1$.

To apply \cref{prop:closed-point-volume-minimum}, we first realize $R$ as the local ring of a closed point over a suitable field, and then pass to an algebraic closure of that field.
Choose a transcendence basis $\bar t_1,\ldots,\bar t_\ell$ of $\kappa/\C$, lift it to $t_1,\ldots,t_\ell\in R$, and put $k:=\C(t_1,\ldots,t_\ell)$.
Every nonzero polynomial in the $t_j$ has nonzero residue in $\kappa$ and is therefore a unit in $R$, so $k\subseteq R$.
After shrinking an affine klt neighborhood $U$ of $\eta_C$, the functions $t_j$ define a morphism $U\to\mathbb A^\ell_\C$.
Set
\[
 W:=U\times_{\mathbb A^\ell_\C}\Spec k,
 \qquad B_W:=B|_W,
 \qquad x:=\eta_C\in W.
\]
Then $W$ is normal, $\mathcal O_{W,x}=R$, and $\dim W=r$.
Since $\kappa/k$ is finite separable, $x$ is a closed point of $W$.
Every valuation of $R$ is trivial on $k$, because $k^\times\subseteq R^\times$.
Log resolutions and their discrepancy formulas restrict to the generic fiber, and resolutions over $W$ spread out over an open subset of $\mathbb A^\ell_\C$.
Thus $(W,B_W)$ is klt at $x$, and
\[
 A_{W,B_W}(u)=A(u)
 \qquad(u\in\operatorname{Val}_R).
\]

Fix an algebraic closure $\bar k$ of $k$ and put $G:=\operatorname{Gal}(\bar k/k)$.
Let $x_1,\ldots,x_\nu$ be the points of $W_{\bar k}$ lying over $x$, and put
\[
 \bar R_i:=\mathcal O_{W_{\bar k},x_i}.
\]
There are finitely many such points, and $G$ acts transitively on them.
Each $x_i$ has residue field $\bar k$, and $\dim\bar R_i=r$.
Let $B_{W_{\bar k}}$ be the pullback of $B_W$, and write $\bar A:=A_{W_{\bar k},B_{W_{\bar k}}}$.
The geometric pair is klt near the $x_i$.

Every valuation centered at $\mathfrak m$ extends to a valuation centered at one of the $x_i$, and conjugation gives an extension centered at each $x_i$.
For an extension $\bar u$ of $u$ centered at $x_i$, one has
\[
 \bar A(\bar u)=A(u),
 \qquad
 \bar u(\mathfrak c_\bullet\bar R_i)
 =u(\mathfrak c_\bullet).
\]
For finite separable extensions of $k$, the discrepancy equality is the étale case of \eqref{eq:finite-divisorial-discrepancy}.
To pass to $\bar k$, use the monomial-retraction description from \cite{JonssonMustata}*{Section~5.2, equation~(5.2)}, together with the transformation formula recalled in Section~\ref{sec:notation}.
For a log smooth model $(Y,D)$ over $(W,B_W)$, with $D$ containing the exceptional divisors and the support of the crepant pullback of $B_W$, write $r_{Y,D}$ for its monomial retraction.
Then
\[
 A(u)=\sup_{(Y,D)}A_{W,B_W}\bigl(r_{Y,D}(u)\bigr),
\]
and the same formula holds after extending the ground field.
Base change preserves the discrepancy coefficients and the values of local equations at the center of $\bar u$.
Hence, restricting the base-changed model to the component containing this center, we have
\[
 A_{W,B_W}\bigl(r_{Y,D}(u)\bigr)
 =\bar A\bigl(r_{Y_{\bar k},D_{\bar k}}(\bar u)\bigr)
 \leq\bar A(\bar u).
\]
Taking the supremum over $(Y,D)$ gives $A(u)\leq\bar A(\bar u)$.

Conversely, every log smooth model used to compute $\bar A(\bar u)$, together with its boundary, descends to a finite extension $k'/k$.
Let $u'$ be the restriction of $\bar u$ to the function field of the corresponding component of $W_{k'}$, and let $A'$ denote its log discrepancy for the pulled-back pair.
Write $r'$ and $\bar r$ for the retractions on the descended model and its base change, respectively.
The same comparison of local equations, followed by the finite-extension equality, gives
\[
 \bar A\bigl(\bar r(\bar u)\bigr)
 =A'\bigl(r'(u')\bigr)
 \leq A'(u')=A(u).
\]
Taking the supremum over the log smooth models over $\bar k$ gives the reverse inequality and proves $\bar A(\bar u)=A(u)$.
The equality of asymptotic orders follows from equality at each ideal $\mathfrak c_m$.
Consequently,
\[
 \lct_{\bar R_i}(\mathfrak c_\bullet\bar R_i)
 =\lct_R(\mathfrak c_\bullet)=1,
\]
where the threshold over $\bar R_i$ is taken with respect to the geometric pair.
The extended filtrations remain linearly bounded.
Therefore \cref{prop:closed-point-volume-minimum} applies at each $x_i$ to the set
\[
 \overline{\mathscr V}_i
 :=\left\{
 u\in\operatorname{Val}_{W_{\bar k},x_i}
 \ \mid\
 \bar A(u)=u(\mathfrak c_\bullet\bar R_i)=1
 \right\}.
\]
Denote its unique volume minimizer by $\bar v_i$.
Each $\bar v_i$ is quasi-monomial.

The action of $G$ preserves the extended ideals, discrepancies, and volumes.
Uniqueness therefore makes the collection $(\bar v_i)_i$ $G$-equivariant.
In particular, the restrictions $\bar v_i|_K$ coincide; denote their common restriction by $v$.
The preceding identities give
\[
 A(v)=v(\mathfrak c_\bullet)=1,
\]
so $v\in\mathscr V$.

We next descend the valuation ideals.
For every integer $m\geq1$, the normalization of $\bar v_i$ gives
\[
 \mathfrak c_m\bar R_i\subseteq\mathfrak a_m(\bar v_i).
\]
Since $R/\mathfrak c_m$ is a finite-dimensional $k$-algebra, there is a decomposition
\[
 (R/\mathfrak c_m)\otimes_k\bar k
 \simeq
 \prod_{i=1}^{\nu}\bar R_i/\mathfrak c_m\bar R_i.
\]
The ideals $\mathfrak a_m(\bar v_i)/\mathfrak c_m\bar R_i$ together define a $G$-stable ideal in this algebra.
Galois descent gives an ideal $J_m\subseteq R$ containing $\mathfrak c_m$ such that
\[
 J_m\bar R_i=\mathfrak a_m(\bar v_i)
 \qquad(1\leq i\leq\nu).
\]
Since all $\bar v_i$ restrict to $v$, contraction to $R$ gives $J_m=\mathfrak a_m(v)$.
Hence
\[
 \mathfrak a_m(v)\bar R_i=\mathfrak a_m(\bar v_i).
\]
The map $R\to\bar R_i$ is flat local, and its maximal ideal is $\mathfrak m\bar R_i$.
It therefore preserves lengths of finite-length modules under scalar extension and, consequently, Hilbert--Samuel multiplicities.
Passing to the limit yields
\[
 \vol_R(v)=\vol_{\bar R_i}(\bar v_i)
 \qquad(1\leq i\leq\nu).
\]

The extension $\operatorname{Frac}(\bar R_i)/K$ is algebraic.
Thus the value groups of $v$ and $\bar v_i$ have the same rational rank, and their valuation residue fields form an algebraic extension.
Since $\kappa/k$ and $\bar k/k$ are algebraic, we obtain
\[
 \operatorname{ratrank}(v)
 +\operatorname{trdeg}_{\kappa}\kappa(v)
 =
 \operatorname{ratrank}(\bar v_i)
 +\operatorname{trdeg}_{\bar k}\kappa(\bar v_i)
 =r.
\]
The last equality follows from quasi-monomiality of $\bar v_i$.
By the Abhyankar characterization in \cite{JonssonMustata}*{Proposition~3.7}, applied on a resolution, $v$ is quasi-monomial.

Finally, let $w\in\mathscr V$ and choose an extension $\bar w$ centered at some $x_i$.
The discrepancy and asymptotic-order identities above give $\bar w\in\overline{\mathscr V}_i$.
For every $m\geq1$, one has
\[
 \mathfrak a_m(w)\bar R_i\subseteq\mathfrak a_m(\bar w).
\]
Preservation of multiplicities under the flat local extension therefore gives
\[
 \vol_R(v)
 =\vol_{\bar R_i}(\bar v_i)
 \leq\vol_{\bar R_i}(\bar w)
 \leq\vol_R(w).
\]
This proves minimality of $v$.
If $w$ is also a minimizer, every inequality is an equality, so uniqueness in \cref{prop:closed-point-volume-minimum} gives $\bar w=\bar v_i$.
Restricting to $K$ yields $w=v$.
\end{proof}

\begin{proof}[Proof of \cref{thm:generic-eigen-log-pair}]
Use \cref{thm:invariant-nklt} to make every lc center and every irreducible component of the limiting non-klt locus totally invariant.
Fix a component $C$ not contained in $\Nklt(X,B)$ and use \cref{not:local-dynamics}.
By \cref{prop:invariant-average,prop:constrained-volume-minimum}, the set $\mathscr V$ has a unique volume minimizer $v$, which is quasi-monomial.
If $\varphi_*u=v$, then \cref{prop:invariant-average,lem:nonflat-volume-equality} give
\[
 \gamma u\in\mathscr V,\qquad
 \vol(v)\leq\vol(\gamma u)
              \leq\vol(F(\gamma u))=\vol(v).
\]
Uniqueness gives $\gamma u=v$, hence $h_*v=\gamma v$.
Set $g:=h^r$ and $q:=\gamma^r=d/d_C\in\Q_{>1}$.
We next replace $v$ by a divisorial eigenvaluation of $g$.

Choose a smooth birational model $\pi\colon Y\to X$ and an snc divisor $D$ on $Y$ such that $v$ is monomial at the generic point $\eta$ of a stratum of $D$, with rationally independent positive weights; see \cite{JonssonMustata}*{Lemma~3.6(ii)}.
Put $t:=\operatorname{ratrank}(v)$, and let $z_1,\ldots,z_t$ be local equations of the components of $D$ through $\eta$.

The induced map $g_Y:=\pi^{-1}\circ g\circ\pi$ is a rational self-map of $Y$.
Let $\Gamma\subseteq Y\times Y$ be the closure of its graph.
Choose a resolution $\widetilde Y\to\Gamma$, and denote the compositions with the two projections by $p,\tau\colon\widetilde Y\to Y$.
Thus $p$ is birational, and we have the commutative diagram
\[
\begin{tikzcd}[row sep=large,column sep=large]
 & \widetilde Y
   \arrow[dl,"p"']
   \arrow[dr,"\tau"] & \\
 Y \arrow[rr,dashed,"g_Y"]
   \arrow[d,"\pi"']
 & &
 Y \arrow[d,"\pi"] \\
 X \arrow[rr,"g"']
 & & X .
\end{tikzcd}
\]
In particular, $\pi\circ\tau=g\circ\pi\circ p$.
Replacing $\widetilde Y$ by a further log resolution and composing $p$ and $\tau$ with it, we may assume that
\[
 \operatorname{Supp}(p^*D)
 \cup\operatorname{Supp}(\tau^*D)
 \cup\operatorname{Exc}(p)
\]
is contained in the support of an snc divisor on $\widetilde Y$.
By \cite{JonssonMustata}*{Lemma 3.6(i)}, $v$ remains monomial with rationally independent positive weights $\alpha_j=v(w_j)$ at its center $\zeta$ on $\widetilde Y$, where $w_1,\ldots,w_t$ are local equations of the snc components through $\zeta$.
Both maps send $\zeta$ to its center $\eta$ on $Y$, and locally
\[
 p^*z_i=u_i\prod_jw_j^{P_{ij}},\qquad
 \tau^*z_i=u'_i\prod_jw_j^{Q_{ij}},
\]
where $u_i,u'_i$ are units and $P,Q$ are $t\times t$ matrices with nonnegative integer entries.
The weights $P\alpha$ are rationally independent, so $P$ has full rank.
The equality $g_*v=qv$ gives $Q\alpha=qP\alpha$; rational independence and $q\in\Q$ imply $Q=qP$.

For any $\beta\in\R_{>0}^t$, let $v_\beta$ be the monomial valuation in the $w_j$.
The residue field maps $\kappa(\eta)\to\kappa(\zeta)$ are injective.
On weighted associated graded rings, the map induced by $p$ is
\[
 \kappa(\eta)[Z_1,\ldots,Z_t]\longrightarrow
 \kappa(\zeta)[W_1,\ldots,W_t],\qquad
 Z_i\longmapsto\bar u_i\prod_jW_j^{P_{ij}},
\]
where the bar denotes residue at $\zeta$; the map induced by $\tau$ uses $Q$ and $u'_i$ instead.
Since both exponent matrices have full rank, distinct monomials have distinct exponent vectors, so both maps are injective even when weights coincide.
Thus $p_*v_\beta$ and $\tau_*v_\beta$ are monomial in the $z_i$ with weights $P\beta$ and $Q\beta$, and
\[
 g_*(p_*v_\beta)=\tau_*v_\beta=q\,p_*v_\beta.
\]

Choose $\beta\in\Q_{>0}^t$.
Since $v_\beta$ is a monomial valuation with rational weights, it is divisorial.
As $p$ is birational, there exist a prime divisor $E$ over $X$ and a constant $c\in\Q_{>0}$ such that
\[
 c\,p_*v_\beta=\ord_E.
\]
The relation $g_*(p_*v_\beta)=q\,p_*v_\beta$ therefore gives
\[
 g_*\ord_E=q\,\ord_E.
\]
Since all weights are positive, this valuation has the same center as $v$, namely $C$ on $X$.
Choose $f\in K^\times$ with $\ord_E(f)=1$.
Then $q=\ord_E(g^*f)\in\mathbb Z$; put $e:=q>1$.

Since $F^r(\ord_E)=\ord_E$, \cref{lem:nonflat-volume-equality} shows that the only extension of $\ord_E$ through $g$ is $\ord_E/e$.
The discrepancy and norm formulas in \cref{lem:finite-divisorial} give
\[
 \ord_E(R_{g,B})=(e-1)A_{X,B}(E),\qquad
 \ord_E(P_gD)=e^{-1}\ord_E(D)
\]
for every $\Q$-Cartier divisor $D$ on $X$.
Iteration and summation give \eqref{eq:global-eigen-series} for $g$; in particular $s_{g,B}(E)=A_{X,B}(E)$.
The current is unchanged by iteration by \cref{thm:log-current-construction}.
Finally, take a common power for the finitely many components $C$.
Eigenvalues are raised to positive integer powers, uniqueness of extensions persists under composition, and the same calculation gives all the asserted identities for this common iterate.
\end{proof}

\begin{remark}
Theorem~\ref{thm:generic-eigen-log-pair} gives an affirmative answer to the projective-space part of \cite[Question~6.1]{LuoMeng26} under the following additional assumption.
Let $f\colon\mathbf{P}^n\to\mathbf{P}^n$ be a $q$-polarized endomorphism, where $q>1$, and let $Z\subseteq \mathbf{P}^n$ be a totally invariant subvariety.
Assume that $Z$ is an irreducible component of $\operatorname{Nklt}(\mathbf{P}^n,0;T_{f,0})$.
Then there exist an integer $s\geq 1$ and a prime divisor $E$ over $\mathbf{P}^n$ with center $Z$ such that
\[
  (f^s)_*\operatorname{ord}_E=q^s\operatorname{ord}_E.
\]
Indeed, apply Theorem~\ref{thm:generic-eigen-log-pair} to $(X,B,h)=(\mathbf{P}^n,0,f)$.
In its construction, the eigenvalue for the chosen iterate is
\[
  \left(
    \frac{\deg(f^s)}{\deg(f^s|_Z)}
  \right)^{1/\operatorname{codim}_{\mathbf{P}^n}Z}
  =q^s.
\]
Thus $\operatorname{ord}_E$ is the primitive divisorial eigenvaluation required in that question.
\end{remark}

\section{Equivariant birational models}\label{sec:equivariant-models}
\label{subsec:equivariant-extraction}

After replacing $h$ by an iterate, we construct an equivariant birational model $(W,B_W,g)$ satisfying
\[
  \operatorname{Nklt}(W,B_W;T_{g,B_W})
  =\operatorname{Nklt}(W,B_W).
\]
We obtain this equality by successively extracting the divisors provided by \cref{thm:generic-eigen-log-pair}.

\begin{definition}\label{def:absorbing-triple}
Let $(X,B,h)$ be an int-amplified dynamical triple, and let $T_{h,B}$ be the current of Theorem~\ref{thm:log-current-construction}.
The triple is \emph{absorbing} if
\begin{equation}\label{eq:absorbing-triple}
 \Nklt(X,B;T_{h,B})=\Nklt(X,B).
\end{equation}
\end{definition}

For any positive integer $a$, Theorem~\ref{thm:log-current-construction} shows that $(X,B,h)$ is absorbing if and only if $(X,B,h^a)$ is absorbing.

\begin{definition}\label{def:equivariant-model}
Let $(X,B,h)$ be an int-amplified dynamical triple.
\begin{enumerate}
\item\label{it:equivariant-model}
An \emph{equivariant birational model} of $(X,B,h)$ is a projective dynamical triple $(W,B_W,g)$ together with a projective birational morphism $\pi\colon W\to X$ satisfying $\pi\circ g=h\circ\pi$ and $\pi_*B_W\geq B$.
We write $\pi\colon(W,B_W,g)\to(X,B,h)$.
The identity $g\circ\pi^{-1}=\pi^{-1}\circ h$ and \cite{MengZhangIntAmplified}*{Proposition 2.6} show that $g$ is int-amplified.
Thus $(W,B_W)$ is lc by \cref{lem:basic-dynamical-properties}.
\item\label{it:absorbing-model}
The model is \emph{absorbing} if $(W,B_W,g)$ is absorbing in the sense of Definition~\ref{def:absorbing-triple}.
\item\label{it:correction-crepancy}
For a finite set $\Lambda\subset\Q_{>1}$, the model is \emph{crepant for the corrections with interpolation conditions for $\Lambda$} if, for every coefficient family $\mathbf c$ from Definition~\ref{def:moment-conditions} satisfying these interpolation conditions, one has
\begin{equation}\label{eq:correction-crepancy}
 K_W+\mathcal C_{g,N}(B_W;\mathbf c)
 =\pi^*(K_X+\mathcal C_{h,N}(B;\mathbf c))
 \qquad(N\geq\max\{1,m\}),
\end{equation}
\end{enumerate}
\end{definition}

\begin{lemma}\label{lem:model-transformation}
Let $\pi\colon(W,B_W,g)\to(X,B,h)$ be an equivariant birational model of an int-amplified dynamical triple.
Put $F:=K_W+B_W-\pi^*(K_X+B)$ and assume $g^*F=eF$ for some $e\in\Q_{>1}$.
For positive integers $j,N$, one has
\begin{align}
 R_{g,B_W}&=\pi^*R_{h,B}-(e-1)F,\label{eq:model-ramification}\\
 P_g^jR_{g,B_W}&=\pi^*P_h^jR_{h,B}-(e-1)e^{-j}F,\label{eq:model-terms}\\
 S_{g,B_W,N}&=\pi^*S_{h,B,N}-(1-e^{-N})F.\label{eq:model-partial-sums}
\end{align}
For every coefficient family $\mathbf c$, one has
\begin{equation}\label{eq:model-correction}
\begin{split}
 K_W+\mathcal C_{g,N}(B_W;\mathbf c)
 -\pi^*(K_X+\mathcal C_{h,N}(B;\mathbf c))
 =\left(e^{-N}-(e-1)\sum_{j=1}^m c_{j,N}e^{-j}\right)F
\end{split}
\end{equation}
for $N\geq\max\{1,m\}$.
In particular, $\pi$ is crepant for the correction with interpolation conditions for $\{e\}$.
\end{lemma}

\begin{proof}
The ramification formula gives \eqref{eq:model-ramification}.
Finite norms give $P_g(\pi^*D)=\pi^*(P_hD)$ for every $\Q$-Cartier divisor $D$ on $X$, while $P_gF=e^{-1}F$.
Iteration and summation give \eqref{eq:model-terms} and~\eqref{eq:model-partial-sums}.
Substitution into Definition~\ref{def:moment-conditions} gives \eqref{eq:model-correction}.
\end{proof}

The following proposition gives an affirmative answer to \cite[Question~6.2]{LuoMeng26}.

\begin{proposition}\label{prop:equivariant-extraction-given-eigenvaluation}
Let $(X,B,h)$ be an int-amplified dynamical triple, and let $E$ be a prime divisor over $X$ with center $C$ on $X$.
Assume that, for some real number $e>1$, the only extension of $\ord_E$ through $h$ is $\ord_E/e$.
Then $e$ is an integer, $C$ is totally invariant, and there exists an equivariant birational model $\pi\colon (W,B_W,g)\to (X,B,h)$ such that $E$ appears on $W$ as a $\Q$-Cartier divisor and
\[
  g^*E=eE,\qquad
  \pi^{-1}(C)=\operatorname{Supp}E,\qquad
  \operatorname{coeff}_E B_W=1.
\]
The morphism $\pi$ is an isomorphism over $X\setminus C$, the divisor $-E$ is $\pi$-ample, and one has
\[
  B_W=\pi^*(K_X+B)+A_{X,B}(E)E-K_W.
\]
The model is crepant for corrections with interpolation conditions for $\{e\}$.
\end{proposition}

\begin{proof}
By the assumption, $h_*\operatorname{ord}_E=e\operatorname{ord}_E$.
Choose $t\in K(X)^\times$ with $\operatorname{ord}_E(t)=1$.
Then $e=\operatorname{ord}_E(h^*t)\in\mathbb Z_{>1}$.
By \cref{lem:finite-divisorial}, every irreducible component of $h^{-1}(C)$ is the center of an extension, so the uniqueness assumption gives $h^{-1}(C)=C$.

Put $a_E:=A_{X,B}(E)\geq 0$, and choose a positive integer $M$ such that $e^{-M}a_E<1$.
Put $\Delta:=B+S_{h,B,M}$.
By \cref{lem:finite-divisorial}, we have
\[
 \alpha:=A_{X,\Delta}(E)=e^{-M}a_E\in[0,1),
\]
and $(X,\Delta)$ is lc by \cref{lem:basic-dynamical-properties}.
As in the initial construction in \cite{MoragaExtraction}*{proof of Theorem~1}, choose a projective crepant $\Q$-factorial dlt model $\mu\colon(Y,\Delta_Y)\to(X,\Delta)$ on which $E$ appears as a prime divisor $E_Y$.
Thus
\[
 K_Y+\Delta_Y=\mu^*(K_X+\Delta),
 \qquad
 \coeff_{E_Y}\Delta_Y=1-\alpha>0.
\]
Set
\[
 \delta:=\frac{1-\alpha}{2},
 \qquad
 \Delta_Y^-:=\Delta_Y-\delta E_Y.
\]
Then $(Y,\Delta_Y^-)$ is dlt and
\[
 K_Y+\Delta_Y^-
 =\mu^*(K_X+\Delta)-\delta E_Y
 \sim_{\Q,X}-\delta E_Y.
\]

Apply \cite{HaconXuClosures}*{Theorem~1.6} with
$\Delta'=\delta E_Y$ and $\Delta''=\Delta_Y^-$.
Since $\mu$ is birational, a Mori fiber space over $X$
cannot occur.
We therefore obtain a good minimal model over $X$,
followed by its relative ample model:
\[
 Y\dashrightarrow Y'\longrightarrow W
 \xrightarrow{\pi}X.
\]
Write $\Delta_W$ and $\Delta_W^-$ for the birational
transforms of $\Delta_Y$ and $\Delta_Y^-$, respectively.
Then
\[
 K_W+\Delta_W=\pi^*(K_X+\Delta),
\]
and $K_W+\Delta_W^-$ is $\Q$-Cartier and $\pi$-ample.

We claim that $E_Y$ is not contracted on $W$.
Otherwise $\Delta_W^-=\Delta_W$, and monotonicity
of log discrepancies along the MMP and passage to
the ample model would give
\[
 \alpha
 =A_{W,\Delta_W^-}(E)
 \geq A_{Y,\Delta_Y^-}(E_Y)
 =\alpha+\delta,
\]
a contradiction.
Denote its birational transform on $W$ by $E$.
Consequently,
\[
 -\delta E
 =K_W+\Delta_W^--\pi^*(K_X+\Delta)
\]
is $\Q$-Cartier and $\pi$-ample.
Thus $E$ is $\Q$-Cartier and $-E$ is $\pi$-ample.

Finally, every curve contracted by $\pi$ is contained
in $\Supp E$: otherwise its intersection with the
effective $\Q$-Cartier divisor $E$ would be nonnegative,
contrary to the $\pi$-ampleness of $-E$.
If a $\pi$-exceptional prime divisor $G\ne E$ existed,
a general positive-dimensional fiber of
$G\to\pi(G)$ would contain a contracted curve not
contained in $E$.
Hence there are no $\pi$-exceptional prime divisors
other than possibly $E$.

Choose a positive integer $k$ so that $kE$ is Cartier and the algebra
\[
\mathcal A:=\bigoplus_{m \ge 0} \pi_*\mathcal O_W(-mkE)
\]
is generated in degree one.
Put $\mathcal I_m:=\pi_*\mathcal O_W(-mkE)$.
For every $m\geq0$ and every nonempty open subset $V\subseteq X$, we have
\begin{equation}\label{eq:lc-extraction-section-ideals}
 \mathcal I_m(V)
 =\begin{cases}
 \{f\in\mathcal O_X(V)\mid\ord_E(f)\geq mk\},
   & V\cap C\ne\varnothing,\\[2pt]
 \mathcal O_X(V),&V\cap C=\varnothing.
 \end{cases}
\end{equation}
Indeed, if $V\cap C\ne\varnothing$, then $V$ contains the generic point of $C$,
and sections are precisely the regular functions whose pullbacks vanish along $E$ to order at least $mk$.
If $V\cap C=\varnothing$, then $E\cap\pi^{-1}(V)=\varnothing$, so normality of $X$ gives $\mathcal I_m|_V=\mathcal O_V$.
Furthermore, $W=\Proj_X\mathcal A$ is the blowup of $\mathcal I_1$, with $\mathcal I_1\mathcal O_W=\mathcal O_W(-kE)$.
Since $V(\mathcal I_1)=C$, this proves
$\pi^{-1}(C)=\Supp E$ and that $\pi$ is an isomorphism over $X\setminus C$.

Normalize the main component of the base change by $h$, obtaining a commutative diagram
\[
\begin{tikzcd}
 W' \arrow[r,"b"] \arrow[d,"\pi'"']
   & W \arrow[d,"\pi"] \\
 X \arrow[r,"h"']
   & X.
\end{tikzcd}
\]
Here $b$ is finite and $\pi'$ is birational.
There is a unique prime divisor $E'$ above $E$, representing $\ord_E$ on the source field, and $b^*E=eE'$.
In particular, $-E'$ is relatively ample over $X$.
Applying the same local description as in \eqref{eq:lc-extraction-section-ideals},
with the valuation condition on open sets meeting $C$ and no condition outside $C$, gives
\[
 \pi'_*\mathcal O_{W'}(-mkb^*E)=\mathcal I_{em}
 \qquad(m\geq0).
\]
The divisor $-b^*E$ is $\pi'$-ample, so $W'$ is the relative Proj of this algebra.
Hence $W'\simeq W$ over $X$.
Thus $b$ induces a lift $g\colon W\to W$ with $g^*E=eE$.

Define
\[
B_W:=\pi^*(K_X+B)+a_EE-K_W.
\]
The discrepancy formula and the absence of other $\pi$-exceptional prime divisors give
\[
 \coeff_E B_W=1,\qquad
 \coeff_D B_W=\coeff_D(\pi_*^{-1}B)
 \quad\text{for every prime divisor }D\ne E\text{ on }W.
\]
In particular, $B_W\geq0$ and $\pi_*B_W\geq B$.
Moreover, $K_W+B_W$ is $\Q$-Cartier and the ramification formula gives
\[
 R_{g,B_W}=\pi^*R_{h,B}-(e-1)a_EE\geq0.
\]
Indeed, the coefficient at $E$ is zero, and the other coefficients are nonnegative.
Hence this is an equivariant birational model.
Lemma~\ref{lem:model-transformation}, with $F=a_EE$, gives crepancy for the interpolation condition at $e$.
\end{proof}

\begin{proposition}\label{prop:absorb-new-component}
Let $(X,B,h)$ be an int-amplified dynamical triple, and let $C$ be an irreducible component of $\Nklt(X,B;T_{h,B})$ not contained in $\Nklt(X,B)$.
After replacing $h$ by an iterate, there are an equivariant birational model $\pi\colon(W,B_W,g)\to(X,B,h)$, a prime $\Q$-Cartier divisor $E$ on $W$ with center $C$, and an integer $e>1$ such that
\begin{equation}\label{eq:lc-extraction-full-fiber}
 g^*E=eE,\qquad \pi^{-1}(C)=\Supp E,\qquad \coeff_E B_W=1.
\end{equation}
The morphism $\pi$ is an isomorphism over $X\setminus C$.
Writing $a_E:=A_{X,B}(E)$, one has
\begin{equation}\label{eq:absorption-step-transform}
 K_W+B_W=\pi^*(K_X+B)+a_EE.
\end{equation}
The model is crepant for corrections with interpolation conditions for $\{e\}$.
\end{proposition}

\begin{proof}
By \cref{thm:invariant-nklt,thm:generic-eigen-log-pair}, we may replace $h$ by an iterate so that $h^{-1}(C)=C$ and there is a prime divisor $E$ over $X$ with center $C$ and an integer $e>1$ such that the only extension of $\ord_E$ through $h$ is $\ord_E/e$.
Then the assertion follows from \cref{prop:equivariant-extraction-given-eigenvaluation}.
\end{proof}

\begin{theorem}\label{thm:all-dimensional-lc-model}
Let $(X,B,h)$ be an int-amplified dynamical triple.
After replacing $h$ by an iterate, there is an absorbing equivariant projective birational model $\pi\colon(W,B_W,g)\to(X,B,h)$ and a finite set $\Lambda_{\rm mod}\subset\mathbb Z_{>1}$ such that $\pi$ is crepant for corrections with interpolation conditions for $\Lambda_{\rm mod}$.
For every positive integer $a$, the same model is absorbing for $(g^a,h^a)$ and is crepant for the corrections associated with these maps and interpolation conditions for $\Lambda_{\rm mod}^a$.
\end{theorem}

\begin{proof}
Put $Z:=\Nklt(X,B;T_{h,B})$.
We argue by induction on the number of irreducible components
of $Z$ not contained in $\Nklt(X,B)$.
If this number is zero, take the identity model and
$\Lambda_{\mathrm{mod}}=\varnothing$.

Otherwise, choose such a component $C$.
By \cref{prop:absorb-new-component}, there are an integer $p\geq1$,
an equivariant birational model
\[
 \rho\colon(Y,B_Y,f)\longrightarrow(X,B,h^p),
\]
and an integer $e>1$ such that
\[
 F:=K_Y+B_Y-\rho^*(K_X+B),
 \qquad f^*F=eF.
\]
Moreover, $\rho$ is an isomorphism over $X\setminus C$ and
$\rho^{-1}(C)\subseteq\Nklt(Y,B_Y)$.

By \cref{thm:log-current-construction}, replacing $h$ by $h^p$
does not change $Z$.
Over $X\setminus C$, the morphism $\rho$ identifies the boundaries
and the currents.
Hence every irreducible component of
$\Nklt(Y,B_Y;T_{f,B_Y})$ not contained in $\Nklt(Y,B_Y)$
is the strict transform of an irreducible component of $Z$
other than $C$ that is not contained in $\Nklt(X,B)$.
The corresponding number of components therefore strictly decreases.

By induction, there are an integer $q\geq1$, an absorbing model
\[
 \tau\colon(W,B_W,g)\longrightarrow(Y,B_Y,f^q),
\]
and a finite set $\Lambda_Y\subset\mathbb Z_{>1}$
with the asserted crepancy property for every further iterate.
Set
\[
 \pi:=\rho\circ\tau,
 \qquad
 \Lambda_{\mathrm{mod}}:=\Lambda_Y\cup\{e^q\}.
\]
Then $\pi\circ g=h^{pq}\circ\pi$.

Fix an integer $a\geq1$ and a coefficient family $\mathbf c$
satisfying the interpolation conditions for
\[
 \Lambda_{\mathrm{mod}}^a
 =\Lambda_Y^a\cup\{e^{qa}\}.
\]
The induction hypothesis applies to $\tau$ with the maps
$g^a$ and $f^{qa}$.
On the other hand,
\[
 (f^{qa})^*F=e^{qa}F,
\]
so \cref{lem:model-transformation} applies to $\rho$ with
the maps $f^{qa}$ and $h^{pqa}$.
Using the same coefficient family $\mathbf c$ in both steps gives
\[
\begin{aligned}
 K_W+\mathcal C_{g^a,N}(B_W;\mathbf c)
 &=
 \tau^*\bigl(
 K_Y+\mathcal C_{f^{qa},N}(B_Y;\mathbf c)
 \bigr)\\
 &=
 \pi^*\bigl(
 K_X+\mathcal C_{h^{pqa},N}(B;\mathbf c)
 \bigr)
\end{aligned}
\]
for every $N\geq\max\{1,m\}$.

Finally, $T_{g^a,B_W}=T_{g,B_W}$, so the model remains
absorbing for every $a\geq1$.
Taking $h^{pq}$ as the chosen iterate of $h$ proves the theorem.
\end{proof}

\begin{remark}\label{rem:polarized-models}
Suppose that $h$ is $q$-polarized, and choose an ample Cartier divisor $H$ with $h^*H\sim qH$.
Every equivariant projective birational lift $u\colon V\to V$ of $h^k$, with structure morphism $\rho\colon V\to X$, satisfies
\[
 u^*\rho^*H\sim q^k\rho^*H.
\]
Since $\rho^*H$ is nef and big, $u$ is $q^k$-polarized by \cite{MengZhangIntAmplified}*{Theorem 2.3}.
If $u^*E=eE$ for a prime $\Q$-Cartier divisor $E$ on $V$, choose an ample Cartier divisor $A$ with $u^*A\sim q^kA$ and put $n:=\dim V$.
Then
\[
 e(q^k)^{n-1}(E\cdot A^{n-1})
 =\deg(u)(E\cdot A^{n-1})
 =(q^k)^n(E\cdot A^{n-1}),
\]
so $e=q^k$.
Consequently, after a total iterate $h^r$, the model in \cref{thm:all-dimensional-lc-model} has its set of interpolation parameters contained in $\{q^r\}$.
Its induced maps on normalizations of invariant lc centers are also $q^r$-polarized, by restriction of a polarization.
\end{remark}

\section{Adjunction and the proof of the main theorem}\label{sec:adjunction-proof}

For a coefficient family $\mathbf c$ and the absorbing model $(W,B_W,g)$ from \Cref{thm:all-dimensional-lc-model}, \Cref{lem:uniform-nlc-confinement} gives
\[
  \operatorname{Nlc}
    (W,\mathcal C_{g,N}(B_W;\mathbf c))
  \subseteq \operatorname{Nklt}(W,B_W)
\]
for all sufficiently large $N$.
We use adjunction and induction on dimension to prove
log canonicity near the lc centers of $(W,B_W)$.

\subsection{Adjunction for corrected pairs}\label{subsec:restriction}
\label{subsec:moduli-divisor}\label{subsec:corrected-adjunction}

\begin{proposition}\label{prop:higher-center-restriction}
Let $X$ be a normal variety and $h\colon X\to X$ a finite endomorphism.
Let $C\subset X$ be an irreducible closed subvariety with $h^{-1}(C)=C$, let $\nu\colon C^\nu\to C$ be its normalization, and let $h_C\colon C^\nu\to C^\nu$ be the induced finite surjective endomorphism.
Put $P_h=(\deg h)^{-1}h_*$ and $P_{h_C}=(\deg h_C)^{-1}(h_C)_*$.
For every $\Q$-Cartier divisor $D$ on $X$ with $C\nsubseteq\Supp D$,
\begin{equation}\label{eq:higher-center-restriction}
 (P_hD)|_{C^\nu}=P_{h_C}(D|_{C^\nu}).
\end{equation}
Here restriction means pullback along $C^\nu\to C\hookrightarrow X$.
If $h$ is int-amplified, then $h_C$ is int-amplified.
\end{proposition}
\begin{proof}
Put $d=\deg h$ and $b=\deg h_C$.
Total invariance makes $h^\#\colon\mathcal O_{X,\eta_C}\to\mathcal O_{X,\eta_C}$ a finite local homomorphism.
Identify its image with a subring $R$ of the source ring $S=\mathcal O_{X,\eta_C}$, and write
\[
 K=\operatorname{Frac}R,\quad L=\operatorname{Frac}S,\quad
 k=\kappa(R),\quad l=\kappa(S).
\]
Then $[L:K]=d$ and $[l:k]=b$.
The normal ring $S$ is the integral closure of $R$ in $L$, and its unique maximal ideal lies over the maximal ideal of $R$.

Take a finite Galois closure $M/K$ containing $L$, let $S'$ be the integral closure of $R$ in $M$, and choose a maximal ideal $\mathfrak q$ of $S'$ over the maximal ideal of $R$.
Put
\[
 G=\operatorname{Gal}(M/K),\qquad J=\operatorname{Gal}(M/L),\qquad
 G_{\mathfrak q}=\{\sigma\in G\mid\sigma(\mathfrak q)=\mathfrak q\}.
\]
By \cite{StacksProject}*{Tag 0BRK}, the group $G$ acts transitively on the maximal ideals of $S'$ over the maximal ideal of $R$.
Since all these ideals lie over the unique maximal ideal of $S$, the same result for $M/L$ shows that $J$ acts transitively on them as well.
Consequently $G=JG_{\mathfrak q}=G_{\mathfrak q}J$, so $G_{\mathfrak q}$ acts transitively on the $d$ embeddings $L\hookrightarrow M$ over $K$, identified with $G/J$.
Every such embedding sends $S$ into $S'$ and reduces at $\mathfrak q$ to an embedding $l\hookrightarrow\kappa(\mathfrak q)$ over $k$.
The extension $\kappa(\mathfrak q)/k$ is normal, hence Galois in characteristic zero, and $G_{\mathfrak q}$ maps onto its Galois group by \cite{StacksProject}*{Tag 0BRK}.
Thus all $b$ residue-field embeddings occur.
The reduction map between the two sets of embeddings is $G_{\mathfrak q}$-equivariant, so transitivity makes all its fibers have the same cardinality $d/b$.
In particular, $b$ divides $d$.
For every $u\in S^*$, multiplying the reductions of its $d$ conjugates therefore gives
\[
 \overline{\Nm_{L/K}(u)}=\Nm_{l/k}(\bar u)^{d/b}.
\]
The norm belongs to $R^*$, since both $u$ and $u^{-1}$ are integral over the normal ring $R$.

Write $\Nm_h=(h^*)^{-1}\circ\Nm_{K(X)/h^*K(X)}$, and define $\Nm_{h_C}$ in the same way on $K(C^\nu)$.
The preceding identity becomes
\[
 \overline{\Nm_h(u)}=\Nm_{h_C}(\bar u)^{d/b}.
\]
Choose a Cartier multiple of $D$.
Locally over the target, trivialize its associated line bundle on the finite semilocal inverse image and write this multiple as $\operatorname{div}(u)$.
Since $D$ does not contain $C$, the function $u$ is a unit at the source generic point of $C$.
Taking divisors in the norm identity on $C^\nu$ and dividing by the Cartier multiple and by $d$ gives \eqref{eq:higher-center-restriction}.
The local identities glue on $C^\nu$.
If $h$ is int-amplified, restrict an ample divisor $H_0$ with $h^*H_0-H_0$ ample to $C$ and pull back to $C^\nu$.
Both divisors remain ample, proving the final assertion.
\end{proof}

Before applying adjunction, we recall the necessary notation for b-divisors.

\begin{definition}\label{def:b-divisors}
Let $X$ be a normal variety.
\begin{enumerate}
\item\label{it:b-divisors}
A \emph{$\Q$-b-divisor} $\mathbf M$ assigns a $\Q$-Weil divisor $M_Y$, called its \emph{trace}, to each normal proper birational model $Y\to X$, with $\tau_*M_{Y'}=M_Y$ whenever $\tau\colon Y'\to Y$ is a morphism over $X$.
It is \emph{$\Q$-b-Cartier} if some trace $M_Y$ is $\Q$-Cartier and $M_{Y'}=\tau^*M_Y$ on every higher model.
In this case we say that $\mathbf M$ \emph{descends to $Y$} and is the \emph{Cartier closure} $\overline{M_Y}$.
A $\Q$-b-divisor $\mathbf M$ is \emph{b-nef} if it descends to a nef divisor, and we write $\mathbf M\sim_\Q0$ if it descends to a $\Q$-linearly trivial divisor.
\item\label{it:adjunction-discriminant}
Let $(X,B)$ be a log pair, let $C$ be an lc center, and let $\nu\colon C^\nu\to C$ be its normalization.
Fix a prime divisor $T$ over $X$ with center $C$ and $A_{X,B}(T)=0$.
For a normal proper birational model $p\colon Z'\to C^\nu$, choose a log resolution $\rho\colon Y\to X$ extracting $T$ and resolving its map $f_T\colon T\to Z'$, and put
\[
 K_Y+B_Y=\rho^*(K_X+B),\qquad B_T=(B_Y-T)|_T.
\]
For a prime divisor $F$ on $Z'$, work near $\eta_F$, where $F$ is Cartier, and define
\[
 \alpha_F=\sup\{t\in\Q\mid (T,B_T+t f_T^*F)
                 \text{ is sub-lc over }\eta_F\}.
\]
By \cite{FHInversion}*{Definition~2.1}, the thresholds
$\alpha_F$ are independent of the chosen resolution $Y$.
Define
\[
 B_{Z'}:=\sum_F(1-\alpha_F)F.
\]
By \cite{FHInversion}*{Theorem~1.2}, these traces form
the discriminant b-divisor $\mathbf B$ of
\cite{FHAdjunction}*{Definition~1.3}.
In particular, $\mathbf B$ is independent of the
choice of the lc place $T$.
\end{enumerate}
\end{definition}

\begin{proposition}\label{prop:invariant-nef-bdivisor}
Let $X$ be a normal projective variety and let $h\colon X\to X$ be an int-amplified endomorphism.
Let $\mathbf M$ be a b-nef $\Q$-b-Cartier b-divisor with trace $M_X$.
Assume that
\begin{equation}\label{eq:invariant-moduli-trace}
 h^*M_X\sim_\Q M_X.
\end{equation}
Then $M_X\sim_\Q0$ is $\Q$-Cartier, and $\mathbf M$ is the Cartier closure of $M_X$.
\end{proposition}

\begin{proof}
There is nothing to prove when $\dim X=0$.
Put $n:=\dim X>0$, and choose a smooth projective model
$\pi\colon Y\to X$ on which $\mathbf M$ descends to a nef
divisor $M_Y$.
By \cite{MengZhangIntAmplified}*{Proposition 2.5},
\eqref{eq:invariant-moduli-trace} gives $M_X\equiv_w0$.
For an ample Cartier divisor $H$ on $X$, put $A:=\pi^*H$.
The projection formula gives $M_Y\cdot A^{n-1}=0$.
Choose an ample Cartier divisor $L$ on $Y$.
Since $A$ is big, there is a rational number $\delta>0$
such that $A-\delta L\sim_\Q G\geq0$.
Nefness of $M_Y,A,L$ gives
\[
 0=M_Y\cdot A^{n-1}
 \geq\delta^{n-1}M_Y\cdot L^{n-1}\geq0.
\]
Thus $M_Y\equiv0$ by the Hodge index theorem
(and by degree when $n=1$).

Resolve $\pi^{-1}\circ h\circ\pi$ to obtain
\[
\begin{tikzcd}[column sep=large,row sep=large]
 & Z \arrow[dl,"a"'] \arrow[dr,"b"] & \\
 Y \arrow[d,"\pi"']
 && Y \arrow[d,"\pi"] \\
 X \arrow[rr,"h"']
 && X,
\end{tikzcd}
\]
where $Z$ is smooth and projective and $a$ is birational.
Put $\mu:=\pi\circ a$.
Choose $f\in K(X)^*$ and an integer $r>0$ such that
\[
 h^*M_X-M_X=\frac1r\operatorname{div}_X(f),
\]
and set
\[
 E:=b^*M_Y-a^*M_Y
       -\frac1r\operatorname{div}_Z(f\circ\mu).
\]
Since we have
\[
 \mu_*E=h^*M_X-M_X-\frac1r\operatorname{div}_X(f)=0,
\]
the divisor $E$ is $\mu$-exceptional.
Moreover, $E\equiv0$.
The negativity lemma \cite{KM98}*{Lemma 3.39}, applied to $E$ and $-E$, gives $E=0$, hence
\[
 b^*M_Y\sim_\Q a^*M_Y.
\]

For a sufficiently divisible positive integer $m$, therefore,
\[
 \mathcal L:=\mathcal O_Y(mM_Y)\in\operatorname{Pic}^0(Y),
 \qquad b^*\mathcal L\simeq a^*\mathcal L.
\]
Birational invariance of $\operatorname{Pic}^0$ defines
\[
 \Phi:=(a^*)^{-1}\circ b^*
 \colon\operatorname{Pic}^0(Y)\to\operatorname{Pic}^0(Y),
 \qquad \Phi(\mathcal L)=\mathcal L.
\]

Let $\alpha\colon Y\to A_Y:=\operatorname{Alb}(Y)$ be an
Albanese morphism.
By \cite{MengBuildingBlocks}*{Theorem 1.8(2)},
the endomorphism $g$ induced by $h$ on $A_Y$ is
int-amplified, and
\[
\begin{tikzcd}[column sep=large,row sep=large]
 Z \arrow[r,"b"] \arrow[d,"\alpha\circ a"']
 & Y \arrow[d,"\alpha"] \\
 A_Y \arrow[r,"g"']
 & A_Y
\end{tikzcd}
\]
commutes.
Write $g(u)=\psi(u)+t$, where $\psi$ is a group endomorphism
and $t\in A_Y$.
Translations act trivially on numerical divisor classes,
so $\psi$ is int-amplified.
By Albanese--Picard duality, $\Phi$ is dual to $\psi$.

Choose an ample Cartier divisor $H_A$ such that
$\psi^*H_A-H_A$ is ample.
If $B:=\ker(\psi-\mathrm{id})^0$ had positive dimension, then
\[
 (\psi^*H_A-H_A)|_B\sim0
\]
would contradict ampleness.
Thus $\psi-\mathrm{id}$, and hence its dual
$\Phi-\mathrm{id}$, is an isogeny.
Since $\Phi(\mathcal L)=\mathcal L$, the line bundle
$\mathcal L$ is torsion, proving $M_Y\sim_\Q0$.

Pushing a principal multiple of $M_Y$ to $X$ gives
$M_X\sim_\Q0$ and $M_Y=\pi^*M_X$.
Hence $\mathbf M$ is the Cartier closure of $M_X$.
\end{proof}

\begin{proposition}\label{prop:ordinary-corrected-adjunction}
Let $(X,B,h)$ be an int-amplified dynamical triple, let $C$ be an lc center of $(X,B)$ with $h^{-1}(C)=C$, and let $h_C\colon C^\nu\to C^\nu$ be the induced endomorphism of the normalization $\nu\colon C^\nu\to C$.
Let $\mathbf B$ be the discriminant b-divisor defined in \cref{def:b-divisors}, and put $B_{C}=\mathbf B_{C^\nu}$.
Let $\mathbf c=(c_{j,N})$ be a coefficient family.
Then $(C^\nu,B_C,h_C)$ is an int-amplified dynamical triple, and for all sufficiently large $N$, if $(C^\nu,\mathcal C_{h_C,N}(B_C;\mathbf c))$ is lc, then
\begin{equation}\label{eq:exact-center-nlc}
 \Nlc(X,\mathcal C_{h,N}(B;\mathbf c))\cap C=\varnothing.
\end{equation}
\end{proposition}

\begin{proof}
By \cref{lem:basic-dynamical-properties}, $(X,B)$ is lc.
Choose a $\Q$-Cartier divisor $D_{C^\nu}$ representing
the restricted $\Q$-line bundle $(K_X+B)|_{C^\nu}$.
Define the moduli b-divisor $\mathbf M$ by
\[
 \mathbf M_{Z'}
 :=p^*D_{C^\nu}-K_{Z'}-\mathbf B_{Z'}
\]
on every normal proper birational model
$p\colon Z'\to C^\nu$, using compatible canonical divisors.
By \cite{FHAdjunction}*{Theorem 1.2 and Remark 2.3},
$B_C$ is effective and $\mathbf M$ is b-nef and
$\Q$-b-Cartier.

The endomorphism $h_C$ is int-amplified by
\cref{prop:higher-center-restriction}.
Moreover, \cref{prop:log-lc-center-avoidance} shows that
$R_{h,B}|_{C^\nu}$ is defined and effective.
The finite-adjunction formula
\cite{CZ}*{Theorem 5.6} gives
\[
 h_C^*\mathbf M_{C^\nu}\sim_\Q\mathbf M_{C^\nu}
\]
and
\begin{equation}\label{eq:ordinary-center-ramification}
 R_{h_C}+B_C-h_C^*B_C=R_{h,B}|_{C^\nu}\geq0.
\end{equation}
By \cref{prop:invariant-nef-bdivisor}, $\mathbf M$
is the Cartier closure of the $\Q$-Cartier divisor
$\mathbf M_{C^\nu}$, and $\mathbf M_{C^\nu}\sim_\Q0$.
In particular,
\[
 K_{C^\nu}+B_C
 =D_{C^\nu}-\mathbf M_{C^\nu}
\]
is $\Q$-Cartier.
Together with $B_C\geq0$ and
\eqref{eq:ordinary-center-ramification}, this shows
that $(C^\nu,B_C,h_C)$ is an int-amplified
dynamical triple.
Since $\mathbf M$ descends to $C^\nu$, we have
\begin{equation}\label{eq:adjunction-discriminant-crepant}
 K_{Z'}+\mathbf B_{Z'}
 =p^*(K_{C^\nu}+B_C)
\end{equation}
on every normal proper birational model
$p\colon Z'\to C^\nu$.

For all sufficiently large $N$, put $D_N:=\mathcal C_{h,N}(B;\mathbf c)-B$.
Since $c_{j,N}\to0$, the divisor $D_N$ is effective for large enough $N$.
By \cref{lem:basic-dynamical-properties,prop:log-lc-center-avoidance},
it is $\Q$-Cartier and its support does not contain $C$.
Iterating \cref{prop:higher-center-restriction} and using
\eqref{eq:ordinary-center-ramification}, we obtain
\[
 (P_h^jR_{h,B})|_{C^\nu}
 =P_{h_C}^jR_{h_C,B_C}
 \qquad(j\geq1).
\]
Summing with the coefficients defining the corrections gives
\begin{equation}\label{eq:exact-center-correction}
 B_C+D_N|_{C^\nu}
 =\mathcal C_{h_C,N}(B_C;\mathbf c).
\end{equation}

Since $C\nsubseteq\Supp D_N$, the pairs $(X,B+D_N)$
and $(X,B)$ agree near the generic point of $C$.
Thus $C$ remains an lc center of $(X,B+D_N)$.
Let $\mathbf B^{(N)}$ be the discriminant b-divisor of $(X,B+D_N)$ along $C^\nu$.
The defining thresholds therefore give
\[
 \mathbf B_{Z'}^{(N)}
 =\mathbf B_{Z'}+p^*(D_N|_{C^\nu})
\]
on every normal proper birational model
$p\colon Z'\to C^\nu$.
Combining \eqref{eq:adjunction-discriminant-crepant}
and \eqref{eq:exact-center-correction}, we obtain
\[
 K_{Z'}+\mathbf B_{Z'}^{(N)}
 =p^*\bigl(K_{C^\nu}
          +\mathcal C_{h_C,N}(B_C;\mathbf c)\bigr).
\]

If $(C^\nu,\mathcal C_{h_C,N}(B_C;\mathbf c))$ is lc,
every coefficient of $\mathbf B_{Z'}^{(N)}$ is at most one.
Applying \cite{FHAdjunction}*{Theorem 1.2(iv)} to the effective $\Q$-divisor $B+D_N=\mathcal C_{h,N}(B;\mathbf c)$ yields
\[
 \Nlc(X,\mathcal C_{h,N}(B;\mathbf c))\cap C=\varnothing,
\]
as required.
\end{proof}

\subsection{Log canonical corrections and the main theorem}\label{sec:induction}

For the induction, a single coefficient family $\mathbf c=(c_{j,N})$ must be used both on the birational model and on the normalizations of its lc centers.
We therefore formulate the next theorem for all coefficient families satisfying finitely many interpolation conditions, and for further iterates of the endomorphism.

\begin{theorem}\label{thm:all-dimensional-lc-correction}
Let $(X,B,h)$ be an int-amplified dynamical triple.
There are an integer $r_0>0$ and a finite set $\Lambda\subset\mathbb Z_{>1}$ with the following property.
For every integer $a>0$ and every coefficient family $\mathbf c=(c_{j,N})$ satisfying the interpolation conditions for $\Lambda^a$, the pair
\[
 (X,\mathcal C_{h^{r_0a},N}(B;\mathbf c))
\]
is lc for all sufficiently large $N$.
The bound may depend on $a$ and $\mathbf c$.
If $h$ is $q$-polarized, one may take $\Lambda=\{q^{r_0}\}$.
\end{theorem}

\begin{proof}
We prove the full assertion, including the polarized case, by induction on $n:=\dim X$.
If $n=0$, take $r_0=1$ and $\Lambda=\varnothing$, or $\Lambda=\{q\}$ in the polarized case.

By \cref{thm:all-dimensional-lc-model}, choose an absorbing model
\[
 \pi\colon(W,B_W,g)\longrightarrow(X,B,h^r)
\]
with interpolation conditions for $\Lambda_0$.
After a further iterate, with the corresponding powers of the recorded indices, assume that every lc center of $(W,B_W)$ is totally invariant under $g$.
Let $C_1,\ldots,C_s$ be its maximal lc centers.
Their adjoints $(C_i^\nu,B_i,g_i)$ are int-amplified dynamical triples of smaller dimension by \cref{prop:ordinary-corrected-adjunction}.
Apply induction to obtain exponents $r_i>0$ and sets $\Lambda_i$ of interpolation parameters.
Choose a common multiple $t$ of the $r_i$, taking $t=1$ if $s=0$, and set
\begin{equation}\label{eq:induction-indices}
 r_0:=rt,\qquad
 \Lambda:=\Lambda_0^t\cup\bigcup_{i=1}^s\Lambda_i^{\,t/r_i}.
\end{equation}
If $h$ is $q$-polarized, \cref{rem:polarized-models} and induction give
\[
 \Lambda_0\subseteq\{q^r\},\qquad
 \Lambda_i=\{q^{rr_i}\}.
\]
Thus $\Lambda\subseteq\{q^{r_0}\}$, and we may enlarge it to this singleton.

Fix $a\geq1$ and a coefficient family $\mathbf c$ satisfying the interpolation conditions for $\Lambda^a$.
For each $i$,
\[
 \Lambda_i^{(t/r_i)a}\subseteq\Lambda^a,
 \qquad
 (g_i^{r_i})^{(t/r_i)a}=g_i^{ta}.
\]
Induction therefore gives lc of
$(C_i^\nu,\mathcal C_{g_i^{ta},N}(B_i;\mathbf c))$ for all sufficiently large $N$.
Using the same coefficients, \cref{prop:ordinary-corrected-adjunction} gives
\[
 \Nlc(W,\mathcal C_{g^{ta},N}(B_W;\mathbf c))
 \cap\bigcup_i C_i=\varnothing.
\]
On the other hand, \cref{lem:uniform-nlc-confinement}, absorption, and invariance of the current under iteration give
\[
 \Nlc(W,\mathcal C_{g^{ta},N}(B_W;\mathbf c))
 \subseteq\Nklt(W,B_W;T_{g^{ta},B_W})
 =\Nklt(W,B_W)=\bigcup_i C_i.
\]
Since there are finitely many centers, both statements hold for all sufficiently large $N$, so $(W,\mathcal C_{g^{ta},N}(B_W;\mathbf c))$ is lc.
If $s=0$, confinement alone gives the conclusion.

Finally, $\Lambda_0^{ta}\subseteq\Lambda^a$ and $rta=r_0a$, so \eqref{eq:correction-crepancy} gives
\[
 K_W+\mathcal C_{g^{ta},N}(B_W;\mathbf c)
 =\pi^*\bigl(K_X+\mathcal C_{h^{r_0a},N}(B;\mathbf c)\bigr).
\]
Log canonicity descends to $X$.
The choices of $r_0,\Lambda$ precede those of $a,\mathbf c$, as required.
\end{proof}

\begin{corollary}\label{cor:polarized-one-term-correction}
Let $(X,B,h)$ be a $q$-polarized dynamical triple.
After replacing $h$ by a suitable iterate and $q$ by the corresponding
power, retain the notation $h,q$ and put
\[
 \Theta_N:=B+\sum_{j=1}^N P_h^jR_{h,B}
       +\frac{q^{1-N}}{q-1}P_hR_{h,B}.
\]
Then $(X,\Theta_N)$ is log canonical for all sufficiently large $N$.
Furthermore, if
\begin{equation}\label{eq:polarized-log-eigenclass}
 h^*(K_X+B)\equiv q(K_X+B),
\end{equation}
then $(X,\Theta_N)$ is a log Calabi--Yau pair for all sufficiently large $N$.
\end{corollary}

\begin{proof}
Choose the iterate as in \cref{thm:all-dimensional-lc-correction}.
With the retained notation, we may take $\Lambda=\{q\}$.
Define a coefficient family $\mathbf c=(c_{1,N})$ by $c_{1,N}:=q^{1-N}/(q-1)$.
Then
\[
 c_{1,N}q^{-1}=\frac{q^{-N}}{q-1},
 \qquad c_{1,N}\longrightarrow0,
\]
so $\mathbf c=(c_{1,N})$ satisfies the interpolation conditions for $\Lambda$.
Since $\Theta_N=\mathcal C_{h,N}(B;\mathbf c)$,
the pair $(X,\Theta_N)$ is lc for all sufficiently large $N$
by \cref{thm:all-dimensional-lc-correction}.
Furthermore, \eqref{eq:numerical-telescope} and
\eqref{eq:polarized-log-eigenclass} give
\[
K_X+\Theta_N \equiv 0,
\]
and \cite{Gongyo}*{Theorem~1.2} yields
$K_X+\Theta_N\sim_\Q0$.
\end{proof}

\begin{proof}[Proof of \cref{thm:main}]
In dimension zero, take $\Theta=0$.
Assume $\dim X>0$ and choose $r_0$ and $\Lambda$ from \cref{thm:all-dimensional-lc-correction}.
Apply \cref{lem:numerical-correction} to $(X,B,h^{r_0})$ with $\Lambda$.
It gives a coefficient family $\mathbf c$ satisfying the interpolation conditions for $\Lambda$.
For all sufficiently large $N$, put $\Theta=\mathcal C_{h^{r_0},N}(B;\mathbf c)$.
Then $\Theta\geq B$, the pair $(X,\Theta)$ is lc, and $K_X+\Theta\equiv0$.
By \cite{Gongyo}*{Theorem 1.2}, one has $K_X+\Theta\sim_\Q0$.
\end{proof}

\begin{proof}[Proof of \cref{cor:all-dimensional-qg-cy}]
If $X$ is $\Q$-Gorenstein, then the assertion follows from \cref{thm:main} with $B=0$.
Next, we assume that $X$ is of klt type.
By \cite{CZ}*{Theorem A1.1}, there exist a projective small birational morphism
$\pi\colon Y\to X$ with $Y$ $\Q$-factorial, a positive integer $s$, and a surjective endomorphism
$g\colon Y\to Y$ such that $\pi\circ g=h^s\circ\pi$.
Since $Y$ is of Calabi--Yau type by \cref{thm:main}, so is $X$.
\end{proof}

\section{Normalized ramification divisors of polarized endomorphisms}
\label{sec:normalized-ramification}

In this section, we prove the following result.

\begin{theorem}\label{thm:nr-lc}
Let $(X,B,h)$ be a $q$-polarized dynamical triple. 
There is an integer $r_0\geq1$ such that
\[
 \left(X,B+\frac{R_{h^{mr_0},B}}{q^{mr_0}-1}\right)
 \quad\text{is lc for every integer }m\geq1.
\]
\end{theorem}

Throughout this section, $(X,B,h)$ is a $q$-polarized dynamical triple. 
Write
\[
 A=A_{X,B},\qquad s=s_{h,B},\qquad
 R_r=R_{h^r,B},\qquad
 Z=\Nklt(X,B;T_{h,B}),\qquad \Phi=q^{-1}h_*.
\]
By \cref{lem:basic-dynamical-properties}, $(X,B)$ is lc.
By \cref{thm:invariant-nklt}, after replacing $h$ by an
iterate and $q$ by the corresponding power, we may assume
that every lc center of $(X,B)$ and every irreducible
component of $Z$ is totally invariant.

\subsection{A discrepancy estimate outside the non-klt locus}

Put
\[
 \mathcal V_Z
 :=\{v\in\operatorname{Val}_X\mid v(\mathcal I_Z)=0\}
 =\{0\}\cup
   \{v\in\operatorname{Val}_X\mid c_X(v)\nsubseteq Z\}.
\]
Since $Z$ is totally invariant, $\mathcal V_Z$ is $\Phi$-invariant.
Since $\Nklt(X,B)\subseteq Z$, every nontrivial finite-discrepancy valuation in $\mathcal V_Z$ has $A(v)>0$.

\begin{lemma}\label{lem:nr-periodic-volume}
Let $C\subsetneq X$ be an irreducible closed subvariety such that $(X,B)$ is klt at the generic point $\eta_C$ of $C$, and suppose that $h^a(C)=C$ for some integer $a \geq1$.
Put $R:=\mathcal O_{X,\eta_C}$.
For every valuation $v$ centered at $C$ with $A(v)<\infty$,
\[
 0<\vol_R(v)\leq\vol_R(\Phi^a(v))<\infty.
\]
If equality holds, then $h^{-a}(C)=C$, and $v/q^a$ is
the only extension of $\Phi^a(v)$ through $h^a$.
In this case,
\[
 (A-s)(\Phi^a(v))=q^{-a}(A-s)(v).
\]
\end{lemma}

\begin{proof}
Put $g:=h^a$, $Q:=q^a$, $c:=\dim R$, and
$w:=\Phi^a(v)=Q^{-1}g_*v$.
Both $v$ and $w$ are centered at $C$ and have finite discrepancy
by \eqref{eq:finite-divisorial-discrepancy}.
The argument in the proof of \cref{lem:local-klt-tools} gives
\[
 0<\vol_R(v),\vol_R(w)<\infty,
\]
as in \eqref{eq:singular-volume-bounds}.

Let $D$ run through the irreducible components of $g^{-1}(C)$,
and put
\[
 R_D:=\mathcal O_{X,\eta_D},\qquad
 \mathfrak m_D:=\mathfrak m_{R_D},\qquad
 d_D:=\deg(g|_D).
\]
Each $D$ dominates $C$ by going down over the normal target,
and $\dim R_D=c$.
The finite $R$-algebra $(g_*\mathcal O_X)_{\eta_C}$ has rank
$\deg g$.
The multiplicity calculation in the proof of
\eqref{eq:nonflat-filtration} therefore gives
\[
 (\deg g)e_R(I)
 =\sum_D d_D e_{R_D}(g^\#(I)R_D)
\]
for every $\mathfrak m_R$-primary ideal $I$.

Set
\[
 \mathfrak b_{D,t}:=g^\#(\mathfrak a_t(w))R_D.
\]
The Izumi estimate for $w$ gives $\epsilon>0$ such that
\[
 \mathfrak a_t(w)\subseteq
 \mathfrak m_R^{\lceil\epsilon t\rceil},
 \qquad
 \mathfrak b_{D,t}\subseteq
 \mathfrak m_D^{\lceil\epsilon t\rceil}.
\]
Thus each $\mathfrak b_{D,\bullet}$ is a linearly bounded
$\mathfrak m_D$-filtration with positive multiplicity.
For $D=C$, we also have
\[
 \mathfrak b_{C,t}\subseteq\mathfrak a_{Qt}(v).
\]
Applying the multiplicity formula to $I=\mathfrak a_t(w)$,
dividing by $t^c$, and taking limits yields
\[
 \begin{aligned}
 (\deg g)\vol_R(w)
 &=\sum_D d_D e_{R_D}(\mathfrak b_{D,\bullet})\\
 &\geq d_CQ^c\vol_R(v)
  =(\deg g)\vol_R(v),
 \end{aligned}
\]
where the last equality follows from polarization.
This proves the volume inequality.

Suppose that $\vol_R(w)=\vol_R(v)$.
Since every summand is positive, there can be no component
$D\ne C$, so $g^{-1}(C)=C$.
Thus $g^\#\colon R\to R$ is finite, and every extension of
$w$ through $g$ is centered at $C$.
The equality argument in the proof of
\cref{lem:nonflat-volume-equality}, applied to this local map
with normalization factor $Q$, shows that its only extension
is $v/Q$.

Finally, applying \eqref{eq:nonflat-stationarity} to $g$,
and using $s_{g,B}=s$ and homogeneity, gives
\[
 (A-s)(w)=(A-s)(v/Q)=Q^{-1}(A-s)(v),
\]
as required.
\end{proof}

\begin{proposition}\label{prop:nr-discrepancy-gap}
Assume $\dim X>0$.
There are constants $C\geq1$ and $1<\rho<q$ such that
\begin{equation}\label{eq:nr-exponential-gap}
 A((h^r)_*v)\leq C\rho^rA(v)
 \qquad(r\geq0,\ v\in\mathcal V_Z,\ 0<A(v)<\infty).
\end{equation}
Consequently, for $U=X\setminus Z$,
\begin{equation}\label{eq:nr-threshold-gap}
 \lct_{(U,B|_U)}
       \left(\frac{R_r|_U}{q^r-1}\right)
 \geq\frac{q^r-1}{C\rho^r-1}\longrightarrow+\infty.
\end{equation}
In particular, $(X,B+R_r/(q^r-1))$ is klt on $U$
for all sufficiently large $r$.
\end{proposition}

\begin{proof}
By \eqref{eq:finite-divisorial-discrepancy},
\begin{equation}\label{eq:nr-discrepancy-identity}
 A((h^r)_*v)=A(v)+v(R_r).
\end{equation}
For a very ample polarization $H$, the projection formula gives
\[
 R_r\cdot H^{\dim X-1}
 =(q^r-1)\bigl(-(K_X+B)\cdot H^{\dim X-1}\bigr).
\]
The divisor $R_r$ contains no lc center by
\cref{prop:log-lc-center-avoidance}.
Hence \eqref{eq:lc-log-degree} gives $M\geq1$ such that
\begin{equation}\label{eq:nr-coarse-bound}
 A(\Phi^r(v))=q^{-r}\bigl(A(v)+v(R_r)\bigr)
 \leq MA(v)\qquad(r\geq0).
\end{equation}

We claim that there is $k\geq1$ such that
\begin{equation}\label{eq:nr-halving}
 A(\Phi^k(v))\leq\tfrac12 A(v)
 \qquad(v\in\mathcal V_Z,\ A(v)<\infty).
\end{equation}
Otherwise, by homogeneity, for every $N\geq1$ there is
$v_N\in\mathcal V_Z$ such that
\[
 A(v_N)=1,\qquad A(\Phi^N(v_N))>\tfrac12.
\]
Equation \eqref{eq:nr-coarse-bound} gives
\begin{equation}\label{eq:nr-orbit-bounds}
 (2M)^{-1}\leq A(\Phi^j(v_N))\leq M
 \qquad(0\leq j\leq N).
\end{equation}
Applying the compactness argument in the proof of \cref{lem:local-klt-tools} on a finite affine cover of
$X\setminus Z$, we see that
\[
 \mathscr K_M:=\{v\in\mathcal V_Z\mid A(v)\leq M\}
\]
is compact Hausdorff.
By \eqref{eq:nr-coarse-bound} and continuity of $\Phi$,
\[
 K:=\overline{\{\Phi^j(v_N)\mid N\geq1,\ j\geq0\}}
 \subseteq\mathscr K_M
\]
is compact and satisfies $\Phi(K)\subseteq K$.
By the argument in the proof of \cref{prop:invariant-average}, the probability measures
\[
 \mu_N:=\frac1N\sum_{j=0}^{N-1}\delta_{\Phi^j(v_N)}
\]
have a weak-* convergent subnet whose limit $\mu$ is a $\Phi$-invariant regular Borel probability measure on $K$.
Moreover, \eqref{eq:nr-discrepancy-identity} and
\eqref{eq:nr-orbit-bounds} give
\[
 \begin{aligned}
 \int v(R_1)\,d\mu_N(v)
 &=\frac{q-1}{N}\sum_{j=0}^{N-1}A(\Phi^j(v_N))
   +\frac qN\bigl(A(\Phi^N(v_N))-1\bigr)\\
 &\geq\frac{q-1}{2M}-\frac qN.
 \end{aligned}
\]
Since $v\mapsto v(R_1)$ is continuous,
\[
 \int v(R_1)\,d\mu(v)\geq\frac{q-1}{2M}>0,
\]
so $\mu\ne\delta_0$.

For an irreducible closed subvariety $C\subsetneq X$, put
\[
 U_C:=\{v\in\Supp\mu\mid v(\mathcal I_C)>0\}
     =\{v\in\Supp\mu\mid c_X(v)\subseteq C\}.
\]
These sets are open in $\Supp\mu$.
Since $\mu\ne\delta_0$, choose $C$ of minimal dimension among
the centers of nontrivial valuations in $\Supp\mu$.
Then $U_C$ is nonempty, so $\mu(U_C)>0$, and minimality gives
\[
 c_X(v)=C\qquad(v\in U_C).
\]

By Poincar\'e recurrence, there is $a\geq1$ such that
$\mu(U_C\cap\Phi^{-a}(U_C))>0$.
For any $v\in U_C\cap\Phi^{-a}(U_C)$, we have
\[
 C=c_X(\Phi^a(v))=h^a(c_X(v))=h^a(C).
\]
Since $\Phi^a(U_C)\subseteq U_C$ and $\mu$ is
$\Phi^a$-invariant, the restriction $\nu:=\mu|_{U_C}$
is a nonzero finite $\Phi^a$-invariant measure.
As $U_C\subseteq K\subseteq\mathcal V_Z$, we have
$C\nsubseteq Z$, and hence $(X,B)$ is klt at $\eta_C$.

Put $V(v):=\vol_{\mathcal O_{X,\eta_C}}(v)$ for $v\in U_C$.
As in the proof of \cref{prop:invariant-average}, $V$ is measurable.
By \cref{lem:nr-periodic-volume},
\[
 0<V(v)\leq V(\Phi^a(v))<\infty
 \qquad(v\in U_C).
\]
Invariance of $\nu$, applied to $V/(1+V)$, forces equality
$\nu$-almost everywhere.
The equality case of \cref{lem:nr-periodic-volume} therefore gives
\[
 (A-s)\circ\Phi^a=q^{-a}(A-s)
 \qquad\nu\text{-almost everywhere}.
\]
The functions $A$ and $s(v)=\lim_N v(S_{h,B,N})$ are Borel
measurable.
Since $C\nsubseteq Z$, \cref{prop:lc-valuation-center} gives
\[
 0<A(v)-s(v)\leq A(v)\leq M
 \qquad(v\in U_C).
\]
As $\nu$ is nonzero, finite, and $\Phi^a$-invariant,
\[
 0<\int(A-s)\,d\nu
   =\int(A-s)\circ\Phi^a\,d\nu
   =q^{-a}\int(A-s)\,d\nu<\infty,
\]
a contradiction. This proves \eqref{eq:nr-halving}.

Write $r=mk+j$ with $0\leq j<k$.
By \eqref{eq:nr-coarse-bound} and \eqref{eq:nr-halving},
\[
 A(\Phi^r(v))\leq M2^{-m}A(v)
 \leq 2M\,2^{-r/k}A(v).
\]
Choose $\max\{1,q2^{-1/k}\}<\rho<q$ and put $C:=2M$.
Then \eqref{eq:nr-exponential-gap} follows from
$A((h^r)_*v)=q^rA(\Phi^r(v))$.
For every divisorial valuation over $U$,
\[
 v(R_r)\leq(C\rho^r-1)A(v).
\]
The valuative threshold formula gives \eqref{eq:nr-threshold-gap}
for $r\geq1$, with the threshold of the zero divisor understood
as $+\infty$.
\end{proof}

\subsection{Birational compatibility, adjunction, and induction}

\begin{lemma}\label{lem:nr-compatibility}
\begin{enumerate}
\item\label{it:nr-model}
There exists an absorbing model,
\[
 \pi\colon(W,B_W,g)\longrightarrow(X,B,h^a),
\]
such that, for every $r\geq1$,
\begin{equation}\label{eq:nr-model-crepancy}
 K_W+B_W+\frac{R_{g^r,B_W}}{q^{ar}-1}
 =\pi^*\left(K_X+B+\frac{R_{h^{ar},B}}{q^{ar}-1}\right).
\end{equation}
\item\label{it:nr-adjunction}
Let $C$ be a totally invariant lc center of $(X,B)$, with adjoint triple $(C^\nu,B_C,h_C)$ from \cref{prop:ordinary-corrected-adjunction}.
Then
\begin{equation}\label{eq:nr-adjunction-identity}
 \left.\frac{R_{h^r,B}}{q^r-1}\right|_{C^\nu}
 =\frac{R_{h_C^r,B_C}}{q^r-1}.
\end{equation}
If $(C^\nu,B_C+R_{h_C^r,B_C}/(q^r-1))$ is lc, then
$(X,B+R_{h^r,B}/(q^r-1))$ is lc near $C$.
\end{enumerate}
\end{lemma}

\begin{proof}
By \cref{thm:all-dimensional-lc-model}, \eqref{eq:absorption-step-transform}, and
\cref{rem:polarized-models}, there exists an absorbing model satisfying
\[
 F:=K_W+B_W-\pi^*(K_X+B),\qquad g^*F=q^aF.
\]
Equation \eqref{eq:model-ramification}, applied to the iterates,
gives
\[
 R_{g^r,B_W}
 =\pi^*R_{h^{ar},B}+F-(g^r)^*F
 =\pi^*R_{h^{ar},B}-(q^{ar}-1)F.
\]
This proves \eqref{eq:nr-model-crepancy}.

Iterating \eqref{eq:ordinary-center-ramification} gives \eqref{eq:nr-adjunction-identity}.
The final assertion follows by the argument in the proof of
\cref{prop:ordinary-corrected-adjunction}, with
$D_N$ replaced by $R_{h^r,B}/(q^r-1)$.
\end{proof}

\begin{proof}[Proof of \cref{thm:nr-lc}]
By \eqref{eq:nr-discrepancy-identity},
\begin{equation}\label{eq:nr-lc-criterion}
 \left(X,B+\frac{R_{h^r,B}}{q^r-1}\right)\text{ is lc}
 \quad\Longleftrightarrow\quad
 A((h^r)_*v)\leq q^rA(v)
 \quad\text{for every divisorial }v.
\end{equation}
Since finite pushforward preserves divisorial valuations up to
scaling, iterating this inequality shows that lc for one exponent
implies lc for all its positive multiples;
cf.\ \cite{LuoMeng26}*{Proof of Theorem~1.3}.

We follow the induction on dimension as in the proof of
\cref{thm:all-dimensional-lc-correction}.
Replace \cref{lem:uniform-nlc-confinement} by
\cref{prop:nr-discrepancy-gap}, and replace the crepancy and
inversion-of-adjunction statements
\eqref{eq:correction-crepancy} and
\cref{prop:ordinary-corrected-adjunction} by
\cref{lem:nr-compatibility}.
The induced maps on the normalizations of the lc centers
are polarized by \cref{rem:polarized-models}.
Taking a sufficiently large common multiple of the exponents
provided by induction proves the assertion.
\end{proof}

\begin{proof}[Proof of \cref{thm:intro-normalized-ramification}]
Taking $B=0$ in \cref{thm:nr-lc} proves
\cref{thm:intro-normalized-ramification}.
\end{proof}

\bibliographystyle{skalpha}
\bibliography{int-amplified-cy}
\end{document}